\documentclass[12pt,a4paper]{article} 
\usepackage{amsmath}
\usepackage{amsthm}
\usepackage{amsfonts}
\usepackage{amssymb}
\usepackage{bussproofs}
\author{Amir Akbar Tabatabai}
\theoremstyle{plain} 
\newtheorem{thm}{Theorem}[section]

\newtheorem{lem}[thm]{Lemma}

\theoremstyle{definition}
\newtheorem{dfn}[thm]{Definition}
\newtheorem{exam}[thm]{Example}
\newtheorem{rem}[thm]{Remark}
\newtheorem{nota}[thm]{Notation}

\def\PA{\mathrm{PA}}
\def\Pr{\mathrm{Pr}}
\def\Prf{\mathrm{Prf}}

\def\S4{\mathrm{S4}}
\def\Cons{\mathrm{Cons}}
\def\Rfn{\mathrm{Rfn}}

\begin{document}
\title{Provability Interpretation of Propositional and Modal Logics} 

\author{Amirhossein Akbar Tabatabai \footnote{The author is supported by the ERC Advanced Grant 339691 (FEALORA)}\\
Institute of Mathematics\\
Academy of Sciences of the Czech Republic\\
tabatabai@math.cas.cz}

\date{\today}

\maketitle

\begin{abstract}
In 1933 \cite{G}, G\"{o}del introduced a provability interpretation of the propositional intuitionistic logic to establish a formalization for the BHK interpretation. He used the modal system, \textbf{S4}, as a formalization of the intuitive concept of provability and then translated \textbf{IPC} to \textbf{S4}. His work suggested the problem to find a concrete provability interpretation of the modal logic \textbf{S4}. In this paper, we will try to answer this problem. In fact, we will generalize Solovay's provability interpretation of the modal logic \textbf{GL} to capture other modal logics such as \textbf{K4}, \textbf{KD4} and \textbf{S4}. Then we will use these results to find a formalization for the BHK interpretation and we will show that with different interpretations of the BHK interpretation, we can capture some of the propositional logics such as Intuitionistic logic, minimal logic and Visser-Ruitenburg's basic logic.\\
Moreover, we will show that there is no provability interpretation for any extension of \textbf{KD45} and also there is no BHK interpretation for the classical propositional logic.
\end{abstract}

\newpage
 
\tableofcontents
 
\newpage
 
\section{Introduction}
\subsection{BHK Interpretation}
In the intuitionistic tradition, mathematics is considered as a theory of mental constructions and hence, truth naturally means the existence of a proof. Thus, provability is the core stone of the whole intuitionistic paradigm. With this fact in mind, like any other logic, the intuitionistic logic would be a calculus to describe the behavior of truth, which in this case, is the concept of provability. In other words, intuitionistic logic is a meta-theory of the concept of provability. Let us explain the role of connectives in this logic. Again, like any other logic, a connective is an operation on the truth content of its inputs, which in the case of intuitionistic logic means the operations on the proofs. If we want an intuitive semantics for intuitionistic logic, we have to find out what the meaning of a connective is. The answer to this question is the well-known BHK interpretation. Its propositional part is the following:\\
\\
$\bullet$ a proof for $A \wedge B$ is a pair of a proof for $A$ and a proof for $B$.\\
$\bullet$ a proof for $A \vee B$ is a proof for $A$ or a proof for $B$.\\
$\bullet$ a proof for $A \rightarrow B$ is a construction which transforms any proof of $A$ to a proof for $B$.\\
$\bullet$ a proof for $\neg A$ is a construction which transforms any proof of $A$ to a proof for $\bot$.\\
$\bullet$ $\bot$ does not have any proof.\\

Clearly, what we proposed as the BHK interpretation is just an informal interpretation and we need to find its exact formalization if we want to use it as a mathematical tool. For instance, if we want to establish an argument which shows that Heyting's formalization of $\mathbf{IPC}$ is an adequate formalization of intuitionistic viewpoint, we have to prove the soundness and completeness of $\mathbf{IPC}$ with respect to the BHK interpretation and this obviously needs an exact formalizion. Now, to formalize the interpretation, we firstly need a formalization of the concept of proof. Based on the extensive works in proof theory that have been done so far, it seems quite possible to find an appropriate formalization of the proof and hence of the BHK interpretation. But, unfortunately, despite all the attempts that have been made, the BHK interpretation has not been formalized so far (for an extensive history of the problem see \cite{Art}). Why does this natural and simple interpretation resist to become formalized? To find an answer to this question, let us investigate one of the key properties of the interpretation. Think of a proposition $A \rightarrow B$. Its proof is a construction that transforms any proof of $A$ to a proof of $B$. It is clear that this construction would be a meta-proof and not just a proof, because it talks about proofs and therefore it should belong to the meta-language of $A$ and $B$. In other words, we could claim that the act of introducing an implication increases the layer of the meta-language which we are arguing in. Therefore, in BHK interpretation all levels of our meta-languages are involved and this is the reason why this interpretation is so complex to be formalized. Since we need to formalize the meaning of a proof, we have to extend our task to find a meaning of a proof at any level of the meta-languages. \\

There are two different approaches to implement this idea. In the first approach, we could be faithful to the intuitionistic paradigm and find an intuitionistically valid interpretation of the proofs. However, in the second approach we could change our viewpoint and construct a bridge to find an appropriate classical interpretation of the concept of a proof to formalize the BHK interpretation. The first approach is Heyting's approach and the second one is Kolmogorov's. At first glance, the first approach seems very natural to try but there is a huge problem there; a conceptual vicious circle which forces us to understand the semantics of the paradigm, the BHK interpretation, in terms of itself and it makes the whole process very complicated. We want to emphasize that this vicious circle does not mean that the first approach is philosophically invalid, but it just shows how complex it could be. (Think of classical logic and its semantics which is based on the classical meta-theory. This is an obvious vicious circle, but these kinds of vicious circles are the inherent properties of any paradigm in the philosophy of mathematics and we have to deal with them.) In this paper we follow the second approach and interpret all proofs as the classical proofs in different layers of meta-languages. But this is not an easy task to do and in the forthcoming part of the Introduction we will investigate the problems in this approach.\\ 

The last thing we want to mention here is that what we are going to formalize, is actually an implicit version of the BHK interpretation, instead of the original one. In the original interpretation we interpret all the connectives as operations on explicitly mentioned proofs. But we could somehow eliminate the \textit{proofs} from the interpretation and just talk about the \textit{provability} of a sentence. For instance, the disjunction case in the original BHK interpretation transforms to the following one: $A \vee B$ is provable if $A$ is provable or $B$ is provable. The problem here, is the case of implication which is not reducible to a simpler one. In order to solve this problem, we need a primitive connective to formalize the concept of provability. A role which would be played by the connective ``box" in modal logics and this is one of the most important contributions to the problem, which was made by Kurt G\"{o}del. Now, G\"{o}del's contribution.

\subsubsection*{G\"{o}del's Translation}
In 1933 \cite{G}, G\"{o}del introduced a provability interpretation of $\mathbf{IPC}$ that can be seen as an implicit version of the well-known BHK interpretation of the intuitionistic logic. By this interpretation he could justify the fact that Heyting's formalization of $\mathbf{IPC}$ is sound and complete for its intended semantics which is the BHK interpretation. Let us review some steps of his work.\\

1. Giving a proof interpretation: Before Giving any provability interpretation of $\mathbf{IPC}$, we should explain our intention of the concept of \textit{provability} and the properties that we want to have. As you expect, G\"{o}del began his work exactly from this point. He used the language of modal logics, in which the symbol ``$\Box$"  is interpreted as a \textit{provability} predicate. In the next step, he formalized the expected properties of this provability predicate by some axioms which have made the well-known modal system $\mathbf{S4}$. Notice that in contrast with using a concrete interpretation of provability, he used a theory for formalizing this concept ($\mathbf{S4}$). In fact, his system just characterizes the properties of our intuitive provability predicate by some formal system, and is totally silent about its real nature.\\
After this introduction, we are ready to give the definition of his interpretation. Consider the translation function 
$ b:\mathcal{L}\to \mathcal{L}_{\Box}$
as follows:\\
$ \mathcal{L}$
and
$ \mathcal{L}_{\Box}$
are the languages of 
$\mathbf{IPC}$
and
$\mathbf{S4}$
respectively.
\footnote{ In fact, our translation is different from the translation of the paper \cite{G}. The differences are the following:
$ p^b=p$, $\bot^b=\bot$, 
$ (A\to B)^b=\Box A^b\to \Box B^{b}$,
and
$ (\neg A)^b=\neg \Box A^b$.
While, both of these two translations basically do the same task, we use the first one, because it is more compatible with our intuition of intuitionistic semantics and it is adequate for the systems weaker than $\mathbf{S4}$.
}
\begin{description}
\item[$(i)$]
$ p^b=\Box p$
and
$ \bot^b=\Box \bot$
\item[$(ii)$]
$(A\wedge B)^b= A^b \wedge B^b$
\item[$(iii)$]
$(A\vee B)^b=A^b\vee B^b$
\item[$(iv)$]
$(A\to B)^b=\Box(A^b \to B^b)$
\item[$(v)$]
$(\neg A)^b=\Box (A^b \rightarrow \Box \bot)$
\end{description}
It is clear that $A^b$ is the implicit BHK interpretation of $A$. In fact, the definition of $b$ is the natural paraphrase of the original BHK interpretation in terms of provability instead of proofs.\\
It is time to investigate the soundness-completeness property of the interpretation.\\
 
2. Soundness and Completeness: Consider the following theorem:
\begin{thm} \label{t1-1} 
For any proposition
$ A\in\mathcal L$,
$\mathbf{IPC} \vdash A$
iff
$\mathbf{S4} \vdash A^b$.
\end{thm}
\begin{proof}
For the complete investigation of this theorem and some related results\footnote{While this theorem is the heart of G\"{o}del's work, he only stated it and left it without any proof. The soundness part is an easy consequence of induction on the length of the proof, but the completeness part was finally proved in 1947 by Tarski and McKinsey by the algebraic semantics for $\mathbf{S4}$.} see \cite{G}.
\end{proof}
We have the system $\mathbf{S4}$ which formalizes what we expect from a provability predicate and based on the mentioned soundness-completeness result we can reduce the problem of finding a formalization of the implicit BHK interpretation to the problem of finding a provability interpretation for $\mathbf{S4}$. Therefore, our task will be to find a concrete interpretation of this provability predicate (the connective box) in terms of the classical provability in classical theories. But, consider the fact that the problem of finding a provability interpretation for $\mathbf{S4}$ has its own importance itself, independent on its relation to the BHK interpretation.\\

The first attempt to find a concrete provability interpretation for $\mathbf{S4}$ was made by G\"{o}del himself. In a very negative way, he showed that the natural expected interpretation of the provability predicate is not sound for $\mathbf{S4}$. Let us explain his result in more detail:\\
The most natural choice to interpret the box operator is the provability predicate of a formal theory \footnote{The system $T$ is formal iff the set of its consequences is recursively enumerable.}. Let $T$ be a formal system; therefore, the meaning of 
$\Box A$
would be
$\Pr_{T}(A)$
such that
$\Pr_{T}(\cdot)$
is a provability predicate for $T$. (Notice that in this case we suppose our formal system $T$ to be sufficiently strong to be able to formalize some parts of  the meta-mathematics.) Consider the theorem $\Box \neg \Box \bot$
of $\mathbf{S4}$. Its interpretation is
$\Pr_{T}(\neg \Pr_{T}(\bot))$ and if it were true we would have
$T \vdash \neg \Pr_T(\bot)$
which contradicts G\"{o}del's second incompleteness theorem.\\
Therefore, we know that on the one hand, the natural way to formalize the concept of proof and provability in the BHK interpretation is to fix a formal system and interpret all the proofs as the proofs in that theory. And on the other hand, the logic $\mathbf{S4}$ is not sound with this natural interpretation. This is for the case of $\mathbf{S4}$. However, we could claim that the natural formalization of the BHK interpretation is not sound as well. For instance, if you try to interpret the sentence $A \wedge (A \rightarrow B) \rightarrow B$ in intuitionistic logic, you find out that it is more or less the same as the modal formula $\Box(\Box p \rightarrow p)$ and you will encounter the same problem in intuitionistic logic. In sum, we can say that the natural formalization of the BHK interpretation and also the natural interpretation of $\mathbf{S4}$ do not work. Based on these observations, we have intuition why finding a formalization of the BHK interpretation is complicated and hard to grasp. \\ 

There is a natural question to ask. If the theory $\mathbf{S4}$ is intuitively valid and we know that we can not interpret the box as a provability predicate in some formal system, then what could be a natural provability interpretation of $\mathbf{S4}$? Unfortunately, despite a lot of attempts which have been made so far, this question remains open. For instance, Kripke \cite{Kr} introduced a provability interpretation which is based on his Kripke models and just captures our provability intuition for formulas without nested modalities. Or in \cite{Bu}, Buss introduced the ``pure provability" which have the same problem with the nested modalities. Actually, the only successful attempt to find a provability interpretation, is Artemov's ``logic of proofs" which is based on the idea of introducing all explicit proofs, investigating the intended behavior of proofs in a theory (logic of proofs) and then interpreting the box as the existence of the proof. These explicitly mentioned proofs could empower us to avoid non-standard proofs which has the main role in G\"{o}del's second incompleteness theorem and all counter-intuitive theorems in meta-mathematics. In Section 9 we will come back to Artemov's logic of proofs and we will investigate its advantages and disadvantages.\\

As this long introduction shows, the main problem is to find a provability interpretation for the modal logic $\mathbf{S4}$ to formalize the BHK interpretation. In this paper, we will try to solve this problem and in the forthcoming part of the Introduction we will sketch the idea of our semantics and our key results.

\subsection{The Main Idea and the Main Results}
Why doesn't the mentioned natural proof interpretation work? The answer is the fact that this interpretation does not distinguish between languages and meta-languages. Let us illuminate this fact by an example. Suppose $p$ is an atom. What should be an intended interpretation of $p$? $p$ is an atomic sentence about the real world, it is just a description of the world and this description is in the first level. But how about $\square p$? The intended interpretation of this formula is the provability of $p$ in some theory. But, what is important here, is the level of the theory and the level of this sentence. Since $p$ is a fact about the real world, the theory in which $p$ is proved, should be a first level theory, i.e. a theory about the world. However, the sentence ($\square p$) is not about the real world; it is about the provability and hence it should be characterized as a sentence in the second level. Therefore, the intended meaning of this second level sentence is $\Pr_{T_0}(p)$. Let us ask about the interpretation of $\square \square p$. This is about the provability of the provability of $p$. The first box refers to a first level theory $T_0$. But the second box is about the provability of the provability, which has higher order, and it means the provability should be investigated in a second level theory, $T_1$. The important thing is the fact that there is no reason to assume that $T_1=T_0$. Actually, our experience in mathematical logic shows that it is genuinely important to distinguish the meta-theory and the object theory, and in some crucial cases the power of the meta-theory should be more than the theory itself. For instance, G\"{o}del's incompleteness theorems show that to answer a very basic meta-mathematical question about the system, i.e. its consistency, we need a more powerful meta-theory. Based on these investigations, the natural way to interpret boxes in a modal sentence is interpreting them in different theories with respect to the complexity of the occurrence of a box. To formalize this idea, we need two different ingredients. First, a model for the real world to interpret atoms as the facts about the world and second a hierarchy of theories which plays the role of the hierarchy of the meta-theories. Hence, the intended model would be $(M, \{T_n\}_{n=0}^{\infty})$ in which $M$ is a classical model and $T_n$ is a theory in the $n$-th level of the hierarchy. (We call these models, the provability models.) Moreover, we need a way of witnessing all boxes as the provability predicates of these theories in an appropriate way. This is the complex part of the formalization and we will talk about it in the next section. But for now, just think of the interpretation intuitively in the sense that any outer box should be interpreted as the provability predicate of a bigger theory. Therefore, our main result for modal logics is the following:
\begin{thm}\label{t1-2} 
\begin{itemize}
\item[$(i)$]
The logic $\mathbf{K4}$ is sound and complete with respect to the provability interpretation in all provability models.
\item[$(ii)$]
The logic $\mathbf{KD4}$ is sound and complete with respect to the provability interpretation in consistent provability models, i.e. $(M, \{T_n\}_{n=0}^{\infty})$ where for any $n$, $M$ thinks that $T_n$ is consistent and $T_{n+1} \vdash \Cons(T_n)$.
\item[$(iii)$]
The logic $\mathbf{S4}$ is sound and complete with respect to the provability interpretation in all reflexive provability models, i.e. $(M, \{T_n\}_{n=0}^{\infty})$ where for any $n$, $M$ thinks that $T_n$ is sound and $T_{n+1} \vdash \Rfn(T_n)$.
\item[$(iv)$]
The logic $\mathbf{GL}$ is sound and complete with respect to the provability interpretation in all constant provability models, i.e. $(M, \{T_n\}_{n=0}^{\infty})$ where for any $n$, $M$ thinks that $T_n = T_0$.
\item[$(v)$]
The logic $\mathbf{GLS}$ is sound and complete with respect to the provability interpretation in all sound constant provability models, i.e. $(M, \{T_n\}_{n=0}^{\infty})$ where for any $n$, $M$ thinks that $T_n$ is sound and $T_n = T_0$.
\item[$(vi)$]
The extensions of the logic $\mathbf{KD45}$ are not sound in any provability model.
\end{itemize}
\end{thm}
Here are some remarks about this main theorem. First of all, it shows that the use of a hierarchy of meta-theories instead of just one theory to witness the box operators, could define a brand new framework to capture different modal logics in terms of the provability interpretations. In fact, it shows that modal logics could be seen as the formal theories to describe the relation between the real world and the theories in the hierarchy of meta-theories which we use; in other words, they are theories for the whole discourse of the provability. Moreover, in the case of the logics $\mathbf{K4}$, $\mathbf{KD4}$ and $\mathbf{S4}$ it shows that they describe the relation of the model and meta-theories in a natural and expected way. For instance, in an informal reading of the axiom $\Box A \rightarrow A$ in $\mathbf{S4}$, we mean that our proofs are sound. And this is exactly one of the conditions we put on the models to capture the logic $\mathbf{S4}$. It is similar for all other axioms, logics and conditions in the aforementioned result.\\
Secondly, the result shows that if we restrict the whole hierarchy of meta-theories to just one theory, we could reconstruct Solovay's results for $\mathbf{GL}$ and $\mathbf{GLS}$. Therefore, it shows that our provability interpretation is a generalization of Solovay's interpretation and our main result is a generalization of Solovay's results.\\ 

If we combine this provability interpretations with G\"{o}del translation, we will have different BHK interpretations with respect to different powers of meta-theories. We have:
\begin{thm}\label{t1-3} 
\begin{itemize}
\item[$(i)$]
The logic $\mathbf{BPC}$ is sound and complete with respect to the BHK interpretation in all provability models.
\item[$(ii)$]
The logic $\mathbf{EBPC}$ is sound and complete with respect to the BHK interpretation in all consistent provability models, i.e. $(M, \{T_n\}_{n=0}^{\infty})$ where for any $n$, $M$ thinks that $T_n$ is consistent and $T_{n+1} \vdash \Cons(T_n)$.
\item[$(iii)$]
The logic $\mathbf{MPC}$ is sound and complete with respect to the weak BHK interpretation in all reflexive provability models, i.e. $(M, \{T_n\}_{n=0}^{\infty})$ where for any $n$, $M$ thinks that $T_n$ is sound and $T_{n+1} \vdash \Rfn(T_n)$.
\item[$(iv)$]
The logic $\mathbf{IPC}$ is sound and complete with respect to the BHK interpretation in all reflexive provability models, i.e. $(M, \{T_n\}_{n=0}^{\infty})$ where for any $n$, $M$ thinks that $T_n$ is sound and $T_{n+1} \vdash \Rfn(T_n)$.
\item[$(v)$]
The logic $\mathbf{FPL}$ is sound and complete with respect to the BHK interpretation in all constant provability models, i.e. $(M, \{T_n\}_{n=0}^{\infty})$ where for any $n$, $M$ thinks that $T_n=T_m$.
\item[$(vi)$]
The logic $\mathbf{CPC}$ does not admit any BHK interpretations.

\end{itemize}
\end{thm}
If you are not familiar with these propositional logics, we will define them in the Preliminaries section. But for now, just assume that the propositional logics $\mathbf{BPC}$, $\mathbf{EBPC}$, $\mathbf{IPC}$ and $\mathbf{FPL}$ are the propositional counterparts of the modal systems $\mathbf{K4}$, $\mathbf{KD4}$, $\mathbf{S4}$ and $\mathbf{GL}$, respectively. Moreover, by weak BHK interpretation, we informally mean the usual BHK interpretation without the consistency condition. This is the last condition in the BHK interpretation which assumes that there is no proof for $\bot$. And finally, by $\mathbf{MPC}$, we informally mean the system $\mathbf{IPC}$ without the Ex Falso rule. The rule which makes possible to prove anything from the contradiction.\\

Some remarks about this result are in place. First of all, it shows that there are different BHK interpretations instead of just one. This observation, somehow contradicts the folklore belief and it is surprising. The reason is that the BHK interpretation just defines the meaning of a connective in terms of the provability in different levels of meta-languages. But, it is silent about what kinds of commitments we impose on our meta-theories. 

Therefore, we can impose different philosophical conditions on the behavior of meta-theories to capture different propositional logics, all of them valid under the BHK interpretation. For instance, we can choose the minimal possible commitment which means that there is no non-trivial condition on the hierarchy of meta-theories. Then the BHK interpretation leads to the logic $\mathbf{BPC}$. On the other hand, if we suppose that our meta-theories are strong enough to prove the reflection principle for lower theories and all the theories are sound, then the BHK interpretation leads to the logic $\mathbf{IPC}$. This observation shows a key fact: There is a web of different intuitionistic logics according to the BHK interpretation; the logics $\mathbf{IPC}$ and $\mathbf{BPC} $ are just two examples of these intuitionistic logics and both of them are philosophically valid. In sum, we have to talk about \textit{intuitionistic logics} instead of \textit{the} intuitionistic logic.\\
Secondly, the result shows that this framework of the provability interpretation can capture different propositional logics and just like the case of modal logics, we are able to say that propositional logics are logics to describe the behavior of the real world and the hierarchy of meta-theories. This formalizes the intuitionist claim that intuitionistic mathematics is a way to talk and only talk about proofs.\\
Thirdly, it is possible to define different kinds of G\"{o}del's translation. Hence, it is possible to capture different propositional logics via these different translations. But it is important to consider that the translation we used in the above result is the valid translation to formalize the BHK interpretation and those different kinds of translations may not be rooted in the usual BHK interpretation. However, they are still provability interpretations and could be useful.
\section{Preliminaries}
In this section we will introduce some of the preliminaries that we need in the following sections. First of all, we will introduce the sequent calculi for the modal logics $\mathbf{K4}$, $\mathbf{KD4}$ and $\mathbf{S4}$. Then we will introduce some propositional logics such as $\mathbf{BPC}$, $\mathbf{MPC}$ and $\mathbf{IPC}$ as the propositional counterparts of some of the modal logics and finally we will state the Solovay's completeness results.
\subsection{Sequent Calculi for Modal Logics}
Consider the following set of rules:
\begin{flushleft}
	\textbf{Axioms:}
\end{flushleft}
\begin{center}
	\begin{tabular}{c c}
		\AxiomC{$  A \Rightarrow A$ }
		\DisplayProof 
			&
		\AxiomC{$ \bot \Rightarrow  $}
		\DisplayProof
	\end{tabular}

\end{center}
\begin{flushleft}
 		\textbf{Structural Rules:}
\end{flushleft}
\begin{center}
	\begin{tabular}{c}
		\begin{tabular}{c c}
		\AxiomC{$\Gamma  \Rightarrow \Delta$}
		\LeftLabel{\tiny{$ (wL) $}}
		\UnaryInfC{$\Gamma,  A  \Rightarrow \Delta$}
		\DisplayProof
			&
		\AxiomC{$\Gamma  \Rightarrow \Delta$}
		\LeftLabel{\tiny{$ ( wR) $}}
		\UnaryInfC{$\Gamma \Rightarrow  \Delta, A$}
		\DisplayProof
		\end{tabular}
			\\[3 ex]
			\begin{tabular}{c c}
		\AxiomC{$\Gamma, A, A \Rightarrow \Delta$}
		\LeftLabel{\tiny{$ (cL) $}}
		\UnaryInfC{$\Gamma,  A  \Rightarrow \Delta$}
		\DisplayProof
		    &
		\AxiomC{$\Gamma \Rightarrow \Delta, A, A$}
		\LeftLabel{\tiny{$ (cR) $}}
		\UnaryInfC{$\Gamma \Rightarrow \Delta, A$}
		\DisplayProof
		\end{tabular}
        \\[3 ex]
	    
	    \AxiomC{$\Gamma_0 \Rightarrow \Delta_0, A$}
	    \AxiomC{$\Gamma_1, A \Rightarrow \Delta_1$}
		\LeftLabel{\tiny{$ (cut) $}}
		\BinaryInfC{$\Gamma_0, \Gamma_1 \Rightarrow \Delta_0, \Delta_1$}
		\DisplayProof
	\end{tabular}
\end{center}		
\begin{flushleft}
  		\textbf{Propositional Rules:}
\end{flushleft}
\begin{center}
  	\begin{tabular}{c c}
  		\AxiomC{$\Gamma_0, A  \Rightarrow \Delta_0 $}
  		\AxiomC{$\Gamma_1, B  \Rightarrow \Delta_1$}
  		\LeftLabel{{\tiny $\vee L$}} 
  		\BinaryInfC{$ \Gamma_0, \Gamma_1, A \lor B \Rightarrow \Delta_0, \Delta_1 $}
  		\DisplayProof
	  		&
	   	\AxiomC{$ \Gamma \Rightarrow \Delta, A_i$}
   		\RightLabel{{\tiny $ (i=0, 1) $}}
   		\LeftLabel{{\tiny $\vee R$}} 
   		\UnaryInfC{$ \Gamma \Rightarrow \Delta, A_0 \lor A_1$}
   		\DisplayProof
	   		\\[3 ex]
   		\AxiomC{$ \Gamma, A_i \Rightarrow \Delta$}
   		\RightLabel{{\tiny $ (i=0, 1) $}} 
   		\LeftLabel{{\tiny $\wedge L$}}  		
   		\UnaryInfC{$ \Gamma, A_0 \land A_1 \Rightarrow \Delta, C $}
   		\DisplayProof
	   		&
   		\AxiomC{$\Gamma_0  \Rightarrow \Delta_0, A$}
   		\AxiomC{$\Gamma_1  \Rightarrow  \Delta_1, B$}
   		\LeftLabel{{\tiny $\wedge R$}} 
   		\BinaryInfC{$ \Gamma_0, \Gamma_1 \Rightarrow \Delta_0, \Delta_1, A \land B $}
   		\DisplayProof
   			\\[3 ex]
   		\AxiomC{$ \Gamma_0 \Rightarrow A, \Delta_0 $}
  		\AxiomC{$ \Gamma_1, B \Rightarrow \Delta_1, C $}
  		\LeftLabel{{\tiny $\rightarrow L$}} 
   		\BinaryInfC{$ \Gamma_0, \Gamma_1, A \rightarrow B \Rightarrow \Delta_0, \Delta_1, C$}
   		\DisplayProof
   			&
   		\AxiomC{$ \Gamma, A \Rightarrow B, \Delta $}
   		\LeftLabel{{\tiny $\rightarrow R$}} 
   		\UnaryInfC{$ \Gamma \Rightarrow \Delta, A \rightarrow B$}
   		\DisplayProof
   		\\[3 ex]
   		\AxiomC{$ \Gamma \Rightarrow \Delta, A $}
   		\LeftLabel{\tiny {$\neg L$}} 
   		\UnaryInfC{$ \Gamma, \neg A \Rightarrow \Delta$}
   		\DisplayProof
   			&
   		\AxiomC{$ \Gamma, A \Rightarrow \Delta $}
   		\LeftLabel{{\tiny $\neg R$}} 
   		\UnaryInfC{$ \Gamma \Rightarrow \Delta, \neg A$}
   		\DisplayProof
	\end{tabular}
\end{center}
\begin{flushleft}
 		\textbf{Modal Rules:}
\end{flushleft}
\begin{center}
  	\begin{tabular}{c c}
		\AxiomC{$ \Gamma, \Box \Gamma \Rightarrow A$}
		\LeftLabel{\tiny{$\Box_4 R$}}
		\UnaryInfC{$\Box \Gamma \Rightarrow \Box A$}
		\DisplayProof
		&
		\AxiomC{$ \Gamma, \Box \Gamma \Rightarrow $}
		\LeftLabel{\tiny{$\Box_D R$}}
		\UnaryInfC{$\Box \Gamma \Rightarrow $}
		\DisplayProof
\end{tabular}
\end{center}		
\begin{center}
  	\begin{tabular}{c c}
  	\AxiomC{$ \Box \Gamma \Rightarrow A$}
		\LeftLabel{\tiny{$\Box_S R$}}
		\UnaryInfC{$\Box \Gamma \Rightarrow \Box A$}
		\DisplayProof
		&
        \AxiomC{$ \Gamma, A \Rightarrow \Delta$}
		\LeftLabel{\tiny{$\Box L$}}
		\UnaryInfC{$\Gamma, \Box A \Rightarrow \Delta$}
		\DisplayProof
		\\[2 ex]
	\end{tabular}
\end{center}
The system $G(\mathbf{K4})$ is the system that consists of the axioms, structural rules, propositional rules and the modal rule $\Box_{4} R$. $G(\mathbf{KD4})$ is $G(\mathbf{K4})$ plus the rule $\Box_D R$ and finally, $G(\mathbf{S4})$ is the system $G(\mathbf{K4})$ when we replace the rule $\Box_{4} R$ by $\Box_{S} R$ and add the rule $\Box L$. All of these systems have the cut elimination property. (See \cite{Po}).\\

\subsection{Propositional Logics}
The next ingredient is the propositional counterparts of the usual modal logics. The intuitionistic logic $\mathbf{IPC}$ and the minimal logic $\mathbf{MPC}$ are the well-known logics in this area, but there are also some weaker systems which are very interesting in terms of the provability interpretation. For instance, we can mention the basic propositional logic $\mathbf{BPC}$ and the formal propositional logic $\mathbf{FPL}$ defined by A. Visser in \cite{Vi} or the extended basic propositional logic $\mathbf{EBPC}$ defined by M. Ardeshir and B. Hesaam in \cite{Ard}. To define these logics, consider the following set of rules:
\begin{flushleft}
   \textbf{Propositional Rules:}
   \end{flushleft}
\begin{center}
  	\begin{tabular}{c c}
		\AxiomC{$ A$}
	   	\AxiomC{$ B$}
	   	\LeftLabel{\tiny{$\wedge I$}}
	   	\BinaryInfC{$A \wedge B $}
	   	\DisplayProof
		&
		\AxiomC{$ A \wedge B $}
		\LeftLabel{\tiny{$\wedge E$}}
		\UnaryInfC{$ A$}
		\DisplayProof
			
		\AxiomC{$ A \wedge B $}
		\LeftLabel{\tiny{$\wedge E$}}
		\UnaryInfC{$ B$}
		\DisplayProof
	   		\\[4 ex]
	   		
        \AxiomC{$  A $}
        \LeftLabel{\tiny{$\vee I$}}
   		\UnaryInfC{$ A \vee B$}
   		\DisplayProof
   		
   		\AxiomC{$  B $}
   		\LeftLabel{\tiny{$\vee I$}}
   		\UnaryInfC{$ A \vee B$}
   		\DisplayProof
			&
   		\AxiomC{$A \lor B$}
        \AxiomC{[$A$]}
        \noLine
   		\UnaryInfC{$\mathcal{D}$}
        \noLine
        \UnaryInfC{$C$}
        
        \AxiomC{[$B$]}
        \noLine
   		\UnaryInfC{$\mathcal{D'}$}
        \noLine
        \UnaryInfC{$C$}
        \LeftLabel{\tiny{$\vee E$}}
        \TrinaryInfC{$C$}
        \DisplayProof
    		\\[4 ex]
    		
        \AxiomC{[$A$]}
   		\noLine
   		\UnaryInfC{$\mathcal{D}$}
   		\noLine
   		\UnaryInfC{$B$}
   		\LeftLabel{\tiny{$\rightarrow I$}}
   		\UnaryInfC{$ A \rightarrow B$}
   		\DisplayProof 
   		&   		
   		\AxiomC{$  \bot $}
    		\LeftLabel{\tiny{$\bot$}}
   		\UnaryInfC{$ A$}
   		\DisplayProof
   \end{tabular}
   \end{center}

\begin{flushleft}
   \textbf{Formalized Rules:}
   \end{flushleft}
   
\begin{center}
\begin{tabular}{c c}
   		   		
   		\AxiomC{$A \rightarrow B$}
    		\AxiomC{$A \rightarrow C$}
    		\LeftLabel{\tiny{$(\wedge I)_f$}}
    		\BinaryInfC{$A \rightarrow B \wedge C$}
    		\DisplayProof
    		&
    		\AxiomC{$A \rightarrow C$}
    		\AxiomC{$B \rightarrow C$}
    		\LeftLabel{\tiny{$(\vee E)_f$}}
    		\BinaryInfC{$A \vee B \rightarrow C$}
    		\DisplayProof
    		\\[3 ex]
    		
\end{tabular}
\end{center}
\begin{center}
      \begin{tabular}{c}
  
    		\AxiomC{$A \rightarrow B$}
    		\AxiomC{$B \rightarrow C$}
    		\LeftLabel{\tiny{$tr_f$}}
    		\BinaryInfC{$A \rightarrow C$}
    		\DisplayProof
   		
	\end{tabular}
\end{center}
Moreover, consider the following set of rules:\\

\begin{center}
  	\begin{tabular}{c c c}
		\AxiomC{$ A$}
		\AxiomC{$ \neg A$}
		\LeftLabel{\tiny{$C$}}
		\BinaryInfC{$\bot$}
		\DisplayProof
		&	
		\AxiomC{$A $}
		\AxiomC{$A \rightarrow B$}
		\LeftLabel{\tiny{$R$}}
		\BinaryInfC{$ B$}
		\DisplayProof
		\\[3 ex]
		
		\AxiomC{}
		\LeftLabel{\tiny{$D$}}
		\UnaryInfC{$A \vee \neg A$}
		\DisplayProof
		&
        \AxiomC{$(A \wedge (A \rightarrow B)) \rightarrow B$}
        \LeftLabel{\tiny{$L$}}
        \UnaryInfC{$A \rightarrow B$}
        \DisplayProof
	\end{tabular}
\end{center}
The logic $\mathbf{BPC}$ is defined as the system consists of the propositional rules and the formalized rules. Then logic $\mathbf{EBPC}$ defined as $\mathbf{BPC}+ C$, logic $\mathbf{FPL}$ is defined as $\mathbf{BPC} + L$, $\mathbf{IPC}$ is defined as $\mathbf{BPC} + R$, $\mathbf{MPC}$ is defined as $\mathbf{IPC}$ without $\bot$ rule, and finally $\mathbf{CPC}$ is defined as $\mathbf{IPC} + D$.
\begin{rem}
Consider the following rules:
\begin{center}
  	\begin{tabular}{c c c}
		
		\AxiomC{$\top \rightarrow \bot$}
		\LeftLabel{\tiny{$C'$}}
		\UnaryInfC{$\bot$}
		\DisplayProof
		&	
		\AxiomC{$\top \rightarrow A$}
		\LeftLabel{\tiny{$R'$}}
		\UnaryInfC{$ A$}
		\DisplayProof
		&
        \AxiomC{$(\top \rightarrow A) \rightarrow A$}
        \LeftLabel{\tiny{$L'$}}
        \UnaryInfC{$\top \rightarrow A$}
        \DisplayProof
		\\[3 ex]
	\end{tabular}
\end{center}

It is possible to define $\mathbf{EBPC}$ as $\mathbf{BPC} + D'$; $\mathbf{IPC}$ as $\mathbf{BPC} + R'$ and $\mathbf{FPL}$ as $\mathbf{BPC} + L'$. It is obvious that $D'$, $R'$ and $L'$ are special cases of $D$, $R$ and $L$, respectively. Therefore it remains to show that $D'$, $R'$ and $L'$ can simulate $D$, $R$ and $L$, respectively. The following proofs show that it is the case:
\begin{center}
  	\begin{tabular}{c c}
	
		\AxiomC{$ A$}
		\UnaryInfC{$\top \rightarrow A $}
		\AxiomC{$ A \rightarrow \bot$}
		\BinaryInfC{$\top \rightarrow \bot$}
		\LeftLabel{\tiny{$C'$}}
		\UnaryInfC{$\bot$}
		\DisplayProof
		\;\;\;\;\;
		&	
		\AxiomC{$ A$}
		\UnaryInfC{$\top \rightarrow A $}
		\AxiomC{$A \rightarrow B$}
		\BinaryInfC{$\top \rightarrow B$}
		\LeftLabel{\tiny{$R'$}}
		\UnaryInfC{$ B$}
		\DisplayProof
		\\[5 ex]
	\end{tabular}
\end{center}

\begin{center}
\begin{tabular}{c}
        \AxiomC{$ $}
        \doubleLine
        \UnaryInfC{$A \rightarrow \top$}
        
        \AxiomC{$[\top \rightarrow (A \rightarrow B)]^2$}
        
	    \AxiomC{$[A]^1$}
	    \doubleLine
	    \UnaryInfC{$\top \rightarrow A$}
	    
        \BinaryInfC{$\top \rightarrow (A \wedge (A \rightarrow B))$}
        \AxiomC{$[(A \wedge (A \rightarrow B)) \rightarrow B]^3$}
        \BinaryInfC{$\top \rightarrow B$}
        
        \BinaryInfC{$A \rightarrow B$}
        \LeftLabel{\tiny{$\rightarrow I_2$}}
        \UnaryInfC{$(\top \rightarrow (A \rightarrow B)) \rightarrow (A \rightarrow B)$}
        \UnaryInfC{$\top \rightarrow (A \rightarrow B)$}
        \LeftLabel{\tiny{$(*)$}}
        \doubleLine
        \UnaryInfC{$A \rightarrow B$}
        \LeftLabel{\tiny{$\rightarrow I_1$}}
        \UnaryInfC{$A \rightarrow ((A \rightarrow B))$}
        \doubleLine
        \UnaryInfC{$A \rightarrow (A \wedge (A \rightarrow B))$}
        
        \AxiomC{$[(A \wedge (A \rightarrow B)) \rightarrow B]^3$}
        \insertBetweenHyps{\hskip -100pt}
        \BinaryInfC{$(A \rightarrow B)$}
        \DisplayProof
\end{tabular}
\end{center}
Notice that the double lines mean simple sub-proofs that we do not mention and $(*)$ is the sub-proof which proves 
\[
A, (\top \rightarrow (A \rightarrow B)), ((A \wedge (A \rightarrow B)) \rightarrow B) \vdash A \rightarrow B
\] 
\end{rem}
\subsection{Solovay's Theorems}
In this subsection we will mention the Solovay's seminal arithmetical completeness theorems. (See \cite{So} and \cite{Bo}.) They will be needed to prove some of our completeness theorems in the next sections. Note that in the case of $\mathbf{GL}$ we will state the uniform version of the completeness theorem which will have a crucial role in our proofs.
\begin{dfn}
Assume that $I\Sigma_1 \subseteq T$ is a $\Sigma_1$-sound arithmetical theory. By an arithmetical substitution $\sigma$ we mean a function from the atomic formulas in the modal language to the set of arithmetical sentences. And if $A \in \mathcal{L}_{\Box}$ is a modal formula, by $A^{\sigma}$ we mean an arithmetical sentence resulted by substituting atoms by $\sigma$, and interpreting boxes as the provability predicate of $T$.
\end{dfn}
\begin{thm}\label{t0-0}
\begin{itemize}
\item[$(i)$](First Theorem)
If $\mathbf{GL} \vdash A$ then for all arithmetical substitutions $\sigma$, $I\Sigma_1 \vdash A^{\sigma}$. Moreover, there is an arithmetical substitution $*$ such that for all modal formulas $A$, if $T \vdash A^*$, then $\mathbf{GL} \vdash A$.
\item[$(ii)$](Second Theorem)
$\mathbf{GLS} \vdash A$ iff for all arithmetical substitutions $\sigma$, $\mathbb{N} \vDash A^{\sigma}$.
\end{itemize}
\end{thm}
\section{Provability models}
In this section we will introduce a provability model as a formalization of the intuitive combination of a model and a hierarchy of theories. Then, we will define the satisfaction relation between modal formulas and provability models. And as a conclusion, we will justify our notion of provability interpretation.
\subsection{Definitions and Examples}
Suppose that we have a modal formula $A$, and we want to interpret any box in the formula as a provability predicate. Note that when you have two boxes in $A$ such that one box is in the scope of the other box, our intuition forces us to accept that the outer box talks about the provability in the meta-theory while the inner box is just capturing the provability in the lower theories. Therefore, we can claim that the natural model for the provability interpretation of modal logics is a pair of one first order structure to interpret the atoms of the language, and a hierarchy of theories to play the role of a hierarchy of meta-theories. Moreover, we choose our structure and our theories as a model and theories for arithmetic, respectively, because in these theories we have a natural way of coding the language, the meta-language, the meta-meta-language and so on. Furthermore, we suppose that all of our theories include $I\Sigma_1$ to have enough power to formalize the basic meta-mathematics of the theories. And, for the same reason we assume $M \vDash I\Sigma_1$, because we want to have the true meta-mathematical properties obviously.
\begin{dfn}\label{t2-1} 
A provability model is a pair $(M, \{T_n\}_{n=0}^{\infty})$ where $M$ is a model of $I\Sigma_1$ and $\{T_n\}_{n=0}^{\infty}$ is a hierarchy of arithmetical r.e. theories such that for any $n$, $I\Sigma_1 \subseteq T_n \subseteq T_{n+1}$ provably in $I\Sigma_1$.
\end{dfn}
We define an expansion of a modal formula.
\begin{dfn}\label{t2-2} 
$E(A)$, the set of all expansions of $A$, is inductively defined as follows:
\begin{itemize}
\item[$\bullet$]
If $A$ is an atom, $E(A)=\{A\}$.
\item[$\bullet$]
If $A=B \circ C$, then $E(A)=\{D \circ E \mid D \in E(B) \; \text{and} \; E \in E(C)\}$ for $\circ \in \{\wedge, \vee, \rightarrow\}$.
\item[$\bullet$]
If $A=\neg B$, then $E(A)=\{\neg D \mid D \in E(B)\}$.
\item[$\bullet$]
If $A=\square B $, then $E(A)=\{\Box \bigvee_{i=1}^{k}D_i \mid \forall 1 \leq i \leq k, \; D_i \in E(B)\}$.
\end{itemize}
Moreover, if $\Gamma$ is a sequence of modal formulas, by a sequence of expansions of $\Gamma$, we mean a sequence such that for any formula in $\Gamma$, it has at least one of its expansions and at most finitely many of them. We will denote these sets by $\bar{\Gamma}$.
\end{dfn}
Informally speaking, an expansion of a formula $A$ is a formula resulted by replacing any formula after a box with disjunctions of the expansions of the formula. 
\begin{exam}\label{t2-3} 
For instance, the formula $\square(\neg \square (\square p \vee \square p) \vee \neg \square\square (p \vee p)) $ is an expansion of the formula $\square \neg \square \square p$
\end{exam}
So far, we have justified the Definition \ref{t2-1}. Let us investigate the intuitive meaning of the witnesses, as well. We claim that a natural interpretation is based on the interpretion of the outer boxes as meta-theories of the inner boxes. For simplicity, we call this kind of interpretation as the ordered interpretation. Therefore, to have an ordered interpretation we need to interpret all of the boxes in $A$ as the provability predicates of the theories in an ordered way. And, since for any theory we have a number which shows its layer in the hierarchy, it is enough to assign a natural number to a box. Consider that if we assign $n$ to a box, the intended meaning is that the interpretation of that box is the provability predicate for the theory $T_n$. This role is played by the concept of witness. In fact, a witness is just an assignment for the boxes in an ordered way.
\begin{nota}\label{t2-4} 
If $w_i$s are sequences of the natural numbers, by $(w_1, w_2, \ldots, w_n)$ we mean the concatenation of $w_i$s.
\end{nota}
\begin{dfn}\label{t2-5} 
Let $w$ be a sequence of natural numbers and $A$ be a modal formula. Then the relation $w \Vdash A$, which means $w$ is a witness for $A$, is inductively defined as follows:
\begin{itemize}
\item[$\bullet$]
If $A$ is an atom, $() \Vdash A$.
\item[$\bullet$]
If $A=B \circ C$, then $(w_1, w_2) \Vdash A$ if $w_1 \Vdash B$ and $w_2 \Vdash C$ for $\circ \in \{\wedge, \vee, \rightarrow\}$
\item[$\bullet$]
If $A=\neg B$, then $w \Vdash A$ if $w \Vdash B$.
\item[$\bullet$]
If $A=\square B $, then $(n, w) \Vdash A$ if $w \Vdash B$ and $n > m$ for all $m$ which appear in $w$.
\end{itemize}
Moreover, if $\Gamma$ is a sequence of modal formulas, by a witness for $\Gamma$, we mean a sequence of witnesses such that any witness $w_i$ in the sequence is a witness for $A_i$ in $\Gamma$.
\end{dfn}
Informally, a witness for a formula $A$ is a sequence of numbers which we assign to occurrences of the boxes in $A$ such that the number for outer box is greater than all numbers of inner boxes. This condition formalizes the idea that any outer box refers to the meta-theories in the hierarchy.
\begin{exam}\label{t2-6} 
For instance, $w=(n,m,k,r)$ is a witness for $\Box (p \rightarrow q) \vee \square( \neg \square p \rightarrow \square q)$ if $m>k,r$.
\end{exam}
The next definition is about evaluating a modal formula by an arithmetical substitution for atoms and a witness for the boxes in the formula.
\begin{dfn}\label{t2-7} 
Let $w$ be a witness for $A$ and $\sigma$ an arithmetical substitution which assigns an arithmetical sentence to a propositional variable. And also let $(M, \{T_n\}_{n=0}^{\infty})$ be a provability model. By $A^{\sigma}(w)$ we mean an arithmetical sentence which is resulted by substituting the variables by $\sigma$ and interpreting any box as the provability predicate of $T_n$ if the corresponding number in the witness for this box was $n$. The interpretation of boolean connectives are themselves. Moreover, if $\Gamma$ is a sequence of modal formulas $A_i$, and $w=(w_i)_{i}$ is its witness, by $\Gamma^{\sigma}(w)$ we mean the sequence of $A_i^{\sigma}(w_i)$.
\end{dfn}
\begin{exam}\label{t2-8} 
For the witness and the formula $A$ of the last example, $A^{\sigma}(w)$ would be $\Pr_{T_n} (p^{\sigma} \rightarrow q^{\sigma}) \vee \Pr_{T_m}( \neg \Pr_{T_k} (p^{\sigma}) \rightarrow \Pr_{T_r} (q^{\sigma}))$.
\end{exam}
We are ready to introduce the concept of the satisfiability of a formula in a provability model.
\begin{dfn}\label{t2-9} 
A sequent $\Gamma \Rightarrow \Delta$ is true in $(M, \{T_n\}_{n=0}^{\infty})$ when there are sequences of expansions $\bar{\Gamma}$ and $\bar{\Delta}$ of $\Gamma$ and $\Delta$, respectively, and witnesses $u$ and $v$ for $\bar{\Gamma}$ and $\bar{\Delta}$ respectively such that for any arithmetical substitution $\sigma$, $M \models \bar{\Gamma}^{\sigma}(u) \Rightarrow \bar{\Delta}^{\sigma}(v)$. Moreover, we say that a sequent $\Gamma \Rightarrow \Delta$ is true in a class of models $\mathcal{C}$, when there are uniform sequences of expansions and witnesses for all models. In a more precise way, we write $\mathcal{C} \vDash \Gamma \Rightarrow \Delta$, if there are sequences of expansion $\bar{\Gamma}$ and $\bar{\Delta}$ and witnesses $u$ and $v$ such that for all arithmetical substitutions $\sigma$ and all provability models $(M, \{T_n\}_{n=0}^{\infty})$ in $\mathcal{C}$, $M \models \bar{\Gamma}^{\sigma}(u) \Rightarrow \bar{\Delta}^{\sigma}(v)$.
\end{dfn}
Informally speaking, truth means the existence of expansions and witnesses such that the interpretation of a formula (or sequent) becomes true, independently of the use of the arithmetical substitutions.
\begin{rem}
Note that our definition of satisfiability allows us to use a disjunction of finitely many expansions of the formula instead of the original formula itself. In other words, if we want to show that $(M, \{T_n\}_{n=0}^{\infty}) \vDash A$, we could use finitely many expansions $B_1, B_2, \ldots, B_k$ for $A$ and find a witness for $\bigvee_{i=1}^k B_i$. The same is true for the sequents.
\end{rem}
Let us illuminate the Definition \ref{t2-9} with some examples.
\begin{exam}\label{t2-10} 
Let $(\mathbb{N}, \{T_n\}_{n=0}^{\infty})$ be a pair where $T_0=\PA$ and for any $n$, $T_{n+1}= T_n + \Rfn(T_n)$. Based on the definition, this pair is obviously a provability model. We want to show that the sentence $\Box(\Box A \rightarrow A)$ is true in the model. To do this, we need some expansions of the formula and a witness for them. For the expansions, just use the formula itself, and for a witness, first find a witness for $A$ and call it $w$; if $n$ is a number greater than all the numbers in $w$, then the sequence $(n+1, n, w, w)$ is a witness for $\Box(\Box A \rightarrow A)$. For any arithmetical substitution $\sigma$, we have $\mathbb{N} \vDash \Pr_{T_{n+1}}(\Pr_{T_n} (A^{\sigma}(w)) \rightarrow A^{\sigma}(w))$ since the theory $T_{n+1}$ can prove the reflection for $T_n$. As you can see, the idea of introducing a hierarchy to witness the boxes in modal sentences could kill the effect of G\"{o}del's second incompleteness theorem.\\
Let us illuminate the importance of the expansions with an example. Consider the sentence $\neg \Box (\neg \Box A \wedge A)$. We want to show that this sentence is true in the above mentioned provability model. (Note that this formula is provable in $\mathbf{S4}$.) Pick a witness $w$ for the sentence $A$, a number $n$ greater than all numbers in $w$ and the formula itself as its expansion. In this case we need two copies of the sentence, therefore we have to find a witness for $B= \neg \Box (\neg \Box A \wedge A) \vee \neg \Box (\neg \Box A \wedge A)$. It is easy to verify that the sequence $(n+2, n+1, w, w, n+1, n, w, w)$ is a witness for $B$. For any arithmetical substitution $\sigma$, we have 
\[
\mathbb{N} \vDash \neg \Pr_{n+2}(\neg \Pr_{n+1} (A^{\sigma}(w)) \wedge A^{\sigma}(w)) \vee \neg \Pr_{n+1}(\neg \Pr_n (A^{\sigma}(w)) \wedge A^{\sigma}(w))
\]
Because if we have both 

\[
\Pr_{n+2}(\neg \Pr_{n+1} (A^{\sigma}(w)) \wedge A^{\sigma}(w))
\]
and 
\[
\Pr_{n+1}(\neg \Pr_n (A^{\sigma}(w)) \wedge A^{\sigma}(w))
\]
then from the first part and the soundness of $T_{n+2}$ we have $\neg \Pr_{n+1} (A^{\sigma}(w))$ and from the second part and the fact that the provability predicate commutes with $\wedge$, we have $\Pr_{n+1}(A^{\sigma}(w))$, which is a contradiction. Therefore, the sentence is true in $\mathbb{N}$. It is easy to see that if we want to show the truth of the sentence $\Box(\neg \Box (\neg \Box A \wedge A))$, we should use $\Box B$ as an expansion of the formula. This observation shows the importance of the expansions, but is it possible to avoid them?
\end{exam}
\begin{exam}\label{t2-11} 
In this example we want to argue that some sentences do not have a witness in some provability models. Finding these kinds of examples is not hard. It is enough to think of formulas such as $p$ or $\Box p$. However, what we want to show here is finding an example to show the importance of the expansions in the definition. Think of the provability model of the last example and consider the formula $\neg \Box (\neg \Box p \wedge p)$. We showed that if we use two different copies of the formula, then the disjunction of those different copies have a witness in the provability model. We want to show that if we just use one copy, it is impossible to witness the formula. Assume that $w=(n,m)$ is a witness for $\neg \Box (\neg \Box p \wedge p)$ in the above mentioned provability model. Then since $w$ is a witness, we have $n>m$. On the other hand, we know that for any arithmetical substitution, we should have $\mathbb{N} \vDash \neg \Pr_n (\neg \Pr_m (p^{\sigma}) \wedge p^{\sigma})$. Use the arithmetical substitution which sends $p$ to $\Cons(T_m)$. Therefore, we have 
\[
\mathbb{N} \vDash \neg \Pr_n (\neg \Pr_m(\Cons(T_m)) \wedge \Cons(T_m))
\]
Based on the formalized G\"{o}del's second incompleteness theorem
\[
I\Sigma_1 \vdash \Cons(T_m) \rightarrow \neg \Pr_m(\Cons(T_m))
\]
since $I\Sigma_1 \subseteq T_{m+1}$ and $T_{m+1} \vdash \Cons(T_m)$ we have
\[
T_{m+1} \vdash \neg \Pr_m(\Cons(T_m))
\]
hence
$\mathbb{N} \vDash \Pr_{m+1}(\neg \Pr_m(\Cons(T_m)))$ and since $T_{m+1}$ has the reflection principle for $T_m$, $\mathbb{N} \vDash \Pr_{m+1}(\Cons(T_m))$. Since $n>m$ we have 
\[
\mathbb{N} \vDash \Pr_{n}(\neg \Pr_m(\Cons(T_m))) \wedge \Pr_{n}(\Cons(T_m))
\]
which contradicts our assumption. As you can see, our provability interpretation is sensitive to the use of expansions and also to the numbers of copies of expansions. In the following discussion, we will show that this property is an inherent property of the informal intuition behind modal formulas.
\end{exam}
\subsection{Discussion}
One of the complexities of our provability interpretation is the use of expansions and in this discussion, we want to justify its role. But before that, we need some observations. First of all, it seems that if we use the intuitive interpretation of the boxes as the provability predicates of different theories in the hierarchy of theories, meta-theories, meta-meta-theories and so on, the natural provability interpretation will be the following:
\begin{center}
\textit{A sentence $A$ is true in a provability model $(M, \{T_n\}_{n=0}^{\infty})$, if there is a witness $w$ for $A$ such that for all $\sigma$, $M \models A^{\sigma}(w)$}.
\end{center}
Which informally says that if you could witness the boxes in the formula $A$ in the provability model, then it is true. Note that this definition is simpler than ours and does not use any kinds of expansions. Let us concentrate on \textbf{S4} as the theory for our intuitive provability, and temporarily use the above definition as the definition of the truth. To interpret all axioms of the system \textbf{S4}, it is easy to see that we need two natural conditions on our model. First of all $T_{n+1}$ should be powerful enough to prove the reflection of the theory $T_n$ and secondly, all $T_n$s should be sound with respect to our model $M$ (This is what the nature of the provability in \textbf{S4} assumes; think of $\Box (\Box A \rightarrow A)$ and $\Box A \rightarrow A$, respectively.)
The sentence $\neg \Box(\neg \Box A \wedge A )$ is a theorem of \textbf{S4} and we expect that it should be true in any model with those two conditions. But in Example \ref{t2-11} we showed that there is no witness for the sentence and hence, with the definition above, the sentence is not true. The reason is the different roles of an occurrence of a box in a modal formula. To illuminate this fact, let us investigate the intuitive proof of the sentence $\neg \Box(\neg \Box A \wedge A )$ in \textbf{S4}. The proof is a proof by contradiction. Assume $\Box(\neg \Box A \wedge A )$, then because all theorems are true (axiom \textbf{T}), we have $\neg \Box A \wedge A$ and hence $\neg \Box A$. On the other hand, since the provability commutes with the conjunction (a consequence of the axiom \textbf{K}), we have $\Box A$, which is a contradiction. Consider the fact that the box in $\neg \Box A$ is inherited from the inner box in $\neg \Box A \wedge A$ and the box in $\Box A$ is inherited from the outer box in $\Box (\neg \Box A \wedge A)$. Therefore, to reach the contradiction, we need these two boxes refer to one layer in the hierarchy of theories which is impossible because the inner one is the theory and the other is the meta-theory and it is impossible to have $T_{n+1}=T_n$, because $T_{n+1}$ should prove the reflection for $T_n$.\\
What these investigations show, is actually the fact that one box in \textbf{S4} could have different roles. (In the above sentence, the outer box has two different roles, one as the meta-theory of the inner box and the other, as the theory itself.) Therefore, the natural way to interpret these boxes, is an approach which captures the different roles of a box at the same time, and this is not possible with the above simplified semantics, because it is obviously based on the assumption that any box has just one role which needs just one witness. Here is where we need expansions. In fact, the intended meaning of the expansions is using different copies of the formula in a disjunction and if you witness this disjunction, you have the power to witness one box in finitely many different ways; this technique empowers us to capture different roles of one box. (See Example \ref{t2-10} to find out how this technique works.)\\
There is another question to ask. Why do we need this kind of iterative expansion method and why is just the simple disjunction of the formula not enough? The answer is that for any fixed role available for one box, it is also possible to have different roles for inner boxes. Therefore, after any box you need a new disjunction. (Think of the sentence $\Box (\neg \Box(\neg \Box A \wedge A))$.) This is just what we call expansions.\\
As a conclusion for this discussion, let us compare our situation here in modal logic with first order logic. In first order logic, if we have a theorem of the form $\forall x \exists y A(x, y)$ where $A(x, y)$ is quantifier-free and if we want to witness $y$, Herbrand's theorem gives the answer; we can witness $y$ by terms in our language. However, we know that one term is not enough. The reason is simple. The existentially quantified $y$ could have different values (roles) and these different values (roles) can be captured by a disjunction of sentences $A(x, t(x))$ for some finite possible set of terms $t(x)$. The situation in modal logic is the same. We read boxes as the existence of theories and we want to witness them. Since there are different roles for any box, we need a disjunction to capture these different roles. In other words, we could interpret the expansions as some kind of Herbrandization of the modal formulas. 
\section{The Logic \textbf{K4}}
Intuitively, the logic $\mathbf{K4}$ is sound with respect to all kinds of provability interpretations. The reason is very simple. $\mathbf{K4}$ has two important modal axioms; the axiom $\mathbf{K}$ which means that the provability predicate is closed under modus ponens, and the axioms $\mathbf{4}$ which means that the provability of a sentence is also provable. The first axiom is a very easy fact and all strong enough meta-theories can prove it. On the other hand, if we have the minimum power in our meta-theory ($\Sigma_1$-completeness), the axiom $\mathbf{4}$ would be also easily proved. Consider the fact that these axioms are not only true but also provable and it justifies the use of the necessitation rule. Hence, $\mathbf{K4}$ is valid in all provability interpretations. In this section we want to formalize this intuitive argument and show that the logic \textbf{K4} is sound and also strongly complete with respect to the class of all provability models.
\subsection{Soundness}
If we denote the class of all provability models by $\mathbf{PrM}$, we have:
\begin{thm} \label{t3-1} (Soundness)
If $\Gamma \vdash_{\mathbf{K4}} A$ then $\mathbf{PrM} \vDash \Gamma \Rightarrow A$.
\end{thm}
\begin{proof}
To prove the soundness theorem for \textbf{K4}, we will use the cut-free sequent calculus for $\mathbf{K4}$ i.e. $G(\mathbf{K4})$. To simplify the proof, we use the following conventions: Firstly, if $\Phi$ and $\Psi$ are sequences of arithmetical sentences and $T$ is an arithmetical theory, by $T \vdash \Phi \Rightarrow \Psi$, we mean $T \vdash \bigwedge \Phi \rightarrow \bigvee \Psi$. Secondly, without loss of generality, we assume that the main formulas in all of the rules, except the exchange rule, are just the rightmost formulas in the sequent. We just use this assumption for the sake of brevity and clarity of the proof.\\ 
We want to prove the following claim by induction on the length of the proof in $G(\mathbf{K4})$.\\
 
\textbf{Claim}.
If $\Gamma \Rightarrow \Delta$ is provable in $G(\mathbf{K4})$, then there are sequences of expansions $\bar{\Gamma}$ and $\bar{\Delta}$ and witnesses $w_1$ and $w_2$ for $\bar{\Gamma}$ and $\bar{\Delta}$ respectively such that for any provability model $(M, \{T_n\}_{n=0}^{\infty})$ and any arithmetical substitution $\sigma$, $I\Sigma_1 \vdash \bar{\Gamma}^{\sigma}(w_1) \Rightarrow \bar{\Delta}^{\sigma}(w_2)$.\\

1. The case of axioms and structural rules. For the axiom $A \Rightarrow A$, it is enough to use $A$ as its expansion in both sides and just an arbitrary witness for $A$ in both sides, again. \\

For the exchange rule, just use the same expansions and witnesses after the application of the corresponding exchange.\\

For the weakening rule, if we prove $\Gamma, A \Rightarrow \Delta$ from $\Gamma \Rightarrow \Delta$, by IH, we could find expansions $\bar{\Gamma}$, $\bar{\Delta}$ and witnesses $w_1$ and $w_2$. Pick an arbitrary witness $w$ for $A$. For $\Gamma, A \Rightarrow \Delta$, use the sequences $\bar{\Gamma}, A$ and $\bar{\Delta}$, and for the witnesses use $(w_1, w)$ and $w_2$. It is easy to show that $I\Sigma_1 \vdash \bar{\Gamma}^{\sigma}(w_1), A^{\sigma}(w) \Rightarrow \bar{\Delta}^{\sigma}(w_2)$. The case for the right weakening is the same.\\

For the contraction rule, if we prove $\Gamma, A \Rightarrow \Delta$ from $\Gamma, A, A \Rightarrow \Delta$, then by IH, there are sequences of expansions $\{\bar{\Gamma}, \{\bar{A}_{i1}\}_{i=0}^{r}, \{\bar{A}_{j2}\}_{j=0}^{s}\}$ and $\Delta$ and also witnesses $w_1=(u, (v_{i1})_{i=0}^{r}, (v_{j2})_{j=0}^{s})$ and $w_2$. For the sequent $\Gamma, A \Rightarrow \Delta$, use the sequences of expansions $\{\bar{\Gamma}, \{\bar{A}_{i1}\}_{i=0}^{r}, \{\bar{A}_{j2}\}_{j=0}^{s}\}$ and $\bar{\Delta}$ and for the witnesses just use the same witnesses. In this case, because of the use of a finite set of different expansions instead of just one expansion, we can say that the semantics absorbs the contraction rule. The case for the right contraction is the same.\\

2. The case of propositional rules. In this case we just prove the case that the last rule is $R\wedge$; the other rules are similar and the argument is the same. If $\Gamma_1, \Gamma_2 \Rightarrow \Delta_1, \Delta_2, A\wedge B$, is proved from $\Gamma_1 \Rightarrow \Delta_1, A$ and $\Gamma_2 \Rightarrow \Delta_2, B$ then by IH we have the sequences of expansions $\bar{\Gamma}_1$, $\{\bar{\Delta}_1, \{\bar{A}_i\}_{i=0}^{r}\}$, $\bar{\Gamma}_2$, $\{\bar{\Delta}_2, \{\bar{B}_j\}_{j=0}^{s}\}$ and witnesses $w_1$ and $w_2=(u, (x_i)_{i=0}^{r})$ and $w'_1$, $w'_2=(u', (y_j)_{j=0}^{s})$. For the sequent $\Gamma_1, \Gamma_2 \Rightarrow \Delta_1, \Delta_2, A\wedge B$ use the sequences of expansions $\{\bar{\Gamma}_1, \bar{\Gamma}_2\}$, $\{\bar{\Delta}_1, \bar{\Delta}_2, \{\bar{A}_i \wedge \bar{B}_j \}_{i=0, j=0}^{i=r, j=s}\}$ and witnesses $(w_1, w'_1)$, $(u, u', ((x_i, y_j))_{i=0, j=0}^{i=r, j=s})$.\\

3. The case of modal rules. If $\Box \Gamma \Rightarrow \Box A$ is proved from $\Gamma, \Box \Gamma \Rightarrow A$, then by IH, we have the sequences of expansions $\{\bar{\Gamma}_1, \overline{\Box \Gamma}_2\}$ and $\{A_i\}_{i=0}^{r}$ and witnesses $w_1=((u_j)_{j=0}^{s}, (v_k)_{k=0}^{t})$ and $w_2=(x_i)_{i=0}^{r}$ where $u_j$ is a witness for the $j$th formula in $\bar{\Gamma}_1$ and $v_k$ is a witness for the $k$th formula in $\overline{\Box \Gamma}_2$. Pick number $n$ greater than all the numbers in $w_1$ and $w_2$. For the sequent $\Box \Gamma \Rightarrow \Box A$ use the sequences of expansions $\{\overline{\Box \Gamma}_1, \overline{\Box \Gamma}_2\}$ and $\Box \bigvee_{i=0}^{r} A_i$ and for the witnesses use $((n, u_j)_{j=0}^{s}, (v_k)_{k=0}^{t})$ and $(n, (x_i)_{i=0}^{r})$. By IH, we know that for any arithmetical substitution $\sigma$, 
\[
I\Sigma_1 \vdash \bigwedge_{j=0}^{s} \bar{\Gamma}^{\sigma}_1(u_j) \wedge \bigwedge_{k=0}^{t} \overline{\Box \Gamma}^{\sigma}_2(v_k) \rightarrow \bigvee_{i=0}^{r} A_i^{\sigma}(x_i).
\]
Since $I\Sigma_1 \subseteq T_n$, we have 
\[
T_n \vdash \bigwedge_{j=0}^{s} \bar{\Gamma}^{\sigma}_1(u_j) \wedge \bigwedge_{k=0}^{t} \overline{\Box \Gamma}^{\sigma}_2(v_k) \rightarrow \bigvee_{i=0}^{r} A_i^{\sigma}(x_i).
\]
Therefore, by $\Sigma_1$-completeness in $I\Sigma_1$ we have
\[
I\Sigma_1 \vdash \Pr_{n} (\bigwedge_{j=0}^{s} (\bar{\Gamma}^{\sigma}_1(u_j) \wedge \bigwedge_{k=0}^{t} (\overline{\Box \Gamma}^{\sigma}_2(v_k))) \rightarrow  \bigvee_{i=0}^{r} A_i^{\sigma}(x_i)),
\]
hence
\[
I\Sigma_1 \vdash \Pr_{n} (\bigwedge_{j=0}^{s} \bar{\Gamma}^{\sigma}_1(u_j)) \wedge \Pr_n(\bigwedge_{k=0}^{t} \overline{\Box \Gamma}^{\sigma}_2(v_k)) \rightarrow \Pr_n (\bigvee_{i=0}^{r} A_i^{\sigma}(x_i)).
\]
By formalized $\Sigma_1$-completeness of $T_n$ in $I\Sigma_1$ we have
\[
I \Sigma_1 \vdash \bigwedge_{k=0}^{t} \overline{\Box \Gamma}^{\sigma}_2(v_k) \rightarrow \Pr_n(\bigwedge_{k=0}^{t} \overline{\Box \Gamma}^{\sigma}_2(v_k))
\]
and hence
\[
I\Sigma_1 \vdash \bigwedge_{j=0}^{s} \Pr_{n} (\bar{\Gamma}^{\sigma}_1(u_j)) \wedge \bigwedge_{k=0}^{t} \overline{\Box \Gamma}^{\sigma}_2(v_k) \rightarrow \Pr_n (\bigvee_{i=0}^{r} A_i^{\sigma}(x_i)),
\]
which is what we wanted to prove and it completes the proof of the claim. \qed \\ 

For the proof of the soundness theorem, if $\Gamma \vdash_\mathbf{K4} A$ then there exists a finite set $\Delta \subseteq \Gamma$ such that $\Delta \vdash_\mathbf{K4} A$. Therefore, $G(\mathbf{K4}) \vdash \Delta \Rightarrow A$. By claim there are some expansions $\bar{\Delta}$ and $\{A_i\}_{i=0}^{r}$ for $\Delta$ and $A$, respectively and witnesses $u$ and $\{w_i\}_{i=0}^{r}$ such that for any arithmetical substitution $\sigma$, we have $I\Sigma_1 \vdash \bar{\Delta}^{\sigma}(u) \Rightarrow \bigvee_{i=0}^{r} A^{\sigma}_i(w_i)$. Since $M \vDash I\Sigma_1$, we have $M \vDash \bar{\Delta}^{\sigma}(u) \Rightarrow \bigvee_{i=0}^{r} A^{\sigma}_i(w_i)$. Pick $\bar{\Gamma}$ the same as $\Gamma$ after replacing the part of $\Delta$ by $\bar{\Delta}$. Moreover, choose $v$ as a witness for $\bar{\Gamma}$ as an arbitrary expansion of $u$ to $\bar{\Gamma}$. Hence, $M \vDash \bar{\Gamma}^{\sigma}(v) \Rightarrow \bigvee_{i=0}^{r} A^{\sigma}_i(w_i)$ which completes the proof of the soundness.
\end{proof}
\subsection{Completeness}
For the completeness theorem, the idea is to reduce the completeness of \textbf{K4} to the completeness of \textbf{GL} which is the well-known Solovay's theorem. (See Preliminaries and \cite{So}.) To do that, we need a translation from \textbf{K4} to \textbf{GL} which could transfer the provability behavior of \textbf{K4} to the provability behavior of \textbf{GL}.
\begin{dfn}\label{t3-2}
Let $A$ be a modal formula with $k$ boxes and let $Q=\{q_i\}_{i=0}^{\infty}$ be a sequence of atoms which are not used in $A$. Then, a translation $t$ based on $Q$ for the modal sentence $A$, is a sequence of $k$ numbers which assigns natural numbers to boxes in $A$ such that the number assigned to the outer box is greater than all the numbers for the inner boxes. And $A^t$ is defined as follows:
\begin{itemize}
\item[$(i)$]
If $A$ is an atom, $A^t=A$.
\item[$(ii)$]
$(B \circ C)^t=B^t \circ C^t$ for all $\circ \in \{\wedge, \vee, \rightarrow \}$
\item[$(iii)$]
$(\neg B)^t=\neg B^t$.
\item[$(iv)$]
$(\Box B)^t=\Box (\bigwedge_{i=0}^{n}q_i \rightarrow B^t)$ where $n$ is the number assigned to the box in $t$.
\end{itemize}
\end{dfn}
Informally, if we interpret a box as the provability predicate for the theory $S$, then the translation $t$ is just changing the provability predicate of the theory $S$ to the provability predicate of the theory $S+\{q_0, \ldots, q_n\}$ where $n$ is the number that $t$ assigns to that box. For instance, if $t=(1, 2, 1)$ and $A=\Box p \rightarrow \Box \Box p$, then $A^t$ will be the following modal formula:
\[
\Box (q_0\wedge q_1 \rightarrow p) \rightarrow \Box (q_0\wedge q_1 \wedge q_2 \rightarrow \Box (q_0\wedge q_1 \rightarrow p)).
\]
We want to show that this translation is complete, i.e.
\begin{thm}\label{t3-3}
If $\mathbf{GL} \vdash A^t$ for some translation $t$, then $\mathbf{K4} \vdash A$.
\end{thm}
The natural proof should be based on a technique of the transformation of transitive Kripke models to conversely well-founded transitive Kripke models, which is implemented by the following lemma.
\begin{lem}\label{t3-4}
Let $(K, R, V)$ be a finite transitive Kripke tree with clusters, $A$ a modal formula and $t$ a translation. Then there is a finite transitive irreflexive Kripke model $(K', R', V')$ such that for any node $k \in K$, there is a node $k' \in K'$ such that if $k \vDash A$ then $k' \vDash A^t$.
\end{lem}
\begin{proof}
First of all, for all subformulas $B$ of $A$, define the complexity of $B$, $C(B)$, as follows: If $B$ is box-free, define $C(B)=-1$. Otherwise, define $C(B)$ as the maximum assigned number in $B$ by $t$. Moreover, suppose that $C(A)=n$. To simplify the proof, let us make some conventions. We will use $I$ for clusters and for any $k \in K$, by $I(k)$ we mean the cluster of $k$. By a path $p=(k_{\alpha})_{\alpha=0}^{M}$, we mean a sequence of nodes in $K$ such that for any $\alpha$, $(k_{\alpha}, k_{\alpha+1}) \in R$ and if all the nodes of the path $p$ belong to the cluster $I$, we write $p \subset I$. Moreover, we write $p \prec p'$, when $p$ is a proper initial segment of $p'$. Finally, by $e(p)$ we mean the rightmost element of $p$, or in other words, the end of $p$.\\

For any cluster $I$ define $X(I)$ as follows: If $I$ consists of one irreflexive node $k$, $X(I)=\{k\}$ and if $I$ consists of some finite reflexive nodes, define $X(I)$ as the subset of all paths $p \subset I$ with length less than or equal to $n+2$. The idea is simple. We want to transform a transitive model to a nonreflexive transitive model. To accomplish this, we will unwind the reflexive clusters by some paths of nodes in that cluster and we will use $q$'s to refer to a copy of the node instead of itself, when we check the truth of the modal formulas. \\ 

Define $K'=\bigcup_{I}X(I)$ and $R'=R_1 \cup R_2$ where
\[
R_1=\bigcup_{(k, l) \in R, I(k)\neq I(l)} \{(a, b) \mid a \in X(I(k)) \; \text{and} \; b \in X(I(l))\}
\]
and
\[
R_2= \bigcup_{I}\{((p,p') \mid p \prec p'; p, p' \subset I \}.
\]
And finally, define 
\[
V'(r)=\{p \in K' \mid e(p) \in V(r)\} \cup \{k \mid k \in V(r) \; \text{and $k$ is irreflexive} \}
\]
for all atoms $r$ in $A$, and 
\[
V(q_i)=\{k \mid \text{$k$ is irreflexive}\} \cup \{p \mid |p| \leq n+2-i\}.
\]
Informally speaking, $K'$ is just the set $K$ where you replace each reflexive cluster $I$ with all paths of length less that or equal to $n+2$ of nodes in $I$; $R'$ and $V'$ are the natural relation and valuation induced by $R$ and $V$, respectively and $q_i$ is true in all irreflexive nodes and also in all paths of nodes in reflexive clusters with length bounded by $n+2-i$. We want to prove the following two claims.\\

\textbf{Claim.1}. The model $(K', R', V')$ is a finite transitive irreflexive Kripke model.\\

The finiteness follows from the definition. For the transitivity, suppose that $a, b, c \in K'$ and $(a, b) \in R'$ and $(b, c) \in R'$. Then, there are two cases. The first case is when $a$ and $b$ come from the same cluster. Hence, by definition, this cluster should be a reflexive cluster. Therefore, $a$ and $b$ are paths in this cluster and $a \prec b$. If $c$ comes also from this cluster, we will have $b \prec c$ and since $\prec$ is transitive, we have $a \prec c$ and hence $(a, c) \in R'$. But, if $c$ comes from another cluster, then the cluster of $c$ should be above the cluster of $b$ and hence it is also above the cluster of $a$ which is the same as $b$'s and then by definition we have $(a, c) \in R'$.\\
The proof of the second case, which is when $a$ and $b$ come from different clusters, is similar to the proof of the first case.\\

For the irreflexivity, suppose $(a, a) \in R'$. If $a$ is an irreflexive node in $K$, then it is impossible, by the definition of $R'$, to have $(a, a) \in R'$. If $a$ comes from a reflexive cluster, then again by the definition of $R'$, the path $a$ should be a proper segment of itself which is impossible.\\ 

\textbf{Claim.2}. For all subformulas of $A$ such as $B$, if $k \vDash B$, then 
\[
\begin{cases}
\forall p, |p| \leq n+1-C(B) \wedge e(p)=k, \; p \vDash B^t &\text{if $k$ is reflexive.}\\
k \vDash B^t  &\text{if $k$ is irreflexive.}
\end{cases} 
\]

and if $k \nvDash B$ then
\[
\begin{cases}
\forall p, |p| \leq n+1-C(B) \wedge e(p)=k, \; p \nvDash B^t &\text{if $k$ is reflexive.}\\
k \nvDash B^t  &\text{if $k$ is irreflexive.}
\end{cases} 
\]
To prove the claim, we use induction on $B$.\\

1. Atomic case. If $B$ is an atom, the claim easily follows from the definition of $V'$.\\

2. If $B=C \wedge D$ and $k \vDash C \wedge D$ then $k \vDash C$ and $k \vDash D$. If $k$ is irreflexive, then by IH, the claim holds. If $k$ is reflexive, then by IH, for all $p$ such that $|p| \leq n+1-C(C)$ and $ e(p)=k$, we have $ p \vDash C^t$. And also for all $p$ such that $|p| \leq n+1-C(D)$ and $ e(p)=k$, we have $p \vDash D^t$, and since $C(C \wedge D)=max\{C(C), C(D)\}$, then for all $p$ such that $|p| \leq n+1-C(C \wedge D)$ and $e(p)=k$, we have $p \vDash C^t \wedge D^t$.\\
If $k \nvDash C \wedge D$, then $k \nvDash C$ or $k \nvDash D$. W.l.o.g. assume $k \nvDash C$. If $k$ is irreflexive, the claim is obvious. If $k$ is reflexive, then by IH, for all $p$ such that $|p| \leq n+1-C(B)$ and $e(p)=k$ we have $ p \nvDash C^t$, and again since $C(C \wedge D)=max\{C(C), C(D)\}$ we have $\forall p, |p| \leq n+1-C(B \wedge D) \wedge e(p)=k, \; p \nvDash (C \wedge D)^t$.\\

3. If $B= \neg C$, then for irreflexive $k$, the claim is obvious from IH. If $k$ is reflexive and $k \vDash \neg C$, then $k \nvDash C$, and by IH, $\forall p, |p| \leq n+1-C(C) \; p \nvDash C^t$. Therefore, $\forall p, |p| \leq n+1-C(C) \; p \vDash \neg C^t$ and since $C(C)=C(\neg C)$ we have what we wanted. The other case is the dual of the first case.\\

4. The case for disjunction and implication is the same as the cases for conjunction and negation and we omit them here.\\

5. The modal case. This is the most important and the most complex part of the proof.\\ 

5.1. If $B=\Box C$ and $k \vDash \Box C$ then for all $l$ which $(k, l) \in R$, $l \vDash C$. Define $C(B)=m$.\\
5.1.1. If $k$ is irreflexive, we know that the nodes above $k$ in $K'$ are of two forms. The $l$'s which are irreflexive and $(k, l) \in R$ or the $p$'s where $p$ comes from a cluster $I$ above $k$ and $e(p)=l$. For the first kind of nodes, by IH we know that $l \vDash C^t$, therefore $l \vDash \bigwedge_{i=0}^{m} q_i \rightarrow C^t$. If we were in the second case, we know that $l \vDash C$ and again by IH, for all $p$ such that $|p| \leq n+1-C(C)$ and $e(p)=l$, we have $p \vDash C^t$. Therefore, for all $p, |p| \leq n+1-C(C)$ we have $p \vDash C^t$ and hence $p \vDash \bigwedge_{i=0}^{k} q_i \rightarrow C^t$. If $|p| > n+1- C(C)$, since $C(C)<C(B)=m$, we have $|p|> n+2-m$, and then by the definition of the valuation we know that $p \nvDash q_m$ and hence $p \nvDash \bigwedge_{i=0}^{m} q_i$ and thus $p \vDash \bigwedge_{i=0}^{m} q_i \rightarrow C^t$. Therefore, for all $p$ above $k$, we have $p \vDash \bigwedge_{i=0}^{m} q_i \rightarrow C^t$. Since for all nodes above $k$, $\bigwedge_{i=0}^{m} q_i \rightarrow C^t$ is true, we have $k \vDash \Box(\bigwedge_{i=0}^{m} q_i \rightarrow C^t)$ which means $k \vDash (\Box C)^t$.\\

5.1.2. If $k$ is reflexive from the cluster $I$, pick $p$ such that $|p| \leq n+1-m$. We want to show that $p \vDash \Box(\bigwedge_{i=0}^{m} q_i \rightarrow C^t)$. We know that all nodes above $p$ are of the form irreflexive $l$'s or $p' \subset J$ where $J$ is a cluster above $I$ or $p' \subset I$ where $p \prec p'$. For the first and second kinds, by a similar proof of 5.1.1, we can show that $l \vDash \bigwedge_{i=0}^{m} q_i \rightarrow C^t$ and $p' \vDash \bigwedge_{i=0}^{m} q_i \rightarrow C^t$. For the third case, if $|p'| > n+2-m$, then $p' \nvDash q_m$ and hence $p' \nvDash \bigwedge_{i=0}^{m} q_i$ and thus $p' \vDash \bigwedge_{i=0}^{m} q_i \rightarrow C^t$. If $|p'| \leq n+2-m$ then since $C(C) \leq m-1$ we have $|p'| \leq n+1-C(C)$. On the other hand, $k \vDash \Box C$, hence all nodes in $I$ satisfies $C$, and specially we have $e(p') \vDash C$, by IH, and by the fact that $|p'| \leq n+1-C(C)$, we have $p' \vDash C^t$ and therefore $\bigwedge_{i=0}^{m} q_i \rightarrow C^t$. We proved that for all nodes above $p$, we have $\bigwedge_{i=0}^{m} q_i \rightarrow C^t$ hence $p \vDash \Box(\bigwedge_{i=0}^{m} q_i \rightarrow C^t)$ which is what we wanted.\\

5.2. If $B=\Box C$ and $k \nvDash \Box C$, then there is a node $l$ such that $l \nvDash C$. Define $C(B)=m$.\\
5.2.1. If $k$ is irreflexive, we want to show that $k \nvDash \Box(\bigwedge_{i=0}^{m} q_i \rightarrow C^t)$. Consider that since $(k, l) \in R$, and $k$ is irreflexive, then $l \neq k$ and it belongs to a cluster above $k$. If $l$ is irrefelexive then by IH, $l \nvDash C^t$ and also since it is irreflexive, for all $i$, $l \vDash q_i$; hence $l \nvDash \bigwedge_{i=0}^{m} q_i \rightarrow C^t$ since $l \neq k$ and $(k, l) \in R$, $(k, l) \in R'$. Therefore, $k \nvDash \Box (\bigwedge_{i=0}^{m} q_i \rightarrow C^t)$. If $l$ is a reflexive node of the cluster $I$, then define $p \subset I$ as a path such that $|p|=n+2-m$ and $e(p)=l$. Since $C(C) \leq m-1$ then $|p| \leq n+1-C(C)$. By IH, $p \nvDash C^t$. (Consider that $m$ is the complexity of a boxed formula and therefore $m \geq 0$, hence $n+2-m \leq n+2$ and it means such a $p$ exists.). Moreover, we know that $p \vDash \bigwedge_{i=0}^{m} q_i$ since $|p| \leq n+2-i$ for all $i \leq m$, therefore, $p \nvDash \bigwedge_{i=0}^{m} q_i \rightarrow C^t$. Since the cluster of $k$ and the cluster of $l$ are different and $(k, l) \in R$, then $(k, p) \in R'$ and it means that $k \nvDash \Box(\bigwedge_{i=0}^{m} q_i \rightarrow C^t)$.\\

5.2.2. Consider the case that $k$ is reflexive. In this case, if $l$ belongs to a cluster above $k$, then the proof is the same as 5.2.1. If the cluster of $l$ and $k$ are the same (say $I$), we have the following construction: Pick $p$ such that $e(p)=k$ and $|p| \leq n+1-m$. We want to show that $p \nVdash \Box(\bigwedge_{i=0}^{m} q_i \rightarrow C^t)$. Pick $p' \subset I$ such that $e(p')=l$, $p \prec p'$ and $|p'|=n+2-m$. (It is enough to extend $p$ to a path with the ending $l$ and the length $n+2-m$. Note that $n+2-m > n+1-m$, which guarantee the existence of the expansion with the ending $l$ possibly different from $k$. Moreover, this length is less that $n+2$ and therefore $p'$ exists in our model as a path). We know that $C(C) \leq m-1$, hence $|p'| \leq n+1-C(C)$. By IH, $p' \nvDash C^t$. On the other hand, $p \vDash \bigwedge_{i=0}^{m} q_i$ since $|p| \leq n+2-i$ for all $i \leq m$, therefore, $p \nvDash \bigwedge_{i=0}^{m} q_i \rightarrow C^t$. Since $p \prec p'$, we can conclude that $p \nvDash \Box(\bigwedge_{i=0}^{m} q_i \rightarrow C^t)$.\\

The lemmas are obvious by the claim 2. For $B$ in the claim, choose $A$ itself, then if $k \vDash A$ and $k$ is irreflexive , then $k \vDash A^t$. But if $k$ is reflexive, pick $p=k$ as a path with length one. Hence $|p|=1 \leq n+1-C(A)$, since $C(A)=n$ and therefore, $p \vDash A^t$. Therefore, for any $k \vDash A$ there is a node $k' \in K'$ such that $k' \vDash A^t$.
\end{proof}
For the proof of Theorem \ref{t3-3} we have:
\begin{proof}
If $\mathbf{K4} \nvdash A$, then there is a finite transitive Kripke tree with clusters $(K, R, V)$ and a node $k$ such that $k \vDash \neg A$. If we apply Lemma \ref{t3-4} for $\neg A$, we can construct a finite transitive irreflexive Kripke model $(K', R', V')$ and a node $k'$ such that $k' \nvDash \neg A^t$. But $(K', R', V')$ is a model of $\mathbf{GL}$ and $\mathbf{GL} \vdash A^t$. A contradiction. Hence $\mathbf{K4} \vdash A$.
\end{proof}
Based on the completeness of the translations, which we have introduced, we are able to prove the completeness theorem. But, since we want to establish a more powerful completeness result, i.e. the strong completeness, we need one more lemma.
\begin{lem}\label{t3-5}
There is a hierarchy of theories $\{T_n\}_{n=0}^{\infty}$ such that for any $n$, $I\Sigma_1 \subseteq T_n$ and $T_n \subseteq T_{n+1}$ provably in $I\Sigma_1$ and also an arithmetical substitution $*$ such that for any modal formula $A$, if there exists a witness $w$ for $A$ such that $(M, \{T_n\}_{n=0}^{\infty}) \vDash A^*(w)$ for all $M \vDash I\Sigma_1$, then $\mathbf{K4} \vdash A$.
\end{lem}
\begin{proof}
Add infinitely many new atoms $Q=\{q_n\}_{n=0}^{\infty}$ to the language of modal logics, and apply all axioms and rules of the logic $\mathbf{K4}$ to the new language to construct a new system $\mathbf{K4(Q)}$ and do the same thing for the logic $\textbf{GL}$ to construct $\mathbf{GL(Q)}$. Pick the substitution $*$ as the uniform substitution of Solovay's theorem (see Preliminaries and \cite{Bo}). It simply says that for any $A$, $I\Sigma_1 \vdash A^*$ iff $\mathbf{GL(Q)} \vdash A$, where $A^*$ means the combination of substituting any atom $p$ with $p^*$ and interpreting all boxes as the provability predicate of $I\Sigma_1$. For any $n$, define $T_n=I\Sigma_1+\{q_i^*\}_{i=0}^{n}$. We claim that this $*$ and this hierarchy $\{T_n\}_{n=0}^{\infty}$ works for the claim of the lemma. First of all, it is easy to show that the hierarchy has the claimed conditions. Secondly, we have $M \vDash A^*(w)$ for all $M \vDash I\Sigma_1$. Therefore, $I\Sigma_1 \vdash A^*(w)$. Use $q_i$'s in the translations from $\mathbf{K4}$ to $\mathbf{GL}$. Since the interpretation of a box in any formula $\Box D$ with witness $m$ is $\Pr_{T_m}(D)$, and it is provably equivalent to $\Pr_{I\Sigma_1}(\bigwedge_{i=0}^{m} q_i \rightarrow D)$, it is easy to see that there is a translation $t$, such that $I\Sigma_1 \vdash A^*(w) \leftrightarrow (A^t)^*$. (In fact $t$ equals to the witness $w$.) Therefore, $I\Sigma_1 \vdash (A^t)^*$, by the uniform version of Solovay's theorem, $\mathbf{GL(Q)} \vdash A^t$, and by Theorem \ref{t3-3}, $\mathbf{K4(Q)} \vdash A$. It means that there exists a proof for $A$ in $\mathbf{K4(Q)}$. Since $A$ does not have any $q_i \in Q$, it is enough to put $q_i=\top$ everywhere in the proof to find a proof for $A$ in $\mathbf{K4}$.
\end{proof}
We want to prove the strong completeness theorem.
\begin{thm}\label{t3-6} (Strong Completeness)
If $\mathbf{PrM} \vDash \Gamma \Rightarrow A$, then $\Gamma \vdash_{\mathbf{K4}} A$.
\end{thm}
\begin{proof}
We know that there are the sequence of expansions $\bar{\Gamma}$, and expansions $B_1, \ldots, B_k$ of $A$ and witnesses $u$ for $\bar{\Gamma}$, and $w_1, \ldots, w_k$ for $B_1, \ldots, B_k$ such that for all provability models and all arithmetical substitution $\sigma$,
\[
M \vDash \bar{\Gamma}^{\sigma}(u) \Rightarrow \{B_i^{\sigma}(w_i)\}_{i=0}^{k}.
\] 
Pick the hierarchy and $*$ from Lemma \ref{t3-5}. Then for all $M \vDash I\Sigma_1$,
\[
M \vDash \bar{\Gamma}^{*}(u) \Rightarrow \{B_i^{*}(w_i)\}_{i=0}^{k}.
\] 
Hence
\[
I\Sigma_1 + \bar{\Gamma}^{*}(u) \vdash \bigvee_{i=0}^{k} B_i^{*}(w_i).
\]
Therefore there is a finite $\Delta \subseteq \bar{\Gamma}$ and a subset of witnesses $v$ from $u$, such that 
\[
I\Sigma_1 + \Delta^{*}(v) \vdash \bigvee_{i=0}^{k} B_i^{*}(w_i).
\]
Hence, for all $M \vDash I\Sigma_1$, we have 
\[
M \vDash \bigwedge \Delta^{*}(v) \rightarrow \bigvee_{i=0}^{k} B_i^{*}(w_i).
\]
By Lemma \ref{t3-5}, $\mathbf{K4} \vdash \bigwedge \Delta \rightarrow \bigvee_{i=0}^{k} B_i$, which means $\bar{\Gamma} \vdash_{\mathbf{K4}} \bigvee_{i=0}^{k} B_i$. Finally, since in the presence of the axiom $\mathbf{K}$, all expansions of a formula are equivalent to itself, $\Gamma \vdash_{\mathbf{K4}} A$.
\end{proof}
\section{The Logic \textbf{KD4}}
The logic $\mathbf{KD4}$ is a modal logic resulted by adding the axiom $\mathbf{D} : \Box A \rightarrow \neg \Box \neg A$ or equivalently $\neg \Box \bot$ to $\mathbf{K4}$. Therefore, intuitively, if we have the consistency of theories and also they are provable in their meta-theories, then the axioms of $\mathbf{KD4}$ should be valid. (Since we have the neccesitation rule, the sentence $\Box \neg \Box \bot$ is also provable and this is why we need the consistency statements to be provable, as well.) The formalization of these models is exactly what we will call consistent provability models and we will show that the logic $\mathbf{KD4}$ is sound and strongly complete with respect to these models.
\begin{dfn}\label{t4-1}
A provability model $(M,\{T_n\}_{n=0}^{\infty})$ is called consistent if for any $n$, $M$ thinks that $T_n$ is consistent and $T_{n+1} \vdash \Cons(T_n)$, i.e. $M \vDash \Cons(T_n)$ and $M \vDash \Pr_{T_{n+1}}(\Cons(T_n))$. Moreover, the class of all consistent provability models will be denoted by $\mathbf{Cons}$.
\end{dfn}
Let us prove the soundness theorem.
\begin{thm}\label{t4-2}(Soundness)
If $\Gamma \vdash_\mathbf{KD4} A$, then $\mathbf{Cons} \vDash \Gamma \Rightarrow A$.
\end{thm}
\begin{proof}
We use the soundness theorem for \textbf{K4}. If $\Gamma \vdash_\mathbf{KD4} A$, then 
\[
\Gamma + \Box \neg \Box \bot \wedge \neg \Box \bot \vdash_\mathbf{K4} A.
\]
Based on the soundness of \textbf{K4}, there are sequences $\bar{\Gamma} + \{\Box (\bigvee_{j=0}^{s_i} \neg \Box \bot) \wedge \neg \Box \bot\}_{i \in I}$ and $\{A_k\}_{k=0}^{t}$ as the expansions of $\Gamma + \Box \neg \Box \bot \wedge \neg \Box \bot$ and $A$, respectively and witnesses $u$, $(n_i,(m_{ij})_{j=0}^{s_i}, k_i)$ and $w_k$ such that for any provability model $(M, \{T_n\}_{n=0}^{\infty})$ and any arithmetical substitution $\sigma$,
\[
M \vDash \bar{\Gamma}^{\sigma}(u) + \{\Pr_{n_i} (\bigvee_{j=1}^{s_i} \neg \Pr_{m_{ij}}(\bot)) \wedge \neg \Pr_{k_i}(\bot)\}_{i \in I} \Rightarrow \bigvee_{k=0}^{t} A_k^{\sigma}(w_k))
\]
If we apply this fact on the consistent provability models, 
since $n_i>m_{ij}$ and for any $n$, $M \vDash \Pr_{n+1}(\neg \Pr_n(\bot))$, we have $M \vDash \Pr_{n_i} (\neg \Pr_{m_{ij}}(\bot)) $ for all $i \leq r$ and $j \leq s_i$. Moreover, since for any $n$, $M \vDash \neg \Pr_n(\bot)$, we have $M \vDash \neg \Pr_{k_i}(\bot)$. Therefore, for any consistent provability model $(M, \{T_n\}_{n=0}^{\infty})$ we have
\[
M \vDash \bar{\Gamma}^{\sigma}(u) \Rightarrow \bigvee_{k=0}^{t} A_k^{\sigma}(w_k)
\]
which completes the proof of the soundness for \textbf{KD4}.
\end{proof}
For the completeness theorem, the idea is reducing the completeness of \textbf{KD4} to the completeness of \textbf{K4} which was proved in the previous section.
\begin{thm}\label{t4-3}(Strong Completeness)
If $\mathbf{Cons} \vDash \Gamma \Rightarrow A$, then $\Gamma \vdash_{\mathbf{KD4}} A$.
\end{thm}
\begin{proof}
We know that there are the multiset $\bar{\Gamma}$, and expansions $B_1, \ldots, B_k$ of $A$ and witnesses $u$ for $\bar{\Gamma}$, and $w_1, \ldots, w_k$ for $B_1, \ldots, B_k$ such that for any consistent provability model and any arithmetical substitution $\sigma$,
\[
(M, \{T_n\}_{n=0}^{\infty}) \vDash \bar{\Gamma}^{\sigma}(u) \Rightarrow \{B_i^{\sigma}(w_i)\}_{i=0}^{k}.
\] 
Define $\Delta$ as a sequence which consists of an infinite number of the formula $\Box \neg \Box \bot$ and also an infinite number of the formula $\neg \Box \bot$. We claim that $\Gamma, \Delta \Rightarrow A$ is true in the class $\mathbf{PrM}$. For the expansions, use the same expansions for $\Gamma$ and $A$, and also use $\Delta$ itself, as its sequence of expansions. For witnesses, use $u$, $w_i$'s and for $\Delta$, for any number $n$, use $(n+1, n)$ for one of the formulas $\Box \neg \Box \bot$ and $n$ for one of the formulas $\neg \Box \bot$. Call this witness $v$. Let $(M, \{T_n\}_{n=0}^{\infty})$ be an arbitrary provability model. We claim that 
\[
M \vDash \bar{\Gamma}^{\sigma}(u), \Delta^{\sigma}(v) \Rightarrow \{B_i^{\sigma}(w_i)\}_{i=0}^{k}.
\]
Because when $M \vDash \bar{\Gamma}^{\sigma}(u), \Delta^{\sigma}(v)$ then $M \vDash \Delta^{\sigma}(v)$ which means for any $n$, 
\[
M \vDash \Pr_{n+1}(\neg \Pr_n (\bot)),
\]
and
\[
M \vDash \neg \Pr_n (\bot).
\]
Therefore, $(M, \{T_n\}_{n=0}^{\infty})$ is a consistent provability model and since $M \vDash \bar{\Gamma}^{\sigma}(u)$ we have, 
\[
(M, \{T_n\}_{n=0}^{\infty}) \vDash \bigvee_{i=0}^{k} B_i^{\sigma}(w_i).
\]
Therefore, for all provability models and all $\sigma$, we have
\[
M \vDash \bar{\Gamma}^{\sigma}(u), \Delta^{\sigma}(v) \Rightarrow \{B_i^{\sigma}(w_i)\}_{i=0}^{k}.
\]
Hence, by the strong completeness of $\mathbf{K4}$, we have 
$\Gamma, \Delta \vdash_{\mathbf{K4}} A$ and since all formulas in $\Delta$ are provable in $\mathbf{KD4}$, we have $\Gamma \vdash_{\mathbf{KD4}} A$.
\end{proof}
\begin{rem}\label{t4-4}
Note that the truth of a formula in a class of provability models means the existence of a \textit{uniform} sequence of expansions and also a \textit{uniform} witness for it. In other words, we have a fix sequence of natural numbers which works for all provability models in the class. Therefore, we could claim that sentences just describe the behavior of the natural numbers instead of some actual theories. What does it mean? It means that sentences do not describe the behavior of a concrete specific provability model, but instead, they talk about the roles of these ingredients in the structure (provability model) which are encoded by the natural numbers. Informally speaking, sentences just transcend the actual theories to their abstract roles in the structure of a provability model. (As an example, think of how the cardinal numbers transcend the concept of cardinality from the actual sets.) For instance, in the case of the logic $\mathbf{KD4}$, it describes the relation between a meta-theory $T_{n+1}$ and its theory $T_n$ which is the condition that the meta-theory is powerful enough to show the consistency of the theory. This is not about some actual theories which we use; it is about the power of the meta-theory in comparison to its theory. In other words, $\mathbf{KD4}$ describes the abstract condition of \textit{consistency} and \textit{provability of consistency}. This fact is true in all soundness-completeness results we propose in this paper.
\end{rem}
\section{The Logic \textbf{S4}}
Intuitively, if we have the property that all theories are sound and the soundness of theories are also provable in their meta-theories, all axioms of $\mathbf{S4}$, would be valid. The formalization of these models is exactly what we will call the reflexive provability models. In fact, we will show that the logic \textbf{S4} is sound and also strongly complete with respect to the class of all reflexive provability models.
\subsection{Soundness}
First of all we need a definition:
\begin{dfn}\label{t5-1}
A provability model $(M,\{T_n\}_{n=0}^{\infty})$ is reflexive if for any $n$, $M$ thinks that $T_n$ is sound and $T_{n+1} \vdash \Rfn(T_n)$, i.e. $M \vDash \Pr_{T_n}(A) \rightarrow A$ and $M \vDash \Pr_{T_{n+1}}(\Pr_{T_n}(A) \rightarrow A)$ for any sentence $A$. Moreover, the class of all reflexive provability models will be denoted by $\mathbf{Ref}$.
\end{dfn}
Let us prove the soundness theorem.
\begin{thm}\label{t5-2}(Soundness)
If $\Gamma \vdash_\mathbf{S4} A$, then $\mathbf{Ref} \vDash \Gamma \Rightarrow A$.
\end{thm}
\begin{proof}
To prove the soundness theorem, we will use the cut-free sequent calculus for $\mathbf{S4}$, i.e. $G(\mathbf{S4})$.
And, we will use the conventions of Theorem \ref{t3-1}. We want to prove the following claim:\\

\textbf{Claim}.
If $\Gamma \Rightarrow \Delta$ is provable in $G(\mathbf{S4})$, then there are sequences of expansions $\bar{\Gamma}$ and $\bar{\Delta}$ and also witnesses $w_1$ and $w_2$ for $\bar{\Gamma}$ and $\bar{\Delta}$, respectively and a number $n$ greater than all the numbers in $w_1$ and $w_2$, such that for any reflexive provability model $(M, \{T_n\}_{n=0}^{\infty})$ and any arithmetical substitution $\sigma$, $T_n \vdash \bar{\Gamma}^{\sigma}(w_1) \Rightarrow \bar{\Delta}^{\sigma}(w_2)$ is true in $M$. We will call the number $n$ as the context number.\\

The proof of the claim is by induction on the length of the proof of $\Gamma \Rightarrow \Delta$ and the proof for the non-modal cases are similar to the proof of Theorem \ref{t3-1}. But the difference is just the presence of the context number $n$ here. To find this number in all non-modal cases, if the case is the axiom case, any number works; for contraction and exchange, just use the same number in the induction hypothesis. For weakening, use the successor of the maximum of the context number of the induction hypothesis and the arbitrary chosen witness for the weakening formula. For the other cases, it is enough to use the maximum numbers of the induction hypothesis. We want to prove the case of the modal rules.\\

1. If $\Gamma, \Box A \Rightarrow \Delta$ is proved by $\Gamma, A \Rightarrow \Delta$, then by IH, we can find sequences of expansions $\{\bar{\Gamma}, \{A_i\}_{i=0}^{r}\}$, $\bar{\Delta}$ and witnesses $w_1=(u, (x_i)_{i=0}^{r})$ and $w_2$ and the context number $n$. For the sequent $\Gamma, \Box A \Rightarrow \Delta$, use the sequences of expansions $\{\bar{\Gamma}, \{\Box A_i\}_{i=0}^{r}\}$, $\bar{\Delta}$ and for the witnesses use $(u, ((n, x_i))_{i=0}^{r})$, $w_2$ and for the context number use $n+1$. By IH, we know that for all reflexive provability models and all arithmetical substitution $\sigma$, $M$ thinks
\[
T_n \vdash \bar{\Gamma}^{\sigma}(w_1), \{A_i^{\sigma}(x_i)\}_{i=0}^{r} \Rightarrow \bar{\Delta}^{\sigma}(w_2).
\]
We claim that there is a proof, formalizable in $I\Sigma_1$, for the following statement: If $T_n \subseteq T_{n+1}$, $T_{n+1} \vdash \Pr_n(A_i^{\sigma}(x_i)) \rightarrow A_i^{\sigma}(x_i)$ for all $q \leq i \leq r$ and 
\[
T_n \vdash \bar{\Gamma}^{\sigma}(w_1), \{A_i^{\sigma}(x_i)\}_{i=0}^{r} \Rightarrow \bar{\Delta}^{\sigma}(w_2)
\]
then
\[
T_{n+1} \vdash \bar{\Gamma}^{\sigma}(w_1), \{\Pr_n(A_i^{\sigma}(x_i))\}_{i=0}^{r} \Rightarrow \bar{\Delta}^{\sigma}(w_2).
\]
The proof is simple. We have $T_{n} \subseteq T_{n+1}$
and $T_{n+1} \vdash \Pr_n(A_i^{\sigma}(x_i)) \rightarrow A_i^{\sigma}(x_i)$.  Therefore,
\[
T_{n+1} \vdash \bar{\Gamma}^{\sigma}(w_1), \{\Pr_n(A_i^{\sigma}(x_i))\}_{i=0}^{r} \Rightarrow \bar{\Delta}^{\sigma}(w_2).
\]
The proof just uses the fact that all first order tautologies are provable and $\Pr$ is closed under modus ponens and all of these properties are provable in $I\Sigma_1$. Since $M \vDash I\Sigma_1$, $M$ thinks that this implication is true. On the other hand both of premises are true in $M$, because of IH and the condition of being a reflexive provability model. Therefore, $M$ thinks
\[
T_{n+1} \vdash \bar{\Gamma}^{\sigma}(w_1), \{\Pr_n(A_i^{\sigma}(x_i))\}_{i=0}^{r} \Rightarrow \bar{\Delta}^{\sigma}(w_2),
\]
which completes the proof.\\

2. If $\Box \Gamma \Rightarrow \Box A$ is proved by $\Box \Gamma \Rightarrow A$, then by IH we have sequences of expansions $\overline{\Box \Gamma}$ and some expansions $\{A_i\}_{i=0}^{r}$ and witnesses $w_1$ and $(x_i)_{i=0}^{r}$ and a context number $n$ such that for all arithmetical substitutions $\sigma$, $M$ thinks
\[
T_n \vdash \overline{\Box \Gamma^{\sigma}(w_1)} \Rightarrow \{A_i^{\sigma}(x_i)\}_{i=0}^{r}.
\] 
For the sequent $\Box \Gamma \Rightarrow \Box A$, use the expansion $\overline{\Box \Gamma}$ and $\Box (\bigvee_{i=0}^{r} A_i)$, and the witnesses $w_1$ and $(n, (x_i)_{i=0}^{r})$ and the context number $n+1$.\\
Based on the $\Sigma_1$-completeness available in $M$, $M$ thinks
\[
I\Sigma_1 \vdash \Pr_n( \bigwedge \overline{\Box \Gamma^{\sigma}(w_1)} \rightarrow \bigvee_{i=0}^{r} (A_i^{\sigma}(x_i))).
\]
Because the provability predicate commutes with the implications provably in $I\Sigma_1$, we have this property in $M$, hence
\[
I\Sigma_1 \vdash \Pr_n( \bigwedge \overline{\Box \Gamma^{\sigma}(w_1)}) \rightarrow \Pr_n(\bigvee_{i=0}^{r} (A_i^{\sigma}(x_i)))
\]
is true in $M$. Again by $\Sigma_1$-completeness, we have
\[
I\Sigma_1 \vdash \bigwedge (\overline{\Box \Gamma^{\sigma}(w_1)}) \rightarrow \Pr_n(\bigvee_{i=0}^{r} (A_i^{\sigma}(x_i)))
\]
true in $M$.
And finally since $T_{n+1}$ is an expansion of $I\Sigma_1$ provably in $I\Sigma_1$, we have the inclusion in $M$, hence
\[
T_{n+1} \vdash \bigwedge (\overline{\Box \Gamma^{\sigma}(w_1)}) \rightarrow \Pr_n(\bigvee_{i=0}^{r} (A_i^{\sigma}(x_i)))
\] 
is true in $M$ which completes the proof of the claim.\\
For the proof of the soundness theorem, if $\Gamma \vdash_\mathbf{S4} A$ then there exists a finite subset $\Delta$ of $\Gamma$ such that $\Delta \vdash_\mathbf{S4} A$. Then $G(\mathbf{S4}) \vdash \Delta \Rightarrow A$, then by the claim, there are sequences of expansions $\bar{\Delta}$ and $\{A_i\}_{i=0}^{r}$ and the witnesses $u$ and $(x_i)_{i=0}^{r}$ and a context number $n$ such that for all reflexive provability models $(M, \{T_n\}_{n=0}^{\infty})$ and all arithmetical substitution $\sigma$, we have $T_n \vdash \bar{\Delta}^{\sigma}(u) \Rightarrow \bigvee_{i=0}^{r} A_i^{\sigma}(x_i)$ in $M$. Therefore, by soundness of $T_n$ in $M$, we have $M \vDash \bar{\Delta}^{\sigma}(u) \Rightarrow \bigvee_{i=0}^{r} A_i^{\sigma}(x_i)$. Define $\bar{\Gamma}$ as the sequence of expansions of $\Gamma$ by using $\Gamma$ and replacing the subset $\Delta$ by $\bar{\Delta}$ and also use any arbitrary witnesses to extend $u$ to a witness for $\bar{\Gamma}$. Call this new witness $v$. We have
\[
M \vDash \bar{\Gamma}^{\sigma}(v) \Rightarrow \bigvee_{i=0}^{r} A_i^{\sigma}(x_i)
\]
which is what we wanted to prove.
\end{proof}
\subsection{Completeness}
For the completeness theorem, the idea is the same as the idea of the original proof of Solovay's theorem. We will modify the technique of encoding Kripke models in arithmetic. In this case, we need to encode transitive reflexive trees with clusters. Therefore we have two tasks. Firstly, finding a method to encode the clusters and secondly, modifying Solovay's construction to work with reflexive trees instead of irreflexive ones. 
\begin{lem}\label{t5-3}
Let $m$ be a natural number and $\{T_n\}_{n=0}^{N}$ be an increasing hierarchy of theories such that $I\Sigma_1 \subseteq T_0$, and for any $n$, $T_{n+1} \vdash \Rfn(T_n)$. Therefore, there are arithmetical sentences $A_1$, $A_2$, $\ldots$, $A_m$ such that:
\begin{itemize}
\item[$(i)$]
For any $i$ and $j$, if $i \neq j$ then $I\Sigma_1 \vdash A_i \wedge A_j \rightarrow \bot$
\item[$(ii)$]
$I\Sigma_1 \vdash \bigvee_{i=1}^{m}A_i$
\item[$(iii)$]
For any $n \leq N$, and any $i \leq m$, $T_{n+1} \vdash \neg \Pr_{T_n}(\neg A_i)$ 
\item[$(iv)$]
If we also assume that all theories in the hierarchy are consistent, then for any $n \leq N$ and any $i \leq m$, $\mathbb{N} \vDash \neg \Pr_{T_n}(\neg A_i)$ and $\mathbb{N} \vDash A_m$.
\end{itemize} 
\end{lem}
\begin{proof}
First of all, we want to prove the following claim:\\

\textbf{Claim.} For any increasing reflexive hierarchy $\{T_n\}_{n=0}^{N}$ and any natural number $1 \leq p$, there is another increasing hierarchy $\{T'_n\}_{n=0}^{Np}$ such that for any $n \leq N$, $T'_{np}=T_n$ and for any $i \leq Np-1$, $T'_{i+1} \vdash \Cons(T'_i)$. Moreover, if all of the theories in the $T$ hierarchy are consistent, all of the theories in the $T'$ hierarchy will be consistent, as well.\\

To prove the claim, define $T'_i$ as follows: For $i=np$, define $T'_i=T_n$, then for the any $np \leq i < (n+1)p-1 $ define $T'_{i+1}$ inductively as the theory $T'_i+\Cons(T'_i)$. First of all, we want to show that for any $np \leq i < (n+1)p-1 $, $T'_{i+1} \subseteq T'_{(n+1)p}$ and also $T'_{(n+1)p}$ proves the reflection principle for $T'_{i+1}$. The proof is based on the induction on $i$. If $i=np$, we know that $T'_{(n+1)p}$ proves the consistency for $T'_{np}$, hence $T'_{np+1} \subseteq T'_{(n+1)p}$. Moreover, since $T'_{(n+1)p} \vdash \Cons(T'_{np})$, it is easy to check that $T'_{(n+1)p}$ can prove the reflection principle for $T'_{np+1}=T'_{np}+\Cons(T'_{np})$. Suppose that we have the claim for $i$, and we want to prove it for $i+1$. By IH, $T'_{(n+1)p}$ proves the reflection principle for $T'_{i}$, hence it proves the consistency of $T'_i$ and hence $T'_{i+1} \subseteq T'_{(n+1)p}$. Again, it is easy to show that since $T'_{(n+1)p} \vdash \Cons(T'_{i})$, $T'_{(n+1)p}$ also proves the reflection principle for  $T'_{i+1}=T'_{i}+\Cons(T'_{i})$.\\

We claim that for any $i$, $T'_i \subseteq T'_{i+1}$ and $T'_{i+1}$ proves the consistency of $T'_i$. The proof is based on two different cases of the definition of $T'_{i+1}$. If we are in the first case, then $i+1=(n+1)p$ for some $n$. Then by what we proved so far, the claim is obvious. If we are in the second case, then $T'_{i+1}=T'_i+\Cons(T'_i)$, and hence the claim is again obvious from the definition. \\ 
Moreover, if the first hierarchy is consistent, then since all $T'_i$'s are subtheories of $T'_{Np}=T_N$, the second hierarchy is consistent, as well.\\

It is time to prove the lemma. If $m=1$, pick $A_1=(0=0)$; then it is easy to verify that this sentence satisfies the conditions of the lemma. The reason is that $T_{n+1}$ proves the consistency of $T_n$ and hence $T_{n+1} \vdash \neg \Pr_n(0 \neq 0)$. Moreover, if all theories are consistent, then $\neg A_1$ is not provable in $T_n$.\\
Assume that $m>1$ and use the hierarchy $T$ from the assumption of the lemma, and also use the aforementioned construction to construct the hierarchy $T'$, for $p=2m$. We want to define the sentences $A_i$ based on this new hierarchy. Define 
\[
B_r=\bigvee_{k=1}^{N} (\Cons(T'_{2km-2r}) \wedge \neg \Cons(T'_{2km-2r+1})) 
\]
for $1 \leq r \leq m-1$. Define $A_1= B_1$ and $A_r=\bigwedge_{i=1}^{r-1} \neg B_i \wedge B_{r}$ for $2 \leq r \leq m-1$ and $A_m=\bigwedge_{i=1}^{m-1}\neg B_i$. We claim that these $A_i$'s have the property in the lemma. First of all, because of the form of $A_i$'s, it is obvious that any two different $A_i$ and $A_j$ are contradictory and also $\bigvee_{r=1}^{m}A_r$. In fact, these claims are first order tautologies and hence they are provable in $I\Sigma_1$. We want to show that 
\[
T'_{2(n+1)m} \vdash \neg \Pr_{T'_{2nm}}(\neg A_r)
\]
We will prove the cases $r \neq 1, m$, $r=1$ and $r=m$ separately. Assume $r \neq 1, m$. Let us argue in $I\Sigma_1$. If $\neg A_r$ is provable in $T'_{2nm}$, then by definition $\bigvee_{i=1}^{r-1}B_i \vee \neg B_{r}$ is provable in $T'_{2nm}$. From $B_t$, $t \leq r-1$, we could conclude 
\[
\bigvee_I (\Cons(T'_{2km - 2t})) \vee \bigvee_J (\neg \Cons(T'_{2km-2t+1}))
\]
where $I=\{k \mid 2km-2t+1 \geq 2nm+1 \}$ and $J=\{k \mid 2km-2t+1 < 2nm\}$. First of all, we know that $T'_{2nm}$ proves $\Cons(T'_{2km-2t+1})$ if $k \in J$. The reason is that if $k \in J$, then $2km-2t+1 < 2nm$ and since the consistency of any theory is provable in the higher theory in $T'$ hierarchy, we can prove the consistency of $T'_{2km-2t+1}$ in $T'_{2nm}$. Therefore, we can conclude that the following is provable in $T'_{2nm}$.
\[
\bigvee_I (\Cons(T'_{2km-2t})).
\]
On the other hand, we know that if $k \in I$, then $k \geq n+1$ because $2km-2t+1 \geq 2nm+1$ is impossible when $k \leq n$. Therefore, $2km-2t \geq 2(n+1)m-2t$. Moreover, $2(n+1)m-2t \geq 2(n+1)m-2(r-1)$ since $t \leq r-1$, and since the hierarchy is increasing, $\Cons(T'_{2km-2t})$ implies $\Cons(T'_{2(n+1)m-2(r-1)})$. Hence, $B_t$ implies $\Cons(T'_{2(n+1)m-2(r-1)})$. Furthermore, from 
\[
\neg B_r=\bigwedge_{k=1}^{N} (\Cons(T'_{2km-2r}) \rightarrow \Cons(T'_{2km-2r+1}))
\] 
we conclude 
\[
\Cons(T'_{2(n+1)m-2r}) \rightarrow \Cons(T'_{2(n+1)m-2r+1}).
\]
Therefore, we have
\[
T'_{2nm} \vdash (\Cons(T'_{2(n+1)m-2r}) \rightarrow \Cons(T'_{2(n+1)m-2r+1})) \vee \Cons(T'_{2(n+1)m-2(r-1)}).
\]
Hence
\[
T'_{2nm} + \Cons(T'_{2(n+1)m-2r}) \vdash \Cons(T'_{2(n+1)m-2r+1}) \vee \Cons(T'_{2(n+1)m-2(r-1)}).
\]
But we have $2(n+1)m-2r+1 \leq 2(n+1)m-2(r-1)$; therefore
\[
T'_{2nm} \vdash \Cons(T'_{2(n+1)m-2(r-1)}) \rightarrow \Cons(T'_{2(n+1)m-2r+1}).
\]
And hence
\[
T'_{2nm} + \Cons(T'_{2(n+1)m-2r}) \vdash \Cons(T'_{2(n+1)m-2r+1}).
\]
Since $r \leq m$, we have $2(n+1)m-2r+1 \geq 2nm$, therefore we have
\[
T'_{2(n+1)m-2r+1} \vdash \Cons(T'_{2(n+1)m-2r+1}).
\]
Note that all the parts of this argument is formalizable in $I\Sigma_1$. For the first time we want to use $T'_{2(n+1)m}$ to reach the contradiction. Since $1 \leq r$, then $2(n+1)m-2r+1 < 2(n+1)m$, hence the consistency of $T'_{2(n+1)m-2r+1}$ is provable in $T'_{2(n+1)m}$. Therefore, since we are arguing in $T'_{2(n+1)m}$, we have the consistency of $T'_{2(n+1)m-2r+1}$. On the other hand, we showed
\[
\Pr_{T'_{2(n+1)m-2r+1}}(\Cons(T'_{2(n+1)m-2r+1})).
\]
By the formalized version of the second incompleteness theorem in $I\Sigma_1$, we know that if a theory proves its own consistency it is inconsistent; hence $T'_{2(n+1)m-2r+1}$ is inconsistent. A contradiction. Therefore, $T'_{2(n+1)m}$ could show that $\neg A_r$ is not provable in $T'_{2nm}$. \\

Note that the proof uses the form of $\neg A_r$ which has some positive $B_t$'s and one negative $B_r$. But Now if we are in the cases $r=1$ or $r=m$, then $\neg A_r$ has just positive $B_t$'s or just negative $B_t$'s. In these cases it is enough to use the part of the proof which investigates the corresponding $B_t$'s. Again argue in $I\Sigma_1$. For the case, $r=1$, if $T'_{2nm}$ proves $\neg A_1$, then $T'_{2nm}$ proves $\neg B_1$. Therefore,
\[
T'_{2nm} \vdash \bigwedge_{k=1}^{N} (\Cons(T'_{2km-2}) \rightarrow \Cons(T'_{2km-1})).
\]
Hence
\[
T'_{2nm} \vdash (\Cons(T'_{2(n+1)m-2}) \rightarrow \Cons(T'_{2(n+1)m-1})).
\]
Since $m \geq 1$, we have $2(n+1)m-1 \geq 2nm$ and hence
\[
T'_{2(n+1)m-1} \vdash (\Cons(T'_{2(n+1)m-2}) \rightarrow \Cons(T'_{2(n+1)m-1}))
\]
 and then since $2(n+1)m-1 > 2(n+1)m-2$, we have 
\[
T'_{2(n+1)m-1} \vdash \Cons(T'_{2(n+1)m-1}).
\]
Argue in $T'_{2(n+1)m}$. We have the consistency of $T'_{2(n+1)m-1}$. On the other hand, $T'_{2(n+1)m-1}$ proves its own consistency, hence by the formalized second incompleteness theorem, it should be inconsistent. A Contradiction. Therefore, $T'_{2(n+1)m}$ proves that $\neg A_1$ is not provable in $T'_{2nm}$.\\

For the proof of the case $r=m$, use the idea of $I$ and $J$ for positive $B_t$'s. It is enough to use $I$ and $J$, to show that if $\neg A_m$ is provable in $T'_{2nm}$, then $\Cons(T'_{2(n+1)m-2(m-1)})$ will be provable in $T'_{2(n+1)m-2(m-1)}$. After that, reaching a contradiction is the same as for the other cases.\\ 

Since $T'_{2nm}=T_{n}$, we have a proof for the part $(iii)$. For $(iv)$, if the hierarchy $T$ is consistent, then the hierarchy $T'$ is also consistent and hence if $\neg A_r$ is provable in $T'_{2nm}$ then we have 
\[
T'_{2(n+1)m-2r+1} \vdash \Cons(T'_{2(n+1)m-2r+1})
\] 
for cases $1<r<m$, and
\[
T'_{2(n+1)m-1} \vdash \Cons(T'_{2(n+1)m-1})
\]
for $r=1$,
and 
\[
T'_{2(n+1)m-2(m-1)} \vdash \Cons(T'_{2(n+1)m-2(m-1)})
\]
for $r=m$.
consider that the arguments for these statements are formalizable in $I\Sigma_1$ and hence they are true. For $1<r<m$, by the second incompleteness theorem, $T'_{2(n+1)m-2r+1}$ should be inconsistent. A contradiction. Therefore, $T'_{2nm}$ can not prove $\neg A_r$ and hence $T_n \nvdash \neg A_r$. The cases $r=1,m$ are similar. For the second part of $(iv)$, note that we know $A_m=\bigwedge_{r=1}^{m}\neg B_r$. We want to show that all $B_r$'s are false. We have
\[
B_r= \bigvee_{k=1}^{N} ( \Cons(T'_{2km-2r}) \wedge \neg \Cons(T'_{2km-2r+1})) 
\]
and since the whole $T'$ hierarchy is consistent, all statements $( \Cons(T'_{2km-2r}) \wedge \neg \Cons(T'_{2km-2r+1})) $ are false and hence $B_r$ is false. Then $\neg B_r$ is true and hence $A_m$ is true.
\end{proof}
\begin{lem}\label{t5-4}
Let $(K, R)$ be a finite reflexive transitive tree with clusters and let $k$ be one of the nodes in the root cluster. Moreover, let $(\mathbb{N}, \{T_n\}_{n=0}^{N})$ be a reflexive provability model. Then there exists a set of arithmetical sentences $\{S_i\}_{i \in K}$ such that
\begin{itemize}
\item[$(i)$]
If $i \neq j$, $T_0 \vdash S_i \rightarrow \neg S_j$.
\item[$(ii)$]
$T_{n+1} \vdash S_i \rightarrow \Pr_{n}(\bigvee_{(i, j) \in R} S_j)$.
\item[$(iii)$]
If $(i, j) \in R$ then $T_{n+1} \vdash S_i \rightarrow \neg \Pr_{n}(\neg S_j)$.
\item[$(iv)$]
$\mathbb{N} \vDash S_k$.
\end{itemize}
\end{lem}
\begin{proof}
Define a primitive recursive function $h: \mathbb{N} \to K$ similar to the $h$ function in the Solovay's proof of the completeness of $\mathbf{GL}$.
\[
h(0)=k \; \text{and} \; h(x+1)=
\begin{cases}
j & \text{if $(i, j) \in R$ \; \text{and}\; $\Prf_N(x, \neg S_j)$}\\
h(x) & \text{otherwise}
\end{cases}
\]
where $S_j=  P_{I(j)} \wedge A_j \wedge j=j$ and $P_{I(j)}= \exists y \forall x \geq y \; h(x) \in I(j)$ in which $I(j)$ means the cluster of $j$. Moreover, $A_j$'s are the sentences constructed in Lemma \ref{t5-3} for $m=Card(I(j))$ and the hierarchy $\{T_n+P_{I(j)}\}_{n=0}^{N}$. In addition, we choose $A_k$ as the sentence $A_m$ from Lemma \ref{t5-3}. By these sentences, we mean the sentences from the proof of Lemma \ref{t5-3}, and not what the lemma claims. The reason is that we have to be sure that these sentences are definable from the code of the function $h$ which has not been defined yet. The reason is the following:\\

The function $h$ should be defined based on the classical circular argument based on the fixed point lemma in $I\Sigma_1$. The important part is that the $A_j$'s constructed in Lemma \ref{t5-3} are arithmetical formulas based on the code of $P_{I(j)}$, which makes the whole circular argument possible. It is provable in $I\Sigma_1$ that $h$ is a function. (Note that we put $j=j$ in the definition of $S_j$ to make sure that there is at most one $j$ such that $x$ would be a proof for $\neg S_j$ and it makes $h$ a function.) It is also provable that $h$ eventually stops in some cluster and since $h$ is a function, this cluster is unique. The existence of such cluster is an obvious application of the fact that $h$ is an increasing function and the tree is finite. Note that all of these facts are provable in $I\Sigma_1$. To prove $(i)$, consider two cases. If $i$ and $j$ belong to different clusters, then $P_{I(i)}$ and $P_{I(j)}$ are contradictory based on what we claimed about the uniqueness of the limit cluster. This contradiction is also provable in $I\Sigma_1$ and hence in $T_0$. If $i$ and $j$ belong to the same cluster, then by Lemma \ref{t5-3}, we know that $A_i$ and $A_j$ are contradictory, provable in $I\Sigma_1$, and hence we reach a contradiction for $S_i \wedge S_j$ in $T_0$. For $(ii)$, we argue in $T_{n+1}$. If we have $S_i$, then we have $P_{I(i)}$ and there exists $x$ such that $h(x) \in I(i)$. Since this formula is $\Sigma_1$, by $\Sigma_1$-completeness we have $\Pr_n(h(x) \in I(i))$. Moreover, $h$ is provably increasing in $I\Sigma_1$ and hence in $T_n$, and also provably in $I\Sigma_1$ we know that $h$ eventually stops in some cluster, i.e. $\Pr_n(\bigvee_{J}P_J)$. But we have $\Pr_n(h(x) \in I(i))$. Therefore, the limit should be above $i$ which means $\Pr_{n}(\bigvee_{(i, j) \in R} P_{I(j)})$. On the other hand, by Lemma \ref{t5-3} we know that $I\Sigma_1 \vdash \bigvee_{i \in I} A_i$, and we can conclude that $\Pr_n(\bigvee_{(i, j) \in R} P_{I(j)} \wedge A_j)$, hence $\bigvee_{(i, j) \in R} S_j$.\\

For $(iii)$, we will argue in $T_{n+1}$ and the proof is by contradiction. If we have $S_i$ and $\Pr_{n}( \neg S_j)$ for some $j$ which $(i, j) \in R$, then there are two possibilities. First, when the clusters of $i$ and $j$ are different. We have $S_i=P_{I(i)} \wedge A_i$, hence we have $P_{I(i)}$ which means that there is some number $z$, such that for all $y \geq z$, $h(y) \in I(i)$. Moreover, we know that $\Pr_{n}(\neg S_j)$ and since $T_n \subseteq T_N$, we have $\Pr_{N}( \neg S_j)$. Therefore, there exists some $x$ such that $\Prf_{N}(x, \neg S_j)$. It is easy to see that we can pick $x \geq z$. Hence, we can conclude that $h(x+1) \in I(i)$. Since $(i, j) \in R$, $j$ is above all nodes in $I(i)$ and $\Prf_{N}(x, \neg S_j)$, hence $h(x+1)=j$. But $h(x+1)$ should belong to $I(i)$ and $j \notin I(i)$; a contradiction. Therefore, $\neg \Pr_{n}(\neg S_j)$.\\ 
Assume that the cluster of $i$ and $j$ is $I$. Then the statement $S_i \rightarrow \Pr_n(\neg S_j)$ is equivalent to
\[
P_I \wedge A_i \rightarrow \Pr_n(P_I \rightarrow \neg A_j).
\]
Since $\{T_n\}_{n=0}^{N}$ is a reflexive hierarchy, the hierarchy $\{T_n+P_I\}_{n=0}^{N}$ is also reflexive. Moreover, $A_t$'s are constructed for this hierarchy, hence by Lemma \ref{t5-3}, we know that
\[
T_{n+1}+P_I \vdash \neg \Pr_{T_n+P_I}(\neg A_j)
\]
which proves what we wanted.\\

For $(iv)$, since $h$ eventually stops in some cluster, there is a cluster $I$, such that $\mathbb{N} \vDash P_I$. If $I \neq I(k)$, since $h(0)=k$, there should be some first element $x$, such that $h(x) \in I$. Assume $h(x)=i$. Since $x\neq 0$, and $h(x)\neq h(x-1)$, we have $\Prf_N(x-1, \neg S_i)$ and hence, $\Pr_N(P_I \rightarrow \neg A_i)$. By Lemma \ref{t5-3}, the theory $T_N+P_I$ should be inconsistent, and therefore we have $T_N \vdash \neg P_I$. On the other hand, the theory $T_N$ is sound, hence $\mathbb{N} \vDash \neg P_I$ which contradicts to our assumption. Hence, $I=I(k)$ and therefore, $\mathbb{N} \vDash P_{I(k)}$. On the other hand, $T_N+P_{I(k)}$ is consistent because it is sound, and consequently by Lemma \ref{t5-3}, $A_k$ which was chosen to be the $A_m$ from the lemma, is true; hence $S_k=P_{I(k)} \wedge A_k$ is true.  
\end{proof}
The following lemma, uses the previous lemma to transfer the truth from a Kripke model to a reflexive provability model.
\begin{lem}\label{t5-5}
Assume the conditions of Lemma \ref{t5-4} and let $\{S_i\}_{i \in K}$ be defined as in that lemma. Define $\sigma$ as the arithmetical substitution  which sends the atom $p$ to $\bigvee_{i \vDash p}S_i$. For any $i \in K$, any modal formula $A$ and any witness $w$ for $A$ with elements less than $N$, we have:
\[
\begin{cases}
T_{max(w)+1} \vdash S_i \rightarrow A^{\sigma}(w) & \text{if} \; i \vDash A\\
T_{max(w)+1} \vdash S_i \rightarrow \neg A^{\sigma}(w) & \text{if} \; i \nvDash A
\end{cases}
\]
\end{lem}
\begin{proof}
We prove the lemma by induction on $A$. If $A$ is an atom and $i \vDash A$, then by the definition we have $T_0 \vdash S_i \rightarrow A^{\sigma}$. If $i \nvDash A$ then all $j$'s in $A^{\sigma}=\bigvee_{j \vDash A}S_j$ are different from $i$, and by $(i)$ in Lemma \ref{t5-4}, we conclude $T_0 \vdash S_i \rightarrow \neg A^{\sigma}$. The proof for the boolean cases is easy. For the modal case, if $i \vDash \Box B$, then for all $j$ which $(i, j) \in R$, we have $j \vDash B$. Since $w$ is a witness for $\Box B$, it is equal to $(n, u)$ where $n$ is greater than all the numbers in $u$. Therefore by IH, $T_{max(u)+1} \vdash S_j \rightarrow B^{\sigma}(u)$ for all $j$ above $i$. Hence,
\[
T_{max(u)+1} \vdash \bigvee_{(i, j) \in R} S_j \rightarrow B^{\sigma}(u).
\]
Since $n\geq max(u)+1$, we have
\[
T_n \vdash \bigvee_{(i, j) \in R} S_j \rightarrow B^{\sigma}(u).
\] 
Then
\[
I\Sigma_1 \vdash \Pr_n(\bigvee_{(i, j) \in R} S_j \rightarrow B^{\sigma}(u)),
\]
and consequently, 
\[
I\Sigma_1 \vdash \Pr_n(\bigvee_{(i, j) \in R} S_j) \rightarrow \Pr_n(B^{\sigma}(u)).
\]
By $(ii)$ in Lemma \ref{t5-4}, we have 
\[
T_{n+1} \vdash S_i \rightarrow \Pr_n(B^{\sigma}(u)),
\]
and $n=max(w)$. Thus, the proof for this case is finished.\\

If $i \nvDash \Box B$, then there exists $j$ which $(i, j) \in R$ and $j \nvDash B$. Again we have $w=(n, u)$, such that $n$ is greater than all the numbers in $u$. By IH, $T_{max(u)+1} \vdash S_j \rightarrow \neg B^{\sigma}(u)$. Since $n \geq max(u)+1$,
\[
T_n \vdash S_j \rightarrow \neg B^{\sigma}(u)
\]
and
\[
I\Sigma_1 \vdash \Pr_n(B^{\sigma}(u) \rightarrow \neg S_j )
\]
and then
\[
I\Sigma_1 \vdash \neg \Pr_n(\neg S_j) \rightarrow \neg \Pr_n(B^{\sigma}(u))
\]
and by $(iii)$ in Lemma \ref{t5-4}, we have
\[
T_{n+1} \vdash S_i \rightarrow \neg \Pr_n(B^{\sigma}(u))
\]
and again since $n=max(w)$, the proof is complete.
\end{proof}
We state and prove the completeness theorem.
\begin{thm}\label{t5-6}(Completeness)
Let $(\mathbb{N}, \{T_n\}_{n=0}^{\infty})$ be a reflexive provability model. If $(\mathbb{N}, \{T_n\}_{n=0}^{\infty}) \vDash A$, then $\mathbf{S4} \vdash A$. Therefore, 
if $\mathbf{Ref} \vDash A$, we have $\mathbf{S4} \vdash A$.
\end{thm}
\begin{proof}
Since $(\mathbb{N}, \{T_n\}_{n=0}^{\infty}) \vDash A$, there are expansions $B_1, \ldots, B_k$ of $A$ and witnesses $w_1, \ldots, w_k$ such that for all arithmetical substitutions $\sigma$, we have $\mathbb{N} \vDash \bigvee_{i=0}^{k}B_i^{\sigma}(w_i)$. Define $C=\bigvee_{i=0}^{k}B_i$ and $w=(w_i)_{i=0}^{k}$. Therefore, we know that $w$ is a witness for $C$ in $(\mathbb{N}, \{T_n\}_{n=0}^{\infty})$. We claim that $\mathbf{S4} \vdash C$. Pick $N$ greater than all the numbers in $w$. If $\mathbf{S4} \nvdash C$ then there exists a finite reflexive transitive tree with clusters $(K, R, V)$, such that in one of the nodes in the root cluster (say $k$), $C$ is false. Then by Lemmas \ref{t5-4} and \ref{t5-5}, we can construct an arithmetical substitution, such that $T_{max(w)+1} \vdash S_k \rightarrow \neg C^{\sigma}(w)$. Since the model is a reflexive provability model, all $T_m$'s are sound and hence $\mathbb{N} \vDash S_k \rightarrow \neg C^{\sigma}(w)$. But by Lemma \ref{t5-4} we know that $\mathbb{N} \vDash S_k$, thus $\mathbb{N} \vDash \neg C^{\sigma}(w)$, which contradicts with the assumption $\mathbb{N} \vDash C^{\sigma}(w)$. Therefore, $\mathbf{S4} \vdash C$. And finally, since in the presence of the axiom $\mathbf{K}$, all the expansions of a formula are equivalent to the formula itself, we have $\mathbf{S4} \vdash A$.\\
For the second part of the theorem, it is easy to verify that if $\mathbf{Ref} \vDash A$, then at least for one of the provability models $(\mathbb{N}, \{T_n\}_{n=0}^{\infty})$ we have $(\mathbb{N}, \{T_n\}_{n=0}^{\infty}) \vDash A$. And then the claim follows from the first part.
\end{proof}
\subsection{Uniform and Strong Completeness}
In this subsection we will strengthen the completeness theorem of the last subsection to a more strong version of uniform strong completeness theorem. The proof will be just the uniform version of the previous completeness proof. Therefore, first of all we need a uniform version of Lemma \ref{t5-3}.
\begin{dfn}\label{t5-7}
A hierarchy $\{T_n\}_{n=0}^{\infty}$ of theories is called uniform if there exists a $\Sigma_1$ formula $\Prf(x, y, z)$ such that for any $n$, $m$ and $A$, $\Prf(n, m, \lceil A \rceil)$ iff $m$ is a code of a proof for $A$ in $T_n$. The hierarchy is called uniformly increasing if it is a uniform hierarchy and also we have $I\Sigma_1 \subseteq T_0$ provably in $I\Sigma_1$ and $I\Sigma_1 \vdash \forall x \forall z (\exists y \; \Prf(x, y, z) \rightarrow \exists w \; \Prf(x+1, w, z))$. And finally it is called uniformly reflexive hierarchy if it is a uniformly increasing hierarchy such that for any formula $A$, $I\Sigma_1 \vdash \forall x  \exists y \; \Prf(x+1, y, \exists w \; \Prf(x, w, A) \rightarrow A)$.
\end{dfn}
\begin{lem}\label{t5-8}
Let $\{T_n\}_{n=0}^{\infty}$ be a uniformly reflexive hierarchy of theories. Then, there is an arithmetical sentence $A(x,y)$ such that:
\begin{itemize}
\item[$(i)$]
$I\Sigma_1 \vdash \forall x,z \leq y \; (x \neq z \wedge A(x,y) \wedge A(z,y) \rightarrow \bot)$
\item[$(ii)$]
For all $m$, $I\Sigma_1 \vdash \bigvee_{i=1}^{m}A(i,m)$
\item[$(iii)$]
For any $n$, and any $i \leq m$, $T_{n+1} \vdash \neg \Pr_{T_n}(\neg A(i,m))$ 
\item[$(iv)$]
If we also assume that all theories in the hierarchy are consistent, then for any $n$, and any $i \leq m$, $\mathbb{N} \vDash \neg \Pr_{T_n}(\neg A(i,m))$ and $\mathbb{N} \vDash A(m,m)$.
\end{itemize} 
\end{lem}
\begin{proof}
The proof is basically the same as the proof of Lemma \ref{t5-3}. The only difference is that, here we have to define everything uniformly. First of all we need to define the hierarchy $T'$. Since $T$ is a uniformly reflexive hierarchy, it is easy to prove that the hierarchy $T'$ is a uniform hierarchy. Note that the definition of this new hierarchy is also uniform in $p$, i.e. there exists a proof predicate $\Prf(x, y, z, t)$ which means that $y$ is a proof for $z$ in $T'_x$ when we choose $t$ as our $p$. Define, $B(x,y)$ as the following:
\[
B(x, y)=\exists z\geq 1 \; (\Cons(T'_{2zy-2x}) \wedge \neg \Cons(T'_{2zy-2x+1})), 
\]
and 
\[
A(x, y)= \forall 1 \leq z \leq x-1 \; \neg B(z, y) \wedge B(x, y).
\]
Note that $A(x, y)$ and $B(x, y)$ are the uniform versions of $A_r$ and $B_r$ in which $x$ stands for the index $r$ and $y$ for the number $m$. The proof of the properties we claimed is exactly same as the proof of Lemma \ref{t5-3}. The reason is that all properties are based on the standard numbers $n$, $i$ and $m$. The only exception is $(i)$, which is easily proved from the definition.
\end{proof}
\begin{thm}\label{t5-9}(Uniform Completeness)
Let $ \{T_n\}_{n=0}^{\infty}$ be a uniform reflexive hierarchy of sound theories. Then there exists an arithmetical substitution $*$, such that for any modal formula $A$, if there exists a witness $w$ such that for all $M \vDash \bigcup_nT_n$, $(M, \{T_n\}_{n=0}^{\infty}) \vDash A^*(w)$ then $\mathbf{S4} \vdash A$.
\end{thm}
\begin{proof}
First, note that according to the filteration method (see \cite{Ch}), there exists a primitive recursive algorithm which reads $A$ as an input and constructs a counter model (finite transitive reflexive tree with clusters) for $A$ if $\mathbf{S4} \nvdash A$, and outputs zero, otherwise. Call this primitive recursive function, $f$. Therefore, if we use $A_a$ to emphasize that the code for $A$ is $a$, we have $f(a)=(W_a, R_a, V_a, w_a)$ in which $w_a$ is a node in the root cluster such that $w_a \nvDash A_a$. The reason why such an $f$ exists is that the size of a counter model is elementary bounded by the size of the code of the formula. (See \cite{Ch}.) Assume that the function $\langle \cdot, \cdot \rangle$ is some canonical pairing function which is primitive recursive. Define $g(a)$ as the following primitive recursive function: Compute $f(a)$, change the name of all nodes $w$ in $W_a$ to $\langle w, a \rangle$ and code the whole model again. \\
Pick all $g(a)$'s and put all of them over one new reflexive root, $k$; and for valuation, use the induced valuation of the model plus the fact that the node $k$ does not accept any atom. Then, use the technique of Lemma \ref{t5-4} and define the function $h$ on the whole new model:
\[
h(0)=k \; \text{and} \; h(x+1)=
\begin{cases}
j & \text{if $R(h(x), z)$ \; \text{and}\; $\Prf_T(x, \neg S(z))$}\\
h(x) & \text{otherwise}
\end{cases}
\]
Where firstly, $T=\bigcup_{n=0}^{\infty}T_n$. It is easy to check that since the hierarchy is uniform, its union is also a recursively enumerable theory which has the following property: $I\Sigma_1 \vdash \Pr_n(A) \rightarrow \Pr_T(A)$. Secondly, $R(y, z)$ is a primitive recursive relation ($\Delta_1$ formula in $I\Sigma_1$) which reads nodes $y$ and $z$ and if $y \neq k$, it decides whether they belong to the same model $g(pr_0(z))$, and if yes, whether $(y, z)$ belongs to the relation of $g(pr_0(z))$, i.e. $R_{g(pr_0(z))}$. And if $y=k$, then the relation $R(y, z)$ decides whether $z$ is in the $g(pr_0(z))$ or not (where $pr_0(z)$ is the index of the model which $z$ belongs to). This $R$ is a formalization of the accessibility relation of the new model. Note that we have to choose $R$ in a way that the following holds:
\begin{itemize}
\item[$(i)$]
$I\Sigma_1 \vdash \forall x, y, z \; (R(x,y) \wedge R(y, z) \rightarrow R(x, z))$
\item[$(ii)$]
For any node $i \neq k$, $I\Sigma_1 \vdash \forall x (R(i, x) \rightarrow \bigvee_{R_{g(pr_0(i))}(i, j)} x=j) $
\end{itemize}
It is easy to find such an $R$. The idea is, first using $g$ to define a primitive recursive function $H(z)$ which reads $z$ and outputs the whole set above $z$. Then define $R(x, y)$ as the existence of a sequence $w$ from $x$ to $y$ such that for any $r$, $w_{r+1}$ belongs to $H(w_r)$. The proof for these two properties are starightforward. $(i)$ holds because of our transitive definition of $R$. $(ii)$ needs the claim that if $w$ is a sequence from $i$ to $x$, then $x \in H(i)$. Use induction on the length of $w$ to prove the claim.\\

And finally, the formula 
\[
S(z)=\exists y \forall x \geq y h(x)\in I(z) \wedge A(z, Card(I(z))) \wedge z=z
\]
where $I(z)$ is a primitive recursive function, which reads $z$ and computes the whole cluster of $z$. Note that here we use a uniform version of $S_i$'s, and consequently we need the uniform version of $A_r$'s. For any $i \neq k$, the model above $w_i$ is a finite reflexive transitive tree with clusters, and hence with the same arguments, we have the following:
\begin{itemize}
\item[$(i)$]
$T_0 \vdash \forall x,y \; (x \neq y \rightarrow (S(x) \rightarrow \neg S(y)))$.
\item[$(ii)$]
$T_{n+1} \vdash S(i) \rightarrow \Pr_{n}(\bigvee_{(i, j) \in R} S(j))$ for all $i \neq k$.
\item[$(iii)$]
If $(i, j) \in R$ then $T_{n+1} \vdash S_i \rightarrow \neg \Pr_{n}(\neg S_j)$ for all $i$.
\item[$(iv)$]
$\mathbb{N} \vDash S_k$.
\end{itemize}
Since the model above any node $i\neq k$ is a finite model, the proof is the same as the proof of Lemma \ref{t5-4}, with only some minor changes. Firstly, for $(i)$, we need the uniform version of the proof of Lemma \ref{t5-4}. It is implied by the facts that $h$ is a provably total function in $I\Sigma_1$ and also the part $(i)$ in Lemma \ref{t5-8}.\\
Secondly, for $(ii)$, we need to prove that if the function reaches $i$, then the limit cluster exists and it is above the cluster $I(i)$. It should be provable in $I\Sigma_1$. The idea is based on the fact that $h$ is increasing and also the fact that if $h$ reaches $i$, we can find the elements above $i$. These simple facts are provable by two properties of $R$ which are mentioned before. \\

Define the arithmetical substitution as follows: $p^{*}= \exists z \; S(z) \wedge V(z, p)$ where $V(z, p)$ is a primitive recursive predicate (i.e. a $\Delta_1$ formula in $I\Sigma_1$) which reads $z$ and $p$ and if $z \neq k$ decides whether $p$ is true in the node $z$ in the model $g(a)$, where $a=pr_0(z)$ is the index of the model which $z$ belongs to. And if $z=k$, then rejects for all $p$. Since $g$ is primitive recursive, this primitive recursive predicate exists. Note that $V$ is a formalization of the valuation of the new model. \\
By a similar proof of Lemma \ref{t5-5} we know that for all $i \neq k$, we have
\[
\begin{cases}
T_{max(w)+1} \vdash S_i \rightarrow A^{\sigma}(w) & \text{if} \; i \vDash A\\
T_{max(w)+1} \vdash S_i \rightarrow \neg A^{\sigma}(w) & \text{if} \; i \nvDash A
\end{cases}
\]
If $\mathbf{S4} \nvdash A$, then $i=w_a \nvDash A$, where $a$ is the code of $A$. We have
\[
T_{max(w)+1} \vdash S_i \rightarrow \neg A^{*}(w).
\]
Hence for all $n \geq max(w)+1$, 
\[
T_n \vdash S_i \rightarrow \neg A^{*}(w).
\]
Then by 
\[
T_{n+1} \vdash S_k \rightarrow \neg \Pr_{n}(\neg S_i),
\]
we have 
\[
T_{n+1} \vdash S_k \rightarrow \neg \Pr_n(A^{*}(w)).
\]
Since $T_{n+1}$ is sound, $\mathbb{N} \vDash \neg \Pr_n(A^{*}(w))$ which means $T_n \nvdash A^{*}(w)$, and since $n$ could be any big number, $T \nvdash A^{*}(w)$, therefore, there is $M$, a model of $T= \bigcup_nT_n$, such that $M \nvDash A^{*}(w)$, which is a contradiction. Hence, $\mathbf{S4} \vdash A$.
\end{proof}
Using the previous lemma, we are able to prove the strong completeness theorem.
\begin{thm}\label{t5-10}(Uniform Strong Completeness)
Let $ \{T_n\}_{n=0}^{\infty}$ be a uniformly reflexive hierarchy of sound theories. Then there exists an arithmetical substitution $*$, such that for any modal sequent $\Gamma \Rightarrow A$, if there exist witnesses $u$ and $v$ such that for all $M \vDash \bigcup_nT_n$, $(M, \{T_n\}_{n=0}^{\infty}) \vDash \Gamma^{*}(u) \Rightarrow A^{*}(v)$, then $\mathbf{S4} \vdash \Gamma \Rightarrow A$. Moreover, If $\mathbf{Ref} \vDash \Gamma \Rightarrow A$, then $\Gamma \vdash_{\mathbf{S4}} A$.
\end{thm}
\begin{proof}
Use the arithmetical substitution from the uniform completeness. Since 
\[
(M, \{T_n\}_{n=0}^{\infty}) \vDash \Gamma^*(u) \Rightarrow A^*(v)
\]
for all $M \vDash \bigcup_nT_n$, then $\bigcup_nT_n + \Gamma^*(u) \vdash A^*(v)$. Therefore, there is a finite subset $\Delta \subseteq \Gamma$ and a witness $w$, a subset of $u$, such that $\bigcup_nT_n + \Delta^*(w) \vdash A^*(v)$. Thus, for all $M \vDash \bigcup_nT_n$, we have 
\[
(M, \{T_n\}_{n=0}^{\infty}) \vDash \Delta^*(u) \Rightarrow A^*(v).
\]
By uniform completeness, we have $\mathbf{S4} \vdash \Delta \Rightarrow A$ and hence, $\mathbf{S4} \vdash \Gamma \Rightarrow A$.\\
The second part of the theorem, is obvious from the first part; because if $\mathbf{Ref} \vDash \Gamma \Rightarrow A$, then the assumption of the first part is true for some sequence of expansions $\bar{\Gamma}$ and $B_1, B_2, \ldots , B_r$. Hence $\bar{\Gamma} \vdash_{\mathbf{S4}} \bigvee_{i=0}^{r}B_i$. Since in the presence of the axiom $\mathbf{K}$, the expansions of a formula are equivalent to the formula itself, we have $\Gamma \vdash_{\mathbf{S4}} A$.
\end{proof}

\section{The Logics \textbf{GL} and \textbf{GLS}}
As Solovay showed in his pioneering work, \cite{So}, the logic $\mathbf{GL}$ is sound and complete for the interpretation that interprets all boxes as provability predicates in some appropriate theory. Moreover, he showed that if we change the definition slightly, we can also capture the logic $\mathbf{GLS}$. We translate his results into our framework and after defining constant and sound-constant provability models, we will show the soundness and completeness of \textbf{GL} and \textbf{GLS} for the classes of all constant provability models and all sound-constant provability models, respectively. In fact, the soundness-completeness theorems of these logics are just a new representation of Solovay's results. Consequently, we can claim that our provability interpretation is actually a generalization of Solovay's provability interpretation.
\subsection{The Case \textbf{GL}}
First of all the definition of the constant and sound-constant provability models:
\begin{dfn}\label{t6-1}
A provability model, $(M, \{T_n\}_{n=0}^{\infty})$ is constant if for any $n$ and $m$, $(M, \{T_n\}_{n=0}^{\infty})$ thinks that $T_n=T_m$, i.e. $M \vDash \Pr_{T_m}(A) \leftrightarrow \Pr_{T_n}(A)$ and $M \vDash \Pr_{T_0}(\Pr_{T_m}(A) \leftrightarrow \Pr_{T_n}(A))$ for any sentence $A$; and it is called a sound-constant model when it is constant and for any $n$, $M$ thinks that $T_n$ is sound, i.e. $M \vDash \Pr_{T_n}(A) \rightarrow A$ for any sentence $A$. The class of all constant provability models and the class of all sound-constant provability models will be denoted by $\mathbf{Cst}$ and $\mathbf{sCst}$, respectively. 
\end{dfn}
\begin{rem}\label{t6-2}
In the previous definition we used a notion for the equality of theories which seems ad-hoc and artificial. Here in this remark, we will justify that definition. Intuitively, $M$ thinks that two theories are equal, when their provability-based properties are the same. In a more precise way, we say that $M$ thinks $T_n$ and $T_m$ are equal, when for any modal sentence $\phi(p)$, any witness $w$ and any arithmetical substitution $\sigma$ for all atoms except $p$, $M \vDash \phi^{\sigma}(\Pr_m(A))(w) \leftrightarrow \phi^{\sigma}(\Pr_n(A))(w)$. We will show that this definition of equality is equivalent to the original one. First of all, if we use $\phi(p)=p$, we will have $M \vDash \Pr_{T_m}(A) \leftrightarrow \Pr_{T_n}(A)$. Moreover, if we use $\phi(p)= \Box (p \leftrightarrow q)$, $w=(0)$ and $\sigma$ where $q^{\sigma}=\Pr_n(A)$, we have $M \vDash \Pr_{T_0}(\Pr_{T_m}(A) \leftrightarrow \Pr_{T_n}(A))$. For the converse, we use induction on $\phi$ to show the following claim.\\

\textbf{Claim.} For any formula $\phi(p)$, any witness $w$ and any arithmetical substitution $\sigma$ for all atoms except $p$, $M$ thinks that both of the following statements are true: $\phi^{\sigma}(\Pr_m(A))(w) \leftrightarrow \phi^{\sigma}(\Pr_n(A))(w)$ and $T_0 \vdash \phi^{\sigma}(\Pr_m(A))(w) \leftrightarrow \phi^{\sigma}(\Pr_n(A))(w)$.\\

The atomic case and the boolean case are obvious. For the modal case, it is an easy consequence of the fact that $\Sigma_1$-completeness and some basic facts about the provability predicate are true in $M$.
\end{rem}
We are ready to prove the soundness-completeness result for $\mathbf{GL}$. First of all, a technical lemma.
\begin{lem}\label{t6-3}
Let $(M, \{T_n\}_{n=0}^{\infty})$ be a constant provability model. Then for any modal formula $A$, any witness $w$ and any arithmetical substitution $\sigma$, if $\mathbf{0}$ assigns zero to all the boxes of $A$, then $M$ thinks that both of the following statements are true: $A^{\sigma}(w) \leftrightarrow A^{\sigma}(\mathbf{0})$ and $T_0 \vdash A^{\sigma}(w) \leftrightarrow A^{\sigma}(\mathbf{0})$.
\end{lem}
\begin{proof}
Use induction on $A$. The case for the atoms and the boolean connectives are easy. For the modal case, if $A=\Box B$, and $w=(n, u)$, then by IH, $M$ thinks $T_0 \vdash B^{\sigma}(u) \leftrightarrow B^{\sigma}(\mathbf{0})$. Hence $ T_n \vdash B^{\sigma}(u) \leftrightarrow B^{\sigma}(\mathbf{0})$
and by $\Sigma_1$-completeness, 
$
M \vDash \Pr_n(B^{\sigma}(u) \leftrightarrow B^{\sigma}(\mathbf{0}))
$.
Thus $\Pr_n(B^{\sigma}(u)) \leftrightarrow \Pr_n(B^{\sigma}(\mathbf{0}))$
is true in $M$. Since $\Pr_n(B^{\sigma}(\mathbf{0}))$ and $\Pr_0(B^{\sigma}(\mathbf{0}))$ are equivalent in $M$, we have
\[
M \vDash \Pr_n(B^{\sigma}(u)) \leftrightarrow \Pr_0(B^{\sigma}(\mathbf{0})).
\]
For the other part of the claim, for $\Box B$, we have 
$
M \vDash \Pr_n(B^{\sigma}(u) \leftrightarrow B^{\sigma}(\mathbf{0})).
$
Therefore by $\Sigma_1$-completeness, $M$ thinks
$
T_0 \vdash \Pr_n(B^{\sigma}(u) \leftrightarrow B^{\sigma}(\mathbf{0})).
$
Hence $T_0 \vdash \Pr_n(B^{\sigma}(u)) \leftrightarrow \Pr_n(B^{\sigma}(\mathbf{0}))$ is true in $M$. But we know that $M$ thinks that
\[
T_0 \vdash \Pr_n(B^{\sigma}(\mathbf{0})) \leftrightarrow \Pr_0(B^{\sigma}(\mathbf{0})),
\]
therefore, $M$ thinks that
\[
T_0 \vdash \Pr_n(B^{\sigma}(u)) \leftrightarrow \Pr_0(B^{\sigma}(\mathbf{0})).
\]
\end{proof}
\begin{thm}\label{t6-4}(Soundness)
If $\Gamma \vdash_{\mathbf{GL}} A$, then $\mathbf{Cst} \vDash \Gamma \Rightarrow A$.
\end{thm}
\begin{proof}
If $\Gamma \vdash_{\mathbf{GL}} A$ then there exists a finite $\Delta \subseteq \Gamma$ such that $\mathbf{GL} \vdash \bigwedge \Delta \rightarrow A$. Then by Theorem \ref{t0-0}, we have $I\Sigma_1 \vdash \Delta^{\sigma}(\mathbf{0}) \rightarrow A^{\sigma}(\mathbf{0})$. Thus for any model $M$, $M \vDash \Gamma^{\sigma}(\mathbf{0}) \Rightarrow A^{\sigma}(\mathbf{0})$. Pick any arbitrary witnesses for $\Gamma$ and $A$ say $w_{\Gamma}$ and $w_A$. By using the Lemma \ref{t6-3} we will have $M \vDash \Gamma^{\sigma}(w_{\Gamma}) \Rightarrow A^{\sigma}(w_A)$.
\end{proof}
For the completeness of $\mathbf{GL}$ we have:
\begin{thm}\label{t6-5}(Uniform Strong Completeness)
Let $I \Sigma_1 \subseteq T$ be an r.e. $\Sigma_1$-sound theory and $\{T_n\}_{n=0}^{\infty}$ be a hierarchy of theories such that for any $n$, $T_n=T$, then there is an arithmetical substitution $*$ such that for any modal sequent $\Gamma \Rightarrow A$, if for all $M \vDash T$, we have $(M, \{T_n\}_{n=0}^{\infty}) \vDash \Gamma \Rightarrow A$, then $\Gamma \vdash_\mathbf{GL} A$. And especially, if $\mathbf{Cst} \vDash \Gamma \Rightarrow A$ then $\Gamma \vdash_\mathbf{GL} A$.
\end{thm}
\begin{proof}
Pick $*$ as the uniform arithmetical substitution in Solovay's completeness theorem for $T$ (see Preliminaries and \cite{Bo}). Pick $M \vDash T$, arbitrarily. We have $(M, \{T_n\}_{n=0}^{\infty}) \vDash \Gamma \Rightarrow A$, hence there are a sequence of expansions $\bar{\Gamma}$ and expansions $\{A_i\}_{i=0}^{r}$ of $A$ and witnesses $u$ and $w_i$ such that 
\[
M \vDash \bar{\Gamma}^*(u) \Rightarrow \bigvee_{i=0}^{r} A_i^{*}(w_i).
\] 
Since all the theories are equal, we can easily verify that for any formula $B$ and any witness $v$, $B^{*}(v)$ is equivalent to $B^{*}$, where $B^{*}$ means a combination of substituting all the atoms by $*$ and interpreting any box as the provability predicate for $T$. Then we have
\[
M \vDash \bar{\Gamma}^* \Rightarrow \bigvee_{i=0}^{r} A_i^{*}.
\]
Moreover, it is easy to prove that if $B$ is an expansion of $C$, then $B^*$ is equivalent to $C^*$ in $I\Sigma_1$ and hence
$
M \vDash \Gamma^* \Rightarrow A^*
$.
Since $M$ is arbitrary, we have 
$
T + \Gamma^* \vdash A^*
$,
therefore, there is a finite subsequence $\Delta \subseteq \Gamma$ such that
$
T + \Delta^* \vdash A^*
$.
Then by Solovay's uniform completeness theorem, we have $\Delta \vdash_{\mathbf{GL}} A$, thus
$
\Gamma \vdash_{\mathbf{GL}} A
$. 
For the second part of the theorem, it is easy to show that if $\mathbf{Cst} \vDash \Gamma \Rightarrow A$, then the assumption of the first part for $T=I\Sigma_1$ is met, and hence $\Gamma \vdash_{\mathbf{GL}} A$.
\end{proof}
\subsection{The Case \textbf{GLS}}
For the case of $\mathbf{GLS}$ we have:
\begin{thm}\label{t6-6}(Soundness)
If $\Gamma \vdash_{\mathbf{GLS}} A$, then $\mathbf{sCst} \vDash \Gamma \Rightarrow A$.
\end{thm}
\begin{proof}
If $\Gamma \vdash_{\mathbf{GLS}} A$, then there are formulas $B_1, B_2, \ldots , B_k$ such that $\Gamma \vdash_{\mathbf{GL}} \bigwedge_{i=1}^{k}(\Box B_i \rightarrow B_i) \rightarrow A$. By the proof of the soundness of $\mathbf{GL}$, we know that for any constant provability model and any arithmetical substitution $\sigma$, $M \vDash \Gamma^{\sigma}(\mathbf{0}) + \bigwedge_{i=1}^{k}(\Pr_0 (B_i^{\sigma}(\mathbf{0})) \rightarrow B_i^{\sigma}(\mathbf{0})) \Rightarrow A^{\sigma}(\mathbf{0})$. Since $M \vDash \Pr_0 (\phi) \rightarrow \phi$ for any arithmetical $\phi$, we have $M \vDash \Gamma^{\sigma}(\mathbf{0}) \Rightarrow A^{\sigma}(\mathbf{0})$. Use Lemma \ref{t6-3} to change the index of the theories from zero to any arbitrary witness.  
\end{proof}
Moreover, we have the completeness theorem.
\begin{thm}\label{t6-7}(Completeness)
Let $I \Sigma_1 \subseteq T$ be a sound r.e. theory and $\{T_n\}_{n=0}^{\infty}$ be a hierarchy of theories such that for any $n$, $T_n=T$. If $(\mathbb{N}, \{T_n\}_{n=0}^{\infty}) \vDash A$, then $\mathbf{GLS} \vdash A$; and especially, if $\mathbf{sCst} \vDash A$, then $ \mathbf{GLS} \vdash A$.
\end{thm}
\begin{proof}
By the assumption, we have $(\mathbb{N}, \{T_n\}_{n=0}^{\infty}) \vDash A$. Hence, there are expansions $\{A_i\}_{i=0}^{r}$ of $A$ and witnesses $w_i$ such that for all arithmetical substitutions $\sigma$,
$
\mathbb{N} \vDash \bigvee_{i=0}^{r} A_i^{\sigma}(w_i)
$. 
Since all the theories are equivalent, it is easy to show that for any formula $B$ and any witness $v$, $B^{\sigma}(v)$ is equivalent to $B^{\sigma}$, where $B^{\sigma}$ means a combination of substituting any atom by $\sigma$ and interpreting any box as the provability predicate for $T$. Therefore,
$
\mathbb{N} \vDash \bigvee_{i=0}^{r} A_i^{\sigma}
$. 
Moreover, it is easy to prove that if $B$ is an expansion of $C$, then $B^{\sigma}$ is equivalent to $C^{\sigma}$ in $I\Sigma_1$, hence
$
\mathbb{N} \vDash A^{\sigma}
$. 
Since $\sigma$ is arbitrary, based on Solovay's second completeness theorem, $\mathbf{GLS} \vdash A$.\\
For the second part of the theorem, it is easy to verify that if $\mathbf{sCst} \vDash A$ then the assumption of the first part for $T=I\Sigma_1$ is met and hence $\mathbf{GLS} \vdash A$.
\end{proof}
\section{The Extensions of \textbf{KD45}}
Intuitively, the logic $\mathbf{S5}$ does not admit any provability interpretation. The informal reason is as follows: The axiom $\mathbf{5}: \neg \Box A \rightarrow \Box \neg \Box A$ simply states that if $A$ is not provable in a theory $T_n$, then this fact will be provable in $T_{n+1}$, i.e.
\[
T_n \nvdash A \Rightarrow T_{n+1} \vdash \neg \Pr_n(A).
\]
Moreover, the axiom $\mathbf{T}$ asserts that all theories are sound, hence
\[
T_n \nvdash A \Leftrightarrow T_{n+1} \vdash \neg \Pr_n(A).
\]
We can use the last equivalence and the fact that the theory $T_{n+1}$ is recursively enumerable to find a decision procedure for the provability in the theory $I\Sigma_1 \subseteq T_n$, which is impossible. \\
The above argument is based on the axiom $\mathbf{5}$ and the fact that all theories are sound. But it is possible to weaken the soundness part to some kind of consistency assumption which generalizes the above argument to all extensions of the logic $\mathbf{KD45}$.
\begin{thm}\label{t7-1}
There is no provability model $(M, \{T_n\}_{n=0}^{\infty})$ such that 
\[
(M, \{T_n\}_{n=0}^{\infty}) \vDash \mathbf{KD45}.
\]
Hence, there are no provability models for any extension of the logic $\mathbf{KD45}$. Specially, $\mathbf{S5}$ does not have any provability interpretation.
\end{thm}
\begin{proof}
The proof we present here is more complex than the natural proof of this theorem, because we use weaker assumptions than what is available in $\mathbf{KD45}$. The reason of our interest in this more complex proof is that we will use the same proof for the case of the classical propositional logic, and in that case we just have access to these weaker assumptions.\\
We prove the claim by contradiction. Suppose that there is a provability model $(M, \{T_n\}_{n=0}^{\infty})$ such that $(M, \{T_n\}_{n=0}^{\infty}) \vDash \mathbf{KD45}$. First, we show that the following three statements are true in $M$, then we will use these statements to reach the contradiction.
\begin{itemize}
\item[$(i)$]
For any $n$, $M$ thinks that $T_{n+1} \nvdash \Pr_n(\bot)$. (Weak version of the consistency assumption.)
\item[$(ii)$]
For any $n$, there exist $N>n$ and $s<N$ such that $M$ thinks that $T_N \vdash \Pr_{n+1}(\Pr_n(\bot)) \rightarrow \Pr_s(\bot)$. (Weak version of the provability of the consistency assumption.)
\item[$(iii)$]
There are $m$, $n$ and $k$ such that $M$ thinks that for any arithmetical statement $\phi$,
\[
\neg \Pr_n(\phi) \rightarrow \Pr_{m+1}(\Pr_k(\phi) \rightarrow \Pr_m(\bot)).
\]
(Weak version of the axiom $\mathbf{5}$).
\end{itemize}
To prove $(i)$, for any number $n$, define $\Box^n \top$ as follows:
$\Box^0 \top=\top$ and $\Box^{n+1}\top=\Box \Box^n \top$. Consider the formula $\neg \Box \Box (\bot \wedge \Box^n \top) $, which is a theorem of $\mathbf{KD45}$. Therefore, we have expansions of this formula, of the form $\neg \Box\bigvee_{j=0}^{s_i} \Box \bigvee_{k=0}^{t_{ij}} (\bot \wedge B_{ijk})$ for $0 \leq i \leq r$, where $B_{ijk}$ is an expansion of $\Box^n \top$. Moreover, there are witnesses $w_i=(n_i, (m_{ij}, (u_{ijk})_{k=0}^{t_{ij}})_{j=0}^{s_i})$ for any of these expansions such that for any arithmetical substitution $\sigma$, we have
\[
M \vDash  \bigvee_{i=0}^{r}\neg \Box \bigvee_{j=0}^{s_i} \Box (\bigvee_{k=0}^{t_{ij}} (\bot \wedge B_{ijk}))^{\sigma}(w_i).
\]
Since the number of the boxes in $\Box^n \top$ is $n$, and witnesses for these boxes should be increasing, we have $m_{ij} \geq n$  and hence $n_i \geq n+1$. Define $M=min_{ij}(m_{ij})$ and $N=min_{i}(n_i)$. Since $B_{ijk}$ is an expansion of the theorem $\Box^n \top$, we can easily show that $B_{ijk}(u_{ijk})$ is provable in $I\Sigma _1$. Hence, it is easy to see that $M \vDash \neg \Pr_N(\Pr_M(\bot))$ and $N>M \geq n$. Therefore, if $M \vDash \Pr_{n+1}(\Pr_n(\bot))$, and since $N>M \geq n$, we have $\Pr_N(\Pr_M(\bot))$, which is a contradiction.\\

For $(ii)$, apply the same method to the formula $\Box (\Box \Box (\Box \bot \wedge \Box^n \top) \rightarrow \Box \bot)$ which is again a theorem of $\mathbf{KD45}$. Then there are expansions of the form $\Box \bigvee_{j=0}^{q_j} (\Box (\bigvee_{k=0}^{p_{ij}} \Box \bigvee_{l=0}^{t_{ijk}} B_{ijkl}) \rightarrow \Box \bot)$ where $B_{ijkl}$ is an expansion of $\Box \bot \wedge \Box^n \top$ and there are witnesses $w_i=(n_i, (m_{ij}, (r_{ijk}, (u_{ijkl})_{l=0}^{t_{ijk}})_{k=0}^{p_{ij}}, s_{ij})_{j=0}^{q_i})$ such that 
\[
M \vDash  \bigvee_{i=0}^{r} (\Box ( \bigvee_{j=0}^{q_j} (\Box (\bigvee_{k=0}^{p_{ij}} \Box \bigvee_{l=0}^{t_{ijk}} B_{ijkl}) \rightarrow \Box \bot))^{\sigma}(w_i).
\]
Once more, with the same reason as in the case $(i)$, $n \leq r_{ijk} < m_{ij} < n_i$. Define $N=max_{i}(n_i)$, $r=min_{ijk}(r_{ijk})$, $m=min_{ij}(m_{ij})$ and $s=max_i(s_i)$. Hence $N>m, r, s$ and $m>r \geq n$. Since the theories in the hierarchy $\{T_n\}_{n=0}^{\infty}$ is provably increasing, it is easy to prove
\[
M \vDash \Pr_N(\Pr_m(\Pr_r(\bot)) \rightarrow \Pr_s(\bot)).
\]
Because $m>r \geq n$, we have
\[
M \vDash \Pr_N(\Pr_{n+1}(\Pr_n(\bot)) \rightarrow \Pr_s(\bot)).
\]
Since $n$ is arbitrary, we have proved that for any $n$, there exists $N>n$, $s<N$ such that 
\[
M \vDash \Pr_N(\Pr_{n+1}(\Pr_n(\bot)) \rightarrow \Pr_s(\bot)),
\]
and this is what we wanted.\\

For $(iii)$ we know that $\neg \Box p \rightarrow \Box (\Box p \rightarrow \Box \bot)$ is provable in $\mathbf{KD45}$ and consequently it is true in the model. Therefore, there are some expansions of the formula $\neg \Box p \rightarrow \Box \bigvee_{j=0}^{s_i}( \Box p \rightarrow \Box \bot)$, and some witnesses $(n_i, m_i, (k_{ij}, l_{ij})_{j=0}^{s_i})$ for them, such that for any arithmetical substitution $\sigma$,
\[
M \vDash \bigvee_{i=0}^{r}(\neg \Box p \rightarrow \Box \bigvee_{j=0}^{s_i}( \Box p \rightarrow \Box \bot))^{\sigma}(n_i, m_i, (k_{ij}, l_{ij})_{j=0}^{s_i}).
\]
Define $n=max_{i}(n_i)$, $k=min_{ij}(k_{ij})$, $m=max_{i}(k_i)$ and $l=max_{ij}(l_{ij})$. It is easy to show that 
\[
M \vDash \neg \Pr_n(p^{\sigma}) \rightarrow \Pr_m( \Pr_k(p^{\sigma}) \rightarrow \Pr_l(\bot)).
\]
It is easily verified that we can increase $m$ and $l$; therefore, w.l.o.g. we can assume that $m=l+1$.
Send $p$ to $\phi$ to prove the claim, and this completes the proof of the statement $(iii)$.\\

For the proof of Theorem $\ref{t7-1}$, we want to use these three statements to reach a contradiction. First of all, to simplify the proof, use the following notation. For any $a$ and $b$, define the theory $T_{b_a}=T_b+\Cons(T_a)$. Thus, by $\Pr_{b_a}(A)$, we mean $\Pr_{T_{b_a}}$. Now, $(iii)$ would be equivalent to
\[
M \vDash \neg \Pr_n(p^{\sigma}) \rightarrow \Pr_{m_l}(\neg  \Pr_k(p^{\sigma})).
\]
Put $\phi=\Pr_{m_l}(\bot)$; therefore,
\[
M \vDash \neg \Pr_n(\Pr_{m_l}(\bot)) \rightarrow \Pr_{m_l}(\neg \Pr_k(\Pr_{m_l}(\bot))).
\]
On other hand by the formalized $\Sigma_1$-completeness, we have 
\[
I\Sigma_1 \vdash \neg \Pr_k(\Pr_{m_l}(\bot)) \rightarrow \neg \Pr_{m_l}(\bot),
\]
hence,
\[
T_{m_l} \vdash \neg \Pr_k(\Pr_{m_l}(\bot)) \rightarrow \neg \Pr_{m_l}(\bot).
\]
Moreover, by $\Sigma_1$-completeness, we have 
\[
I\Sigma_1 \vdash \Pr_{m_l}(\neg \Pr_k(\Pr_{m_l}(\bot)) \rightarrow \neg \Pr_{m_l}(\bot)).
\]
Therefore,
\[
I\Sigma_1 \vdash \Pr_{m_l}(\neg \Pr_k(\Pr_{m_l}(\bot))) \rightarrow \Pr_{m_l}(\neg \Pr_{m_l}(\bot)).
\]
And since $M \vDash I\Sigma_1$, we have
\[
M \vDash \neg \Pr_n(\Pr_{m_l}(\bot)) \rightarrow \Pr_{m_l}(\neg \Pr_{m_l}(\bot)).
\]
Based on G\"{o}del's second incompleteness theorem formalized in $I\Sigma_1$, we can conclude
\[
I\Sigma_1 \vdash \neg \Pr_{m_l}(\bot) \rightarrow \neg \Pr_{m_l}(\neg \Pr_{m_l}(\bot)).
\]
However, by $(i)$, we have 
\[
M \vDash \neg \Pr_{l+1}(\Pr_l(\bot)),
\]
hence $M \vDash \neg \Pr_{m_l}(\bot)$. Since $M \vDash I\Sigma_1$, 
\[
M \vDash \neg \Pr_{m_l}(\neg \Pr_{m_l}(\bot)).
\]
Therefore,
\[
M \vDash \Pr_n(\Pr_{m_l}(\bot)),
\]
and thus by definition of $T_{m_l}$ we have
\[
M \vDash \Pr_n(\Pr_m(\Pr_l(\bot))).
\]
By $(ii)$, there is some $N \geq l$ such that $M \vDash \Pr_N(\Pr_{l+1}(\Pr_l(\bot)) \rightarrow \Pr_s(\bot))$. W.l.o.g. pick this $N \geq n$. Since $N \geq n$, $M \vDash \Pr_N(\Pr_m(\Pr_l(\bot)))$, and therefore, $M \vDash \Pr_N(\Pr_s(\bot))$. Because $N>s$, we have $M \vDash \Pr_{N}(\Pr_{N-1}(\bot))$, which contradicts with $(i)$, and the proof follows.
\end{proof}
\section{A Remark on the Logic of Proofs}
As we mentioned in the Introduction, and as far as we know, the only successful attempt to find a natural provability interpretation for $\mathbf{S4}$ and hence, a formalization of the BHK interpretation is done by Artemov \cite{Art} and is called the logic of proofs. In this section, we will look into this approach and investigate some of its advantages and disadvantages.\\

The main idea of the logic of proofs, $\mathbf{LP}$, is using explicit proofs to avoid the non-standard proofs and hence to eliminate the incompleteness phenomenon. Let us give a more detailed account of this result. The language of $\mathbf{LP}$ is two sorted; one sort is for the explicit proofs and the other for the propositions. The first sort consists of proof terms constructed by the proof variables, proof constants and the proof connectives $+$, $\cdot$ and $!$, while the second sort contains terms constructed by the propositional variables, propositional connectives and the predicate $t:A$ in which $t$ is a proof term and $A$ is a proposition. Let us explain the intuitive meaning of these operations:\\

First of all we have to emphasize that in this interpretation, despite the usual case in mathematics, proofs can be multi-conclusion. To find a natural candidate for these multi-conclusion proofs, it is enough to consider any usual proof as a proof for all intermediate statements it uses to prove the conclusion. For instance, the usual proof $A_1, A_2, \ldots, A_n$ of $A_n$ will be interpreted as a proof for all $A_i$'s.\\
\\
1. The operation ``$!$". If $t$ is a proof for $A$, then $!t$ is a proof for the fact that ``$t$ is a proof for $A$". Therefore, the operator $!$ is the proof checker and could be interpreted as a self-awareness operator.\\
2. The operation ``$\cdot$". If $t$ is a proof for $A \rightarrow B$, and $s$ is a proof for $A$, then $t\cdot s$ is a proof for $B$. Intuitively, $\cdot$ means the application of Modus Ponens on the proofs.\\
3. The operation ``$+$". $t+s$ means the union of the proofs $t$ and $s$. Recall that our proofs are multi-conclusion and $t+s$ can be served as a proof for all conclusions of $t$ and $s$. Therefore if $t$ is a proof for $A$ and $s$ is a proof for $B$, then $t+s$ is a proof for both $A$ and $B$. To gain a better understanding, if we use the canonical way of changing usual proofs to multi-conclusion proofs, i.e. reading a usual proof as a proof for all intermediate statements in the proof, then $t+s$ just means putting $t$ and $s$ together. This is exactly what the symbol $+$ suggests.\\ 
4. The predicate ``$\; : \;$". The intuitive meaning of $t:A$ is that $t$ is a proof for $A$.\\

The formal system $\mathbf{LP}$ is a theory in this language to capture the intended meaning of the symbols defined above. The axioms are the following:\\
\\
1. A finite complete set of axioms for the classical propositional logic for the language of $\mathbf{LP}$,\\
2. $t:A \rightarrow A$,\\
3. $t:A \rightarrow B \rightarrow (s:A \rightarrow t\cdot s:B)$,\\
4. $t:A \rightarrow !t:t:A$,\\
5. $t:A \rightarrow s+t:A$,\\
6. $s:A \rightarrow s+t:A$.\\

The rules are the modus ponens and the neccesitation rule. The latter means that for any axiom $A$, we have $ \vdash c_A:A$, where $c_A$ is an appropriate constant exclusively used for $A$.\\

The natural interpretation for $\mathbf{LP}$ would be based on the usual proofs in Peano arithmetic. To formalize this idea, first of all we need a proof predicate: A proof predicate is a provably $\Delta_1$ formula (in $\mathbf{PA}$) $\Prf(x,y)$ with some natural basic properties (which we skip here. See \cite{Art}), and the following fundamental property:
\[
\mathbf{PA} \vdash A \;  \Leftrightarrow \exists x \Prf(x, \lceil A \rceil ).
\]

We want to interpret the language of $\mathbf{LP}$ with this natural provability interpretation. Define an arithmetical substitution $*$ as the following: Firstly, it interprets $\cdot$, $!$, $+$ and constants as the recursive functions on proofs in $\mathbf{PA}$ in the intended way. For instance, the function for $\cdot$ i.e., $\cdot^*$, will be the recursive function which reads the codes of the proofs for $A$ and $A \rightarrow B$ and replies the code of a proof for $B$. Why can we define such recursive functions? To show the fact that these functions exist, we need a proof; but here we just want to explain the main idea instead of a formal proof. For this reason, let us limit ourselves to the canonical proof predicate of $\mathbf{PA}$. In this case, it can be easily shown that we can define these functions in a recursive way. For instance, if $x$ and $y$ are proofs for $A \rightarrow B$ and $A$ respectively, for $\cdot^* (x, y)$ it is enough to put $y$ after $x$ and add the formula $B$ at the end. This is obviously a proof for $B$ and this process is clearly a recursive function. Moreover, note that for any $c_A$, $c_A^*$ is one of the proofs for the axiom $A^*$. The existence of such a $c_A^*$ also needs a proof, which we skip here. (See \cite{Art}.)\\

Up to this point, we have interpreted all the proof connectives as recursive functions. Use these interpretations to interpret all proof terms $t$. Note that for interpreting proof variables we use arbitrary natural numbers as the code of proofs. Extend the interpretation $*$ to formulas. The idea is just interpreting all atoms as arithmetical sentences, reading $t:A$ as the proof predicate $\Prf(t^*, \lceil A^* \rceil)$ and commute $*$ with all boolean connectives. For instance, the interpretation of $!x:p \rightarrow p$ would be $\Prf(!^*(n), \lceil \phi \rceil) \rightarrow \phi$ where the interpretations of $x$ and $p$ are $n$ and $\phi$, respectively.\\

These arithmetical interpretations are the natural and concrete interpretations of the proofs, and in \cite{Art} Artemov proved that $\mathbf{LP}$ is sound and complete with respect to the class of these arithmetical interpretations.
\begin{thm}
$\mathbf{LP} \vdash A$ iff $A^*$ is true for all arithmetical interpretations $*$.
\end{thm}
So far, we have found a natural proof interpretation for the system $\mathbf{LP}$. Finding a natural interpretation for $\mathbf{S4}$ into $\mathbf{LP}$ would be the next step. Subsequently, we can use the composition of these interpretations to find a proof interpretation for $\mathbf{S4}$ and hence for $\mathbf{IPC}$. We do not go into detail about the interpretation of the modal language into the system $\mathbf{LP}$, but the basic idea is the following: Interpret any box as the existence of a proof; thus, any modal sentence will be equivalent to a first order formula in the language of $\mathbf{LP}$. Therefore, we have quantifiers everywhere and specially in the scope of the predicate ``$:$". We know that there is no way to exchange the quantifiers with the proof predicate (which is the reason why the incompleteness phenomenon and non-standard proofs appear), but since we require all the codes of the proofs to be standard numbers, we extract all the quantifiers and convert the translated formula into the prenex form. Use the Skolemization technique to witness the existential quantifiers by the universal ones. These witnesses are called realizations. (This is where we essentially need ``$+$". It is important to note that by using Skolemization, we usually find a finite set of different witnesses and then we can roughly use $+$ to merge these finite witnesses into one.) Note that this is not how Artemov argues in \cite{Art}; however, we explained the realizations in the way that we think is more accessible and to show that why it is natural to have such a concept in the heart of the interpretation of the modal sentences. Let us illuminate the above interpretation by an example.
\begin{exam}\label{8}
Consider the modal formula $(\Box (p \rightarrow p) \wedge \neg \Box p) \rightarrow \Box \neg \Box p$. First, we have to interpret all of the boxes as the existence of the proofs. Hence, we have
\[ 
(\exists w: (p \rightarrow p) \wedge \neg \exists x:p \rightarrow \exists y: (\neg \exists z:p) .
\]
Then, by extracting the quantifiers, we have
\[
(\exists w: (p \rightarrow p) \wedge \forall x \neg x:p) \rightarrow \exists y \forall z \; y:\neg z:p,
\]
which is equivalent to
\[
\forall w \exists x \exists y \forall z  ((w: (p \rightarrow p) \wedge  \neg x:p) \rightarrow y:\neg z:p)).
\]
And finally by witnessing $y$ and $x$ by some terms $t(w, z)$ and $s(w, z)$, we have
\[
(w: (p \rightarrow p) \wedge \neg s(w, z):p) \rightarrow t(w, z):\neg z:p.
\]
This new formula is \textit{a} realization for the modal formula $(\Box (p \rightarrow p) \wedge \neg \Box p) \rightarrow \Box \neg \Box p$. Note that this realization is just \textit{one} possible realization of the formula and if we change the witnessing terms $t(w, z)$ and $s(w, z)$, we can find different realizations for the same formula.
\end{exam}

After introducing the realizations, Artemov proved the following: (See \cite{Art}.)
\begin{thm}
$\mathbf{S4} \vdash A$ iff there exists some realization $r$ such that $\mathbf{LP} \vdash A^r$.
\end{thm}
In sum, we can say that Artemov used two ingredients to find a provability interpretation for $\mathbf{S4}$. The first one is the interpretation of modal sentences via realizations into the system $\mathbf{LP}$. (Here the main idea is the interpretation of the boxes as the existence of the standard proofs.) And the second ingredient is the interpretation of the system $\mathbf{LP}$ via natural arithmetical proof interpretations. Therefore, the main idea of what Artemov did, is to use the system $\mathbf{LP}$ as a bridge to interpret $\mathbf{S4}$ via arithmetical proof interpretations.\\

Let us explain the advantages of this approach. First of all, it uses the explicit proofs and by the method of using realizations, it makes sure that everything is a standard proof in this context. Therefore, this approach actually kills the effect of G\"{o}del's incompleteness theorems and makes the proof interpretation more intuitive. Note that naturally, we do not count infinite non-standard proofs as proofs. Moreover, regardless of the relation between modal logics and explicit proofs, the system $\mathbf{LP}$ has its own applications. In fact, since it is a formal system for explicit proofs, it can be used as a theory to investigate the concept of proof and its natural calculus. Consequently, these formal systems are appropriate to investigate the formal verification in computer science or the behavior of justifications in formal epistemology.\\

However, this utopia of explicit proofs comes at a price. The price is a combination of two unintended properties: The first one is related to the fundamental change in the interpretation of the concept of provability and the second one is about the role of $\mathbf{LP}$ as an unbiased bridge. The problem is that the bridge is not neutral and somehow reflects its own behavior, which is not what we wanted.\\ 

Let us explain the first property by a simple example: Consider the modal sentence $\Box \neg \Box p$. The intended meaning of this sentence is the existence of a proof that shows $p$ is not provable. In other words, it states that there \textit{exists} a proof which shows that for \textit{any} possible proof $x$ for $p$, $x$ is not a proof for $p$. Let us use the logic of proofs' interpretation of the sentence. Since the occurrences of the inner and the outer box are negative and positive respectively, the meaning of the sentence is the existence of a term $t(x)$ such that $t(x): \neg x: p$. Forgetting the condition that the term $t(x)$ should be a term in the language, it means that for all $x$, there exists a proof $y=t(x)$ which proves  $\neg x: p$. In other words, it says that for \textit{any} possible proof $x$ for $p$, there \textit{exists} a proof which shows that $x$ is not a proof for $p$. It is easy to check that while the first interpretation is an $\exists \forall$ statement, the second one is a $\forall \exists$ statement, and it is obviously weaker than the first one. In fact, when we claim that we have a proof for unprovability of $p$, we mean a fixed uniform proof of the fact and we do not mean a machine (term) to transform a possible proof of $x$ to a proof $y$ that shows $x$ is not a proof for $p$.\\
What we showed above is just the difference for one statement. Nevertheless, the argument actually works for different kinds of sentences. The reason is simple: Logic of proofs needs to kill the presence of non-standard numbers. For this matter, it pushes out all the quantifiers. (It also changes the order of quantifiers to find a functional interpretation of proofs.) Since quantifiers do not commute with proof predicates, the sentence before pushing out the quantifiers is different from the sentence after that. The first sentence is the intended interpretation of provability and the latter is what the logic of proofs interprets as the meaning of provability. While this new interpretation is interesting and useful, it is not the intended interpretation of the informal provability and hence not \textit{the} interpretation of $\mathbf{S4}$.\\

In the following, we accept the functional interpretation of provability as what the logic of proofs proposed and we want to investigate the role of terms which we ignored in the previous argument. Let us explain the second property by a thought experiment: Think of the situation that you have another binary connective ``$?$" in the language of $\mathbf{LP}$ with the following intuitive meaning: If $s$ is a proof for $A \rightarrow A$ and $t$ is not a proof for $A$, then $?(s, t)$ is a proof of the proposition that ``$t$ is not a proof for $A$". Add the axiom
\[
(s:(A \rightarrow A) \wedge \neg t:A) \rightarrow ?(s, t):\neg t:A
\]
to the system $\mathbf{LP}$ and call it $\mathbf{LP?}$. It is clear that the connective $?$ and the above sentence are the negative versions of the connective $!$ and its corresponding axiom, respectively. What is not clear is the use of the seemingly useless part $s:A \rightarrow A$. We can explain this issue as the following: Assume that we have a non-proof $t$ for $A$ and we want to construct a proof of the sentence $\neg t:A$. We call this proof $r$. The important fact is that the sole access to $t$ is not enough to construct $r$ because the code of $A$ is also needed and this is actually where $s$ plays its role: $s$ is a proof for $A \rightarrow A$, hence we can use $s$ to compute the code of $A$ and now we have enough information to construct $r$.\\
Our method here seems ad-hoc and is certainly ugly, but remember that our goal is to perform an experiment about $\mathbf{LP}$ and fortunately this ad-hoc example is good enough to make our point. Now, let us be more formal about the natural arithmetical interpretation of this connective and this new system. Since we used explicit standard proofs, we know that there exists a recursive function which reads $t$ and the code of $A$ and if $t$ is not a proof for $A$, finds a proof of this fact. The reason is as follows: We know that $\Prf(x, y)$ is provably $\Delta_1$, hence if $\neg \Prf(t, \lceil A \rceil)$, we have 
\[
\mathbf{PA} \vdash \neg \Prf(t, \lceil A \rceil).
\]
Therefore, by the definition of a proof predicate we have 
\[
\exists r \Prf(r, \lceil \neg \Prf(t, \lceil A \rceil) \rceil).
\]
Use an unbounded search to find this $r$. Since it exists, our program halts and finds it. Now interpret $?(s, t)$ as the recursive function which reads $s$, finds the code of $A$ and then by the above-mentioned method finds the intended proof $r$. Thus, based on this new natural arithmetical interpretation, we can interpret the new axiom $(s:(A \rightarrow A) \wedge \neg t:A) \rightarrow ?(s, t) : \neg t:A$. Hence, we have a natural arithmetical interpretation for the system $\mathbf{LP?}$. On the other hand, one of the instances of the new axiom, i.e. $(w: (A \rightarrow A) \wedge \neg z:A) \rightarrow ?(w, z) : \neg z:A$, where $z$ and $w$ are proof variables, is the realization of the modal statement $\mathbf{5}': (\Box (A \rightarrow A) \wedge \neg \Box A) \rightarrow \Box \neg \Box A$ in this new language. (Simply, put $t(w, z)=?(w, z)$ and $s(w, z)=z$ in the Example \ref{8}.) The above discussion means that we can find a very natural provability interpretation of a variant of the axiom  $\mathbf{5}$. Recall that this axiom is not provable in $\mathbf{S4}$ and it seems contradictory with Artemov's completeness result. However, there is no contradiction. The reason is that ``$?$" is not in the original language of $\mathbf{LP}$, and hence you can not use it as a witness in the realization.\\
This observation shows that the arithmetical interpretation actually interprets a variant of the axiom $\mathbf{5}$, but the lack of the appropriate symbol in $\mathbf{LP}$ interferes with this fact. Therefore, the system $\mathbf{LP}$ does not reflect the whole power of the explicit proofs; it just chooses the appropriate part to witness all the theorems of $\mathbf{S4}$ and nothing more than that. In other words, the formalization of the provability interpretation via the explicit proofs is very sensitive to the language we use. If we change the language, then with the same arithmetical interpretation, we will capture different modal logics. Therefore, we can conclude that the soundness-completeness result for $\mathbf{S4}$ with respect to this kind of arithmetical interpretations is a soundness-completeness result for the language we use and not the natural arithmetical interpretation we choose. Now, a natural question would be the following: If we eliminate this language barrier and make the relation between modal logics and arithmetical interpretations as ``direct" as possible, then which modal logic corresponds to the whole power of the arithmetical interpretations of the proofs? By the direct connection, we roughly mean the following: For any modal sentence $A$, write it in the prenex form in a way that we defined before. Then, instead of witnessing the existential quantifiers by some terms in some language, witness them by some \textit{natural recursive functions} on the proofs in Peano arithmetic. Define the logic $E$ as the logic of all statements which are valid for this kind of arithmetical interpretations. Clearly, the question mentioned above is informal, but it is easy to verify that the answer is not $\mathbf{S4}$. The reason is that we can find an appropriate way to interpret a variant of $\mathbf{5}$ as we have shown above. It is appropriate because there is no \textit{a priori} reason to accept the recursive function $!$ and reject $?$. The first one finds a proof for $\Prf(m, n)$ if $\Prf(m, n)$ is true and the second function finds a proof for $\neg \Prf(m, n)$ if $\Prf(m, n)$ is false. Both of them are recursive and hence accessible for us as human beings. Note that $\Prf$ is a provably recursive predicate, and hence finding a proof for $\Prf(m,n)$ or a proof for its negation are similar computational tasks. (In the modal setting, the axioms $\mathbf{4}$ and $\mathbf{5}$ are intuitively different because we read $\Box A$ as $\exists x \Prf(x, A)$. This interpretation makes the sentence $\Sigma_1$ which is different from its negation.)\\

To sum up, the explicit proofs approach first kills all the quantifiers and puts some explicit witnesses for them. Therefore, it ignores the order of quantifiers and changes the canonical meaning of sentences and then as a consequence, it eliminates the computability based difference between provability and unprovability ($\Sigma_1$ vs $\Pi_1$) and maps both predicates to the boolean combinations of the explicit proof predicate $\Prf$, which belongs to the class $\Delta_1$. Consequently, the axioms $\mathbf{4}$ and $\mathbf{5}$ become similar and hence arithmetical interpretations can interpret a variant of $\mathbf{5}$ in a very natural way. Finally, to avoid this fact, the logic of proofs uses the language of $\mathbf{LP}$ to regain the difference between $\mathbf{4}$ and $\mathbf{5}$ by choosing what we need for $\mathbf{S4}$ and ignore the other natural functions which in this case is the function $?$. This argument shows that the approach of explicit proofs does not distinguish $\mathbf{4}$ from $\mathbf{5}$ in a \textit{natural} and \textit{essential} way and hence, it can not be considered as a formalization of the provability interpretation of $\mathbf{S4}$.\\

As the final part of this section, let us compare what we do in this paper with the approach of the explicit proofs. First of all, we use the canonical meaning of provability instead of the logic of proofs' functional interpretation. Moreover, we do not use any language as a bridge. Therefore, our soundness-completeness results represent the provability behavior of our arithmetical interpretations in a direct way. Secondly, to capture different modal logics, we impose different natural conditions on our provability models, specifically on the hierarchy of the theories. Therefore, we can claim that our approach can characterize different modal logics based on their different provability natures. Thirdly, our interpretation is based on the implicit proofs approach and hence it is a natural generalization of Solovay's work on $\mathbf{GL}$. But since the L\"{o}b axiom is based on the incompleteness phenomenon, the explicit approach does not capture it and thus does not accept Solovay's provability interpretation as a special case. Hence, the explicit approach can not serve as the general framework for provability interpretations.
\section{BHK Interpretations}
Briefly, what we are going to do in this section, is to introduce a formalization of the BHK interpretation. Indeed, we will generalize this goal to make a framework to formalize different kinds of provability interpretations which includes the BHK interpretation as a special case. Note that the usual BHK interpretation is not the unique provability interpretation of the propositional language; in fact, there are many of them. Some of them, can be characterized as the variants of the original BHK interpretation, and some can't. The reason is that those provability interpretations do not satisfy the intended philosophical conditions which we want to have, but they are still provability interpretations and they need an exact formalization if we want to use them. Let us illuminate the idea by two examples. The first one is a controversial variant of the BHK interpretation; it is obtained from the original BHK interpretation after relaxing the condition which says that there does not exist a proof for $\bot$. This interpretation informally corresponds to the minimal propositional logic, $\mathbf{MPC}$. The second example of the provability interpretation is also obtained from the original BHK interpretation, but now we read $\bot$ as the inconsistency, instead of the provability of the inconsistency. More precisely, and using the notation of G\"{o}del's translation, we have $\bot^g=\bot$, where $g$ stands for this new translation (which is different from what we used in the Introduction). This provability interpretation can not be characterized as a variant of the BHK interpretation because of some philosophical reasons, which we do not get into here. \\

In this section, we try to justify the claim that our provability interpretation can prepare an appropriate framework to formalize these different provability interpretations of the propositional logics. To implement this idea, we need two steps. First, we have to interpret all the connectives as what the provability interpretation demands; this step is done by the G\"{o}del's translation. The second step is interpreting the provability predicates (i.e. boxes in the modal translation) as the classical provability of the classical theories. For that reason, we need a hierarchy of theories to formalize the hierarchy of the intuitive provabilities in the definition of the provability interpretation and also a model to evaluate the truth value of our statements. This second step is done by the provability models.\\

What we discussed above is the general framework. Let us come back to the specific case, which is the original BHK interpretation. Is there a \textit{right} formalization of this interpretation? As we will show later, for different kinds of provability models, we have different BHK interpretations and these interpretations could show inherently different provability behaviors. Consequently, there are different formalizations for the BHK interpretation, instead of just a canonical one. The reason is that the BHK interpretation just interprets propositional connectives in a discourse of provability, but it does not say anything about the internal structure of the concept of provability. For instance, it does not say anything related to the power of the meta-theories compared to the lower theories. Since the BHK interpretation is the intended semantics for the intuitionistic logic, we have to accept that there could be different intuitionistic logics in terms of different interpretations of the power of our model and our theories. All of them are equally intuitionistic if we have just the BHK interpretation as the criterion.\\

The natural question is that what these intuitionistic logics are if we impose some natural conditions on the behavior of our model and our theories.\\
In the following, we will show that for some natural classes of the provability models such as the class of all models or the class of all reflexive models, we can characterize some propositional logics such as $\mathbf{BPC}$ and $\mathbf{IPC}$, respectively. For instance, in the case of reflexive models, the result shows that if we use the BHK interpretation with the philosophical commitment which states that all of the theories, meta-theories, meta-meta-theories and so on are sound and also, any meta-theory is powerful enough to prove the soundness of the lower theories, then the logic of the formulas which are valid under this kind of BHK interpretation, is the usual propositional intuitionistic logic. But, if we choose the minimal power, which does not assume any non-trivial condition on the hierarchy of the meta-theories, then the logic will change to $\mathbf{BPC}$. However, what is important here is that all of these logics could be characterized as intuitionistic logics. This fact can explain the reason behind the disputes about finding the correct formalization of the intuitionistic logic. For instance, in \cite{Ru}, Ruitenburg argues that the \textit{truly} intuitionistic logic is not $\mathbf{IPC}$ and he proposed $\mathbf{BPC}$ as the right one. Our approach here has a plural nature, and it tries to explain why with the same informal semantics (the BHK interpretation) there are different proposed logics.\\

Finally, a remark about classical logic. Since we have the axiom of the excluded middle in classical logic, we should have the following condition on provability models: \textit{Either the ``provability of $p$" is provable or it is provable that the provability of $p$ implies the provability of $\bot$}. This means that the meta-theory should be powerful enough to prove the unprovability of almost all unprovable formulas. As we saw in the case of the logic $\mathbf{S5}$, it contradicts with the natural condition that all the theories should be recursively enumerable. Therefore, intuitively speaking, we have to say that classical logic is beyond the scope of the BHK interpretation. In the following, we will prove this fact in a precise way. 
\begin{dfn}\label{t9-1}
A provability interpretation for the propositional language is a translation from the propositional language to the language of modal logics.
\end{dfn}
To illuminate the Definition \ref{t9-1}, let us introduce three provability interpretations as examples.
\begin{dfn}\label{t9-2}
The BHK interpretation $b$ is the following translation:
\begin{description}
\item[$(i)$]
$ p^{b}=\Box p$
and
$ \bot^{b}=\Box \bot$
\item[$(ii)$]
$(A\wedge B)^b= A^{b} \wedge B^{b}$
\item[$(iii)$]
$(A\vee B)^{b}=A^{b}\vee B^{b}$
\item[$(iv)$]
$(A\to B)^{b}=\Box(A^{b} \to B^{b})$
\item[$(v)$]
$(\neg A)^{b}=\Box (A^{b} \rightarrow \Box \bot)$
\end{description}
\end{dfn}
Our translation is the same as the usual one, except for the case of $\bot$, which is translated to $\bot$ in the usual translation. (The negation of a formula $A$ is considered as $A \rightarrow \bot$ and it inherits this change in the translation from $\bot$. ) The reason for slightly changing the definition of the translation is because the usual translation can not capture the intended intuition of the BHK interpretation. Actually, the intended intuitionistic meaning of $\bot$, similar to the other atomic formulas, is its provability. Therefore, the natural interpretation of $\bot
$ is $\Box \bot$. On the other hand, we know that the BHK interpretation claims that there is not any proof of $\bot$, which means $\neg \Box \bot$. Based on these two observations, we can justify the usual translation of $\bot$ as $\Box \bot \wedge \neg \Box \bot$, which is the same as $\bot$. Nevertheless, we have to emphasize that the condition of the unprovability of the inconsistency is not related to the meaning of the connectives, and hence it should not interfere in the BHK interpretation; it is actually a commitment we impose on the discourse of the  provability. In our terms, the unprovability of the inconsistency asserts that the theories and meta-theories are consistent and it is obviously a property of the provability model and not a property of the connectives which we want to define. Hence, to formalize the original BHK interpretation, we need two ingredients; one is the $b$ translation which is the formalization of the implicit BHK interpretation, and the second is the consistency condition on the provability models. The following definition formally states the second condition.
\begin{dfn}\label{t9-3}
A provability model $(M, \{T_n\}_{n=0}^{\infty})$ is called a BHK model if for any $n$, $M \vDash \neg \Pr_{n+1}(\Pr_n(\bot))$.
\end{dfn}
\begin{rem}\label{t9-4}
It seems that the natural consistency condition would be the consistency of all the theories. Yet, it is not enough. For instance, it is possible that all the theories in the hierarchy are consistent, but some meta-theory thinks that the lower theory is inconsistent, which contradicts with what an intuitionist assumes. For the intuitionist, the hierarchy of theories are just different layers of the story of the mind, and obviously these stories must be consistent in accordance with the BHK interpretation. However, this condition should be mentioned in the story itself. One way is assuming that any meta-theory actually proves the consistency of the lower theories. This is a natural condition, but it imposes a strong commitment on our theories. To keep the commitments as minimal as possible, we believe that the right condition to impose on the theories is the weaker condition which states that any meta-theory does not think that the lower theory is inconsistent. As we will see, this weaker condition widens the horizon of the BHK interpretation to capture the basic propositional logic on the one hand, and avoid artificial and degenerate models in which we could capture classical logic, on the other. 
\end{rem}
Based on the aforementioned considerations, when we talk about the formalization of the BHK interpretation, we always refer to the BHK models. Let us formalize what we will call the weak BHK interpretation.
\begin{dfn}\label{t9-5}
Let $q$ be a new atom which does not belong to the propositional language. The weak BHK interpretation, $w$, is the following translation:
\begin{description}
\item[$(i)$]
$ p^w=\Box p$
and
$ \bot^w=\Box q$
\item[$(ii)$]
$(A\wedge B)^w= A^w \wedge B^w$
\item[$(iii)$]
$(A\vee B)^w=A^w \vee B^w$
\item[$(iv)$]
$(A\to B)^w=\Box(A^w \to B^w)$
\item[$(v)$]
$(\neg A)^w=\Box (A^w \rightarrow \Box q)$
\end{description}
\end{dfn}
The translation is based on the idea that in this variant of the BHK interpretation, we eliminate the consistency condition from the discourse of the provability. As a result, with this interpretation the intuitionist can not distinguish the inconsistency statement from any other statements. Therefore, in her viewpoint, $\bot$ is just a new atomic sentence which could be provable.\\ 
 
And finally, we will define G\"{o}del's translation to show that there could be different provability models apart from the BHK interpretations.
\begin{dfn}\label{t9-6}
G\"{o}del's provability interpretation, $g$, is the following translation:
\begin{description}
\item[$(i)$]
$ p^g=\Box p$
and
$ \bot^g= \bot$
\item[$(ii)$]
$(A\wedge B)^g= A^g \wedge B^g$
\item[$(iii)$]
$(A\vee B)^g=A^g \vee B^g$
\item[$(iv)$]
$(A\to B)^g=\Box(A^g \to B^g)$
\item[$(v)$]
$(\neg A)^g=\Box (\neg A^g)$
\end{description}
\end{dfn}
It is time to define the satisfaction of a propositional formula in a provability model with respect to some provability interpretation $i$.
\begin{dfn}\label{t9-7}
Let $i$ be a provability interpretation. Then, by an expansion of a propositional formula $A$, and a witness for $A$ under the interpretation $i$, we mean an expansion and a witness for $A^i$. And by $(M, \{T\}_{n=0}^{\infty}, i) \vDash \Gamma \Rightarrow A$ we mean $(M, \{T\}_{n=0}^{\infty}) \vDash \Gamma^i \Rightarrow A^i$. Moreover, if $C$ is a class of provability models, by $(C, i)$ we mean $\{(M, \{T\}_{n=0}^{\infty}, i) \mid (M, \{T\}_{n=0}^{\infty}) \in C\}$ and by $(C, i) \vDash \Gamma \Rightarrow A$ we mean $C \vDash \Gamma^i \Rightarrow A^i$. 
\end{dfn}
The next step is establishing the soundness-completeness theorem for the provability interpretations we defined. But first, we need a technical lemma.
\begin{lem}\label{t9-8}
If $\Gamma^b \vdash_{\mathbf{KD4}} A^b$, then $\mathbf{EBPC} \vdash \Gamma \Rightarrow A$. 
\end{lem}
\begin{proof}
If $\Gamma^b \vdash_{\mathbf{KD4}} A^b$ then there is a cut-free proof for $\Gamma^b \Rightarrow A^b$ in $G(\mathbf{KD4})$. Call it $\pi$. It is clear that all formulas occurring in $\pi$ are sub-formulas of $A^b$ or sub-formulas of formulas in $\Gamma^b$. We know that all of these sub-formulas have the following forms: $B^b$; $B^b \rightarrow C^b$ and atoms $p$. ($\top$ and $\bot$ are considered atomic formulas in this proof.) Therefore, every sequent in $\pi$ has the following form:
\[
\Gamma^b, \{B^b_i \rightarrow C^b_i\}_{i \in I}, \{p_j\}_{j \in J} \Rightarrow \Delta^b, \{D^b_r \rightarrow E^b_r\}_{r \in R}, \{q_s\}_{s \in S} 
\]
Now we will prove the following claim:\\
 
\textbf{Claim.} If 
\[
G(\mathbf{KD4}) \vdash \; \Gamma^b, \{B^b_i \rightarrow C^b_i\}_{i \in I}, \{p_j\}_{j \in J} \Rightarrow \Delta^b, \{D^b_r \rightarrow E^b_r\}_{r \in R}, \{q_s\}_{s \in S} 
\]
where $\{p_j\}_{j \in J} \cap \{q_s\}_{s \in S}=\emptyset$ and $\bot \notin \{p_j\}_{j \in J}$ then for any $X \subseteq I$
\[
\Gamma, \{D_r\}_{r \in R}, \{C_i\}_{i \in X} \vdash_{\mathbf{EBPC}} \bigvee \{\Delta, \{E_r\}_{r \in R}, \{B_i\}_{i \notin X}\}
\]

The proof is by induction on the length of the cut-free proof in $G(\mathbf{KD4})$. To simplify the proof, we will call a sequent with the conditions $\{p_j\}_{j \in J} \cap \{q_s\}_{s \in S}=\emptyset$ and $\bot \notin \{p_j\}_{j \in J}$, a good sequent.\\

The case for axioms and structural rules are easy to check. If the last rule is a conjunction or disjunction rule, then the main formula has the first form. Then since it is possible to simulate all conjunction and disjunction rules in $\mathbf{EBPC}$, the case of conjunction and disjunction rules are also easy to check. If the last rule is an implication rule, since we define our claim up to using implicational rules, there is nothing to prove in this case. Moreover, notice that if the consequent sequent is good then the premises are so. Therefore, it is possible to use the induction hypothesis for them. Finally, if the last rule is a modal rule, then, we have the following two cases:\\

1. If the last rule is a modal rule $\Box_4 R$, based on the form of formulas and the fact that in those three forms a boxed formula should be of the first kind, we have two cases. The first case is when the boxed formula in the right side has the form $\Box(D^b \rightarrow E^b)$. The second case is when the formula has the form $\Box p$. For the first case, the last rule is like the following:
\begin{center}
  	\begin{tabular}{c}
		\AxiomC{$ \{p_j, \Box p_j \}_{j \in J}, \{ B_i^b \rightarrow C_i^b, \Box(B_i^b \rightarrow C_i^b) \}_{i \in I} \Rightarrow D^b \rightarrow E^b$}
		\UnaryInfC{$ \{\Box p_j \}_{j \in J}, \{ \Box(B_i^b \rightarrow C_i^b) \}_{i \in I} \Rightarrow \Box (D^b \rightarrow E^b)$}
		\DisplayProof
		\end{tabular}
\end{center}
and we want to prove
\[
\{p_j \}_{j \in J}, \{ B_i \rightarrow C_i \}_{i \in I} \vdash_{\mathbf{EBPC}} D \rightarrow E
\]
Since every formula in the consequent sequent are boxed, it is a good sequent. Moreover, the only way for the premise sequent to not be good is that for some $j$, $p_j=\bot$. Therefore the claim is obvious from the $\bot$ rule in $\mathbf{EBPC}$. Hence, we can also assume that the premise sequent is a good one. Then, by IH we know that for any $X \subseteq I$ we have
\[
\{p_j \}_{j \in J}, \{ B_i \rightarrow C_i \}_{i \in I}, \{C_i\}_{i \in X}, D \vdash_{\mathbf{EBPC}} \{B_i\}_{i \notin X}, E
\]
By the rule $\rightarrow I$ the following is provable by $\Sigma=\{p_j \}_{j \in J} \cup \{ B_i \rightarrow_{n_i} C_i \}_{i \in I}$
\[
\bigwedge \{C_i\}_{i \in X} \wedge D \rightarrow \bigvee \{B_i\}_{i \notin X} \vee E
\]
Fix $i \in I$ and also fix some $Z \subseteq I-\{i\}$. Both of the following statements are theorems of $\Sigma$:
\[
\bigwedge \{C_i\}_{i \in Z} \wedge D \rightarrow \bigvee \{B_i\}_{i \notin Z} \vee \mathbf{B_i} \vee E
\]
and
\[
\bigwedge \{C_i\}_{i \in Z} \wedge \mathbf{C_i} \wedge D \rightarrow \bigvee \{B_i\}_{i \notin Z} \vee E
\]
Since $\Sigma \vdash \mathbf{B_i} \rightarrow \mathbf{C_i}$. Then by using appropriate formalized rules we will have
\[
\bigwedge \{C_i\}_{i \in Z} \wedge D \rightarrow  \bigvee \{B_i\}_{i \notin Z}\vee E
\]
provable by $\Sigma$ in $\mathbf{EBPC}$.
By iterating this method we can eliminate all elements in $I$. Therefore we will have
\[
\Sigma \vdash_{\mathbf{EBPC}} D \rightarrow E
\]
which is what we wanted to prove.\\

If the boxed formula in the right side of the rule is $\Box p$, then the last rule has the form 
\begin{center}
  	\begin{tabular}{c}
		\AxiomC{$ \{p_j, \Box p_j \}_{j \in J}, \{ B_i^b \rightarrow C_i^b, \Box(B_i^b \rightarrow C_i^b) \}_{i \in I} \Rightarrow p$}
		\UnaryInfC{$ \{\Box p_j \}_{j \in J}, \{ \Box(B_i^b \rightarrow C_i^b) \}_{i \in I} \Rightarrow \Box p$}
		\DisplayProof
		\end{tabular}
\end{center}
and we want to prove
\[
\{p_j \}_{j \in J}, \{ B_i \rightarrow C_i \}_{i \in I} \vdash_{\mathbf{EBPC}} p
\]
There are two different cases. The first case is when $p \in \{p_j \}_{j \in J}$ or $\bot \in \{p_j \}_{j \in J}$. In this case the claim is an obvious consequence of an axiom in $\mathbf{EBPC}$. The second case is when $p \notin \{p_j \}_{j \in J}$ and $\bot \notin \{p_j \}_{j \in J}$. Therefore, the premise sequent is a good one. Hence by IH and for any $X \subseteq I$ we have
\[
\{p_j \}_{j \in J}, \{ B_i \rightarrow C_i \}_{i \in I}, \{C_i\}_{i \in X} \vdash_{\mathbf{EBPC}}  \{B_i\}_{i \notin X}
\]
with the same method as above we can deduce
\[
\{p_j \}_{j \in J}, \{ B_i \rightarrow C_i \}_{i \in I} \vdash_{\mathbf{EBPC}} \top \rightarrow  \bot
\] 
Then by the rule $C$, we will have
\[
\{p_j \}_{j \in J}, \{ B_i \rightarrow C_i \}_{i \in I} \vdash_{\mathbf{EBPC}} \bot
\]
which is what we wanted.\\

2. If the last rule is $\Box_{D}R$, then everything in the proof is the same as the proof for the case 1 when we put $D=\top$ and $E=\bot$. Therefore, we will have
\[
\{p_j \}_{j \in J}, \{ B_i \rightarrow C_i \}_{i \in I} \vdash_{\mathbf{EBPC}} \top \rightarrow  \bot
\] 
Then by the rule $C$, we will have
\[
\{p_j \}_{j \in J}, \{ B_i \rightarrow C_i \}_{i \in I} \vdash_{\mathbf{EBPC}} \bot
\]
which is what we wanted.\\
After proving the claim, the theorem is an easy consequences of the claim. Since there is a proof of $\Gamma^b \Rightarrow A^b$ in $G(\mathbf{KD4})$ then the sequent is obviously a good one and hence by the claim we will have $\Gamma \vdash_{\mathbf{EBPC}} A$. \\
\end{proof}
\begin{thm}\label{t9-9}
\begin{itemize}
\item[$(i)$]
$\Gamma \vdash_{\mathbf{BPC}} A$ iff $\Gamma^b \vdash_{\mathbf{K4}} A^b$
\item[$(ii)$]
$\Gamma \vdash_{\mathbf{EBPC}} A$ iff $\Gamma^b \vdash_{\mathbf{KD4}} A^b$
\item[$(iii)$]
$\Gamma \vdash_{\mathbf{IPC}} A$ iff $\Gamma^b \vdash_{\mathbf{S4}} A^b$
\item[$(iv)$]
$\Gamma \vdash_{\mathbf{FPL}} A$ iff $\Gamma^b \vdash_{\mathbf{GL}} A^b$
\item[$(v)$]
$\Gamma \vdash_{\mathbf{MPC}} A$ iff $\Gamma^w \vdash_{\mathbf{S4}} A^w$
\end{itemize}
\end{thm}
\begin{proof}
The proof of the soundness part is easy and routine. For the completeness part, the case $(iv)$ is proved by Visser in \cite{Vi}. The same proof also works for $(i)$. $(iii)$ is a well-known result. (See \cite{Tr} for instance.) $(ii)$ is proved by Lemma \ref{t9-8}. 
For the case $(v)$, we know that $\mathbf{MPC}$ and $\mathbf{S4}$ are sound and strongly complete with respect to the class of reflexive transitive Kripke models. (For $\mathbf{MPC}$ the model should also be persistent.) However, in the case of $\mathbf{MPC}$, the nodes can also satisfy $\bot$. Soundness is again easy. For the completeness part, if we have a counter $\mathbf{MPC}$-Kripke model for $\Gamma \Rightarrow A$, we can construct a counter $\mathbf{S4}$-model for $\Gamma^w \Rightarrow A^w$ in the following way: Use the same Kripke model, with the same values, but assume that $q$ is true in a node, if $\bot$ is true in that node. Then, it is easy to show that for any propositional formula $B$, $B$ is true in the node $l$ iff $B^w$ is so. Therefore, if the first model is a counter example for $\Gamma \Rightarrow A$, then the new one is a counter example for $\Gamma^w \Rightarrow A^w$. This construction proves the completeness part.
\end{proof}
We can use the soundness and completeness of these translations to transfer our results from the modal setting to the propositional one.
\begin{dfn}\label{t9-10}
The class $\mathbf{BHK}$ is the class of all BHK models and the class $\mathbf{cBHK}$ is the class of all BHK models which are constant. 
\end{dfn}
\begin{thm}\label{t9-11}
\begin{itemize}
\item[$(i)$]
$\Gamma \vdash_{\mathbf{BPC}} A$ iff $(\mathbf{PrM}, b) \vDash \Gamma \Rightarrow A $. And $\mathbf{BPC} \vdash A$ iff $(\mathbf{BHK}, b) \vDash A$.
\item[$(ii)$]
$\Gamma \vdash_{\mathbf{EBPC}} A$ iff $(\mathbf{Cons}, b) \vDash \Gamma \Rightarrow A $.
\item[$(iii)$]$\Gamma \vdash_{\mathbf{IPC}} A$ iff $(\mathbf{Ref}, b) \vDash \Gamma \Rightarrow A $.
\item[$(iv)$]
$\Gamma \vdash_{\mathbf{FPL}} A$ iff $(\mathbf{Cst}, b) \vDash \Gamma \Rightarrow A $. And $\mathbf{FPL} \vdash A$ iff $(\mathbf{cBHK}, b) \vDash A$.
\item[$(v)$]
Let $(M, \{T_n\}_{n=0}^{\infty})$ be a provability model. Then $(M, \{T_n\}_{n=0}^{\infty}, b) \vDash \mathbf{CPC}$ iff there exists $n$ such that $M \vDash \Pr_{n+1}(\Pr_n(\bot))$. Therefore, there is not any BHK interpretation for classical logic.
\end{itemize}
\end{thm}
\begin{proof}
Based on Theorem \ref{t9-9}, the strong soundness-completeness parts are just easy consequences of the soundness-completeness results for the corresponding modal logics. For the BHK completeness part for $(i)$, if $(\mathbf{BHK}, b) \vDash A$, then there are expansions $B_i$'s for $A^w$ and a witness for $\bigvee B_i$, such that for all arithmetical substitutions $\sigma$, and all BHK models $(M, \{T_n\}_{n=0}^{\infty})$, we have $M \vDash (\bigvee_{i=0}^{r} B_i)^{\sigma}(w)$. Let $\Gamma$ be a sequence of infinite copies of $\neg \Box \Box \bot$ and $u$ a witness, which witnesses each of these formulas by $(n+1, n)$. We claim that for any provability model $(M, \{T_n\}_{n=0}^{\infty})$ and any arithmetical substitution $\sigma$, we have $M \vDash \Gamma^{\sigma}(u) \Rightarrow (\bigvee_{i=0}^{r} B_i)^{\sigma}(w)$. If $M \vDash \Gamma^{\sigma}(u)$, then for any $n$, we have $M \vDash \neg \Pr_{n+1}(\Pr_n(\bot))$. Hence, $(M, \{T_n\}_{n=0}^{\infty})$ is a BHK model and therefore, $M \vDash (\bigvee_{i=0}^{r} B_i)^{\sigma}(w)$. We know $\mathbf{PrM} \vDash \Gamma \Rightarrow A^b$; therefore, by strong completeness for $\mathbf{K4}$, we have $\Gamma \vdash_{\mathbf{K4}} A^b$. Thus, $\mathbf{K4} \vdash \neg \Box \Box \bot \rightarrow A^b$ and then, $\mathbf{K4} \vdash ((\top \rightarrow \bot) \vee A)^b$. By Theorem \ref{t9-9}, $\mathbf{BPC} \vdash (\top \rightarrow \bot) \vee A$, and therefore by the disjunction property of $\mathbf{BPC}$, we know that $\mathbf{BPC} \vdash A$ or $\mathbf{BPC} \vdash \top \rightarrow \bot$. The latter is impossible by simple facts about $\mathbf{BPC}$, therefore $\mathbf{BPC} \vdash A$.\\
The case $(iv)$ also needs an argument exactly similar to the case $(i)$. Moreover, since the consistent and reflexive models admit the consistency condition of the BHK interpretation, the cases $(ii)$ and $(iii)$ are just a combination of Theorem \ref{t9-9} and the completeness results for the corresponding theories.\\ 

For $(v)$ we need some justification. First of all we want to show that if for any $n$, $M \vDash \neg \Pr_{n+1}(\Pr_n(\bot))$, then $(M, \{T_n\}_{n=0}^{\infty})$ is not a model for $\mathbf{CPC}$. We prove this claim by contradiction. Assume that for any $n$, $M \vDash \neg \Pr_{n+1}(\Pr_n(\bot))$ and $(M, \{T_n\}_{n=0}^{\infty}, b) \vDash \mathbf{CPC}$. We want to show that all three statements of the proof of Theorem \ref{t7-1} is also true in our case. Firstly, $(i)$ is true by assumption. Secondly, consider the formula $\Box^n \top$ which is a translation of the propositional classical theorem $\top^n$ with the definition $\top^0=\top$ and $\top^{n+1}=\top \rightarrow \top^n$. Therefore, the formula $\Box (\Box \Box(\Box \bot \wedge \Box^n \top) \rightarrow \Box \bot)$ is the translation of the tautology $((\top \rightarrow (\top \rightarrow (\bot \wedge \top^n))) \rightarrow \bot)$. Thus,
\[
(M, \{T_n\}_{n=0}^{\infty}) \vDash \Box (\Box \Box(\Box \bot \wedge \Box^n \top) \rightarrow \Box \bot).
\]
Since we used this formula to show $(ii)$, we can claim that we also have $(ii)$ here. Thirdly, we know that $p \vee \neg p$ is a theorem of $\mathbf{CPC}$. Hence, $(M, \{T_n\}_{n=0}^{\infty}) \vDash (p \vee \neg p)^b$, which means $(M, \{T_n\}_{n=0}^{\infty}) \vDash (\Box p \vee \Box (\Box p \rightarrow \Box \bot)$. Therefore, $(iii)$ is also true in $M$. Thus, we have a contradiction and it proves the claim. \\

For the converse, assume that there is some $n$ such that $M \vDash \Pr_{n+1}(\Pr_n(\bot))$; we will show that $(M, \{T_n\}_{n=0}^{\infty}, w) \vDash \mathbf{CPC}$. First of all, to simplify the proof, define the complexity of any box as the maximum depth of the nested boxes in front of that box. For instance, the complexity of the inner box in $\Box(\Box p \wedge q)$ is zero, and the complexity of the outer box is one. Define the canonical witness starting from $n$, as follows: Witness any box by its complexity plus $n$. It is easy to show that this witness is an ordered one, because the witness for any outer box is bigger than the witness for the inner boxes. Define $A^{\sigma}$ as the formula resulted by substituting all the atoms by $\sigma$ and witnessing all the boxes by the canonical witness starting from $n$. It is easy to verify that for any propositional formula $A \rightarrow B$, $M \vDash ((A \rightarrow B)^w)^{\sigma}$. To show this, firstly, note that the following claim holds: For any propositional formula $B$,
\[
I \Sigma_1 \vdash \Pr_n(\bot) \rightarrow (B^w)^{\sigma}.
\] 
The proof of the claim is based on induction on $B$ and easily follows. 
Assume that the complexity of the outmost box in $\Box(A^w \rightarrow B^w)$ is $k\geq n+1$. (Since witnesses begin with $n$ and there is at least one box in $A^w$, $k$ is at least $n+1$.) By $\Sigma_1$-completeness we have
\[
I\Sigma_1 \vdash \Pr_k(\Pr_n(\bot) \rightarrow (B^w)^{\sigma}),
\]
and hence,
\[
I\Sigma_1 \vdash \Pr_k(\Pr_n(\bot)) \rightarrow \Pr_k((B^w)^{\sigma}).
\]
Then since $M \vDash I\Sigma_1$, then
\[
M \vDash \Pr_k(\Pr_n(\bot)) \rightarrow \Pr_k((B^w)^{\sigma}).
\]
We know that 
$M \vDash \Pr_{n+1}(\Pr_n(\bot)) $ and $k \geq n+1$; hence $M \vDash \Pr_{k}(\Pr_n(\bot)) $. Therefore,
\[
M \vDash \Pr_k((B^w)^{\sigma}),
\]
and thus,
\[
M \vDash \Pr_k((A^w)^{\sigma} \rightarrow (B^w)^{\sigma}),
\]
and the proof follows.\\

It is easy to check that for any formula $B$, there exists another formula $C$ such that $C$ is in the CNF form, in which all the literals are implicational formulas, positive atoms and $\bot$ and classically equivalent to $B$. Note that the process of constructing this $C$ just uses the classical rules for conjunction and disjunction. Since $w$ and the canonical witness respect the conjunction and disjunction and their basic rules, $(B^w)^{\sigma}$ and $(C^w)^{\sigma}$ are equivalent in $M$. Suppose that $\mathbf{CPC} \vdash B$; we want to show that $M \vDash (B^w)^{\sigma}$. It is enough to show that $M \vDash (C^w)^{\sigma}$. Considering that all the literals in $C$ are implicational formulas, positive atoms and $\bot$, the literals of $C^b$ are translations of implications, boxed atoms or $\Box \bot$. If $M \nvDash (C^w)^{\sigma}$, there must be some clause in which all the literals are false. Since the translations of the implications are true in $M$, there has to be a clause in $C$ consisting of atoms and $\bot$. Therefore, $C$ can not be a classical tautology and hence $B$ will not be, as well. But $\mathbf{CPC} \vdash B$; a contradiction. Thus, $M \vDash (B^w)^{\sigma}$.\\
So far, we have shown that if $\mathbf{CPC} \vdash B$, then $M \vDash (B^w)^{\sigma}$. If we send $q$ in the definition of $\bot^w=\Box q$, to $\bot$, then we have $M \vDash (B^b)^{\sigma}$, which proves the theorem.
\end{proof}
There is another type of the BHK interpretation in which there is not any kind of assumption on the non-existence of a proof of the contradiction. 
\begin{thm}\label{t9-12}
\begin{itemize}
\item[$(i)$]
$\Gamma \vdash_{\mathbf{MPC}} A$ iff $(\mathbf{Ref}, w) \vDash \Gamma \Rightarrow A $.
\item[$(ii)$]
Let $(M, \{T_n\}_{n=0}^{\infty})$ be a provability model. Then $(M, \{T_n\}_{n=0}^{\infty}, w) \vDash \mathbf{IPC}$ iff $(M, \{T_n\}_{n=0}^{\infty}, w) \vDash \mathbf{CPC}$ iff there exists $n$ such that $M \vDash \Pr_{n+1}(\Pr_n(\bot))$.
\end{itemize}
\end{thm}
\begin{proof}
For $(i)$, use Theorem \ref{t9-9} and the soundness-completeness results for $\mathbf{S4}$. For $(ii)$, if there exists $n$ such that $M \vDash \Pr_{n+1}(\Pr_n(\bot))$, then by the proof of Theorem \ref{t9-11} part $(v)$, we know that $(M, \{T_n\}_{n=0}^{\infty}, w) \vDash \mathbf{CPC}$. Moreover, if $(M, \{T_n\}_{n=0}^{\infty}, w) \vDash \mathbf{CPC}$, then we can easily verify that we have $(M, \{T_n\}_{n=0}^{\infty}, w) \vDash \mathbf{IPC}$. It remains to show that if $(M, \{T_n\}_{n=0}^{\infty}, w) \vDash \mathbf{IPC}$, then there exists $n$ such that $M \vDash \Pr_{n+1}(\Pr_n(\bot))$.\\
Assume that $(M, \{T_n\}_{n=0}^{\infty}, w) \vDash \mathbf{IPC}$ and for any $n$, $M \vDash \neg \Pr_{n+1}(\Pr_n(\bot))$. We want to reach a contradiction. We know that $\mathbf{IPC} \vdash \bot \rightarrow p$. Hence, $(M, \{T_n\}_{n=0}^{\infty}) \vDash (\bot \rightarrow p)^w$. Thus, $(M, \{T_n\}_{n=0}^{\infty}) \vDash \Box (\Box q \rightarrow \Box p)$. Consequently, there are expansions of the form, $\Box(\bigvee_{j=0}^{s_i} (\Box q \rightarrow \Box p))$ for $0 \leq i \leq r$ and witnesses $w_i=(n_i, (m_{ij}, k_{ij})_{j=0}^{s_i})$ such that for any arithmetical substitution $\sigma$,
\[
M \vDash \bigvee_{i=0}^{r} \Box(\bigvee_{j=0}^{s_i} (\Box q \rightarrow \Box p))^{\sigma}(w_i).
\]
Define $k=max_{ij}(k_{ij})$, $m=min_{ij}(m_{ij})$ and $n=max_{i}(n_i)$. It is easy to see that
\[
M \vDash \Pr_n((\Pr_m (q^{\sigma}) \rightarrow \Pr_k (p^{\sigma}))).
\]
And if we choose a substitution $\sigma$ such that $q^{\sigma}= (0=0)$ and $p^{\sigma}=(0=1)$, then we have
\[
M \vDash \Pr_n((\Pr_m (0=0) \rightarrow \Pr_k (0=1)),
\]
and hence $M \vDash \Pr_n(\Pr_k(\bot))$. Thus, for some number $N>n,k$, we have $M \vDash \Pr_{N+1}(\Pr_N(\bot))$ which is a contradiction.
\end{proof}
\vspace{4pt}
\textbf{Acknowledgment.} We are indebted to Pavel Pudl\'{a}k for the helpful discussions, his careful reading of the earlier draft, and his invaluable comments. We wish to thank Mohammad Ardeshir for his helpful suggestions and specially introducing G\"{o}del's problem and its key role to us. We are also grateful to Emil Je\v{r}\'{a}bek and Lev Beklemishev for pointing out some errors in the earlier proofs and arguments. And we are thankful to Raheleh Jalali for her careful technical and language editing.

\end{document}